\documentclass[reqno]{amsart}
\usepackage{graphicx}
\usepackage{xcolor,bbm,easybmat,bm}
\usepackage{amsmath,amsfonts,amsthm,amscd,amssymb,mathrsfs,esint}
\usepackage{enumerate}
\usepackage[shortlabels]{enumitem}
\usepackage[colorlinks,citecolor=red,pagebackref,hypertexnames=false]{hyperref}

\usepackage{accents}

\DeclareMathOperator{\supp}{supp}
\DeclareMathOperator{\loc}{loc\ }
\DeclareMathOperator{\divg}{div}
\DeclareMathOperator{\dist}{dist}
\DeclareMathOperator{\diam}{diam}

\DeclareMathOperator*{\fiint}{\ensuremath{\iint\text{\kern-1.36em{\raisebox{5.87pt}{\rotatebox{-93}{$\setminus$}}}}}\,}

\theoremstyle{plain}

\newtheorem{lemma}[equation]{Lemma}

\newtheorem{lem}[equation]{Lemma}

\newtheorem{thm}[equation]{Theorem}
\newtheorem{cor}[equation]{Corollary}

\theoremstyle{definition}
\newtheorem{defi}[equation]{Definition}
\newtheorem{defn}[equation]{Definition}

\theoremstyle{remark}

\newtheorem{rem}[equation]{Remark}

\numberwithin{equation}{section}

\newcommand{\om}{\Omega}
\newcommand{\pom}{\partial\Omega}

\newcommand{\norm}[1]{\left\Vert#1\right\Vert}
\newcommand{\abs}[1]{\left\vert#1\right\vert}
\newcommand{\br}[1]{\left(#1\right)}
\newcommand{\set}[1]{\left\{#1\right\}}
\newcommand{\Real}{\mathbb R}

\newcommand{\eps}{\varepsilon}

\newcommand{\C}{\mathcal{C}}

\renewcommand{\H}{\mathcal{H}}
\newcommand{\vp}{\varphi}

\newcommand{\wt}{\widetilde}
\newcommand{\bdy}{\partial}

\newcommand{\stcomp}[1]{{#1}^{\mathsf{c}}}

\newcommand{\mZ}{\mathcal{Z}}
\newcommand{\1}{\mathbbm{1}}

\newcommand{\Tlz}{T_\lambda^z}

\newcommand{\NN}{{\mathbb{N}}}
\newcommand{\ZZ}{{\mathbb{Z}}}

\newcommand{\rd}{\mathbb{R}^d}

\newcommand{\reu}{\mathbb{R}^{d+1}_+}
\newcommand{\ree}{\mathbb{R}^{d+1}}

\newcommand{\R}{\mathbb{R}} 
\newcommand{\cN}{\mathfrak{N}} 
 
\newcommand{\cA}{\mathfrak{A}}
\newcommand{\sm}{\setminus} 
\renewcommand{\epsilon}{\varepsilon} 
\newcommand{\wh}{\widehat}

\renewcommand{\d}{\partial} 
\newcommand{\ms}{\medskip}

\renewcommand{\emptyset}{\o}

\def\div{\mathop{\operatorname{div}}\nolimits}

\newcommand\SM[1]{{\color{red}#1}}

\begin{document}

\title[Small $A_\infty$ results]
{Small $A_\infty$ results for Dahlberg-Kenig-Pipher operators in sets with uniformly rectifiable boundaries}
\thanks{
G. David was partially supported by the European Community H2020 grant GHAIA 777822,
and the Simons Foundation grant 601941, GD.
S. Mayboroda was partly supported by the NSF RAISE-TAQS grant DMS-1839077 and the Simons foundation grant 563916, SM}

\author{Guy David}
\author{Linhan Li}
\author{Svitlana Mayboroda}

\newcommand{\Addresses}{{
  \bigskip
  \footnotesize

 Guy David, \textsc{Universit\'e Paris-Saclay, CNRS, 
 Laboratoire de math\'ematiques d'Orsay, 91405 Orsay, France} 
 \par\nopagebreak
  \textit{E-mail address}: \texttt{guy.david@universite-paris-saclay.fr}

  \medskip

  Linhan Li, \textsc{School of Mathematics, University of Minnesota, Minneapolis, MN 55455, USA}\par\nopagebreak
  \textit{E-mail address}: \texttt{linhan\_li@alumni.brown.edu}

  \medskip

  Svitlana Mayboroda, \textsc{School of Mathematics, University of Minnesota, Minneapolis, MN 55455, USA}\par\nopagebreak
  \textit{E-mail address}: \texttt{svitlana@math.umn.edu}

}}

\begin{abstract}

In the present paper we consider elliptic operators $L=-\divg(A\nabla)$ in a domain
bounded by a chord-arc surface $\Gamma$ with small enough constant, and whose 
coefficients $A$ satisfy a weak form of the Dahlberg-Kenig-Pipher condition of
approximation by constant coefficient matrices, with a small enough Carleson norm, 
and show that the elliptic measure with pole at infinity associated to $L$ is
$A_\infty$-absolutely continuous with respect to the surface measure on $\Gamma$,
with a small $A_\infty$ constant. In other words, we show that for relatively flat uniformly rectifiable sets and for operators with slowly oscillating coefficients the elliptic measure satisfies the $A_\infty$ condition with a small constant and the logarithm of the Poisson kernel has small oscillations.
\end{abstract}

\maketitle


\tableofcontents

\section{Introduction and main results}

Recent developments have showed that for the class of Dahlberg-Kenig-Pipher operators the elliptic measure is absolutely continuous with respect to the Hausdorff measure on all uniformly rectifiable sets with some mild topological conditions (we will provide the relevant list of references below). The conditions on the operator and on the geometry are both essentially necessary and in fact, one can establish the equivalence of the aforementioned geometric features of the domain to the property of absolute continuity of harmonic measure (\cite{azzam2020harmonic}). 
The present paper focuses on the ``small constant" version of these results, showing, roughly speaking, that for relatively flat domains and for slowly oscillating coefficients, the elliptic measure is $A_\infty$ with respect to the Hausdorff measure with a small norm and that the oscillations of the logarithm of the Poissin kernel are small. This turned out to be surprisingly non-trivial primarily because the standard methods of work with Carleson measures, uniform rectifiability and related notions ``bootstrap" and hence, enlarge the constants.

To be precise, in this paper we prove small $A_\infty$ estimates for the elliptic measure, with pole at infinity,
of a divergence form elliptic operators $L=-\divg(A\nabla)$, with a matrix of coefficients $A$ 
that satisfies a weak Dahlberg-Kenig-Pipher condition (see the definitions below) 
with a small corresponding Carleson norm. 
We will do this both in the upper half-space $\R^{d+1}_+$  and in any of the two 
connected components $\Omega$ bounded by a Chord-arc surface with small constant 
$\Gamma \subset \R^{d+1}$.

Concerning the case of $\R^{d+1}_+$, this paper can be viewed as a continuation of a paper of Bortz, Toro, and Zhao
\cite{bortz2021elliptic}, where a similar result was proved in the vanishing case. 
That is, they assume in addition that the Carleson measure associated to $A$ has a 
vanishing trace, as in Definition \ref{d13bis}, and get that the elliptic measure with pole at 
infinity is $A_\infty$ with respect to the Lebesgue measure on $\R^d$, with a vanishing constant. We shall use their result, plus a compactness argument,
to say that if our result were not true on $\R^{d+1}_+$, we would be able to construct 
a counterexample to their result too. 

Let us say a bit more about the proof of \cite{bortz2021elliptic}. 
It is based on estimates for the Green function for $L$ (with a pole at infinity), 
which come from \cite{david2021carleson},
and, to go from the Green function to the elliptic measure, 
a lemma of Korey \cite{korey1998carleson}
on weights, which itself is a vanishing constant variant of a result of Fefferman, Kenig and Pipher \cite{fefferman1991theory}.
Surprisingly, the small constant variant of these results on weights is not as easy to deduce from 
 \cite{fefferman1991theory} 
and \cite{korey1998carleson} as what we do here at the level of operators (hence the proof here),
but as we were writing this paper, we found out that Bortz, Egert and Saari \cite{bortz2021theorem} proved 
the needed the small constant variant and thus obtained the same small $A_\infty$ result as here 
on $\R^{d+1}_+$. Nonetheless, we believe that the present proof by compactness is interesting in its own right, also
because the argument in \cite{bortz2021theorem} is not so short.

Now let us discuss the case of Chord-arc surfaces with small constant ($CASSC$). These were introduced by Semmes in \cite{semmes1989criterion, semmes1990differentiable, semmes1990analysis},
and they can also be seen as the correct notion of uniformly rectifiable sets with a small constant,
a small generalization of Lipschitz graphs with small constant. The most convenient definition for us will be by ``very big pieces of small Lipschitz graphs'' (see Definition \ref{d1b12}), 
but many other definitions are equivalent. We will say more about $CASSC$ in Section \ref{Scassc},
and in particular we will check that if $\Gamma$ is a $CASSC$, then $\R^{d+1} \sm \Gamma$ 
has exactly two connected components, which are both NTA domains. Our main result is 
that if $\Omega$ is one of these components, and the elliptic operator $L=-\divg(A\nabla)$ satisfies the weak DKP condition on $\Omega$ (see Definition~\ref{d1b10}), 
with a small enough norm, then the corresponding elliptic measure on $\Gamma$, with pole at infinity, is absolutely continuous with respect to the 
Hausdorff measure on $\Gamma$, with a small $A_\infty$ density
(see Definition \ref{d1b14}).

The main case will be when $\Gamma$ is a Lipschitz graph with small constant, and will be deduced from the case of $\R^{d+1}$ with a suitable change of variable from $\Omega$ to 
$\R^{d+1}_+$. Here the point is that the distortion coming from the change of variable is compatible
with the weak DKP condition. After this, we will go from small Lipschitz graphs to $CASSC$ with a 
comparison argument, where we put small Lipschitz domains 
{\it outside of} $\Omega$ 
with a large boundary intersection. 
See Corollary \ref{cor U} for the precise statement.  Typically, one might want to approach $\om$ from {\it inside} by a Lipschitz graph (see Lemma \ref{l5a7} ) as the comparison of elliptic measures in two domains would be easier. However, this approach would not work in our situation as the weak DKP condition does not cooperate well when restricting to subdomains. Therefore, 
we approximate $\Omega$ by small Lipschitz graphs from outside and need, somewhat unconventionally in this context, to extend the coefficients 
to the Lipschitz domains.

Let us now give more precise definitions and state the main results. We will be working on the domain
$\Omega \subset \R^{d+1}$, bounded by $\Gamma$, which will either be a $CASSC$ (defined below)
or $\R^d$. Points of $\Gamma$ will be denoted by lower-case letters (like $x$), and points of 
$\Omega$ by upper-case letters (like $X$). For $x\in\Gamma$ and $r>0$, we shall use the Carleson boxes
\begin{equation} \label{1b1}
T(x,r) = \Omega \cap B_r(x),
\end{equation}
where $B_r(x) = B(x,r)$ denotes the open ball in $\R^{d+1}$, the surface balls $\Delta(x,r)=B(x,r)\cap\Gamma$,
and 
the Whitney boxes 
\begin{equation} \label{1b2}
W_\Omega(x,r) = \big\{ X \in \Omega \cap B_r(x) \, ; \, \dist(X,\Gamma) \geq r/2\big\}
\subset T(x,r).
\end{equation}

We now introduce
Carleson measures on $\Gamma \times (0,+\infty)$.
There is a similar notion of Carleson measures on $\Omega$, which will not be 
needed here, but notice that in the special case of $\reu$, the two notions are the same.

\ms
\begin{defi}[Carleson measures on $\Gamma \times (0,+\infty)$]\label{d13bis}
A \underbar{Carleson measure} on $\Gamma \times (0,+\infty)$
is a nonnegative Borel  measure $\mu$ on $\Gamma \times (0,+\infty)$ whose Carleson norm 
\begin{equation} \label{1b4bis}
\norm{\mu}_{\mathcal{C}}
:=\underset{(x,r) \in \Gamma \times (0,+\infty)}{\sup} r^{-d}\mu(\Delta(x,r) \times (0,r))
\end{equation}
is finite. We use $\mathcal{C}$, or $\mathcal{C}_\Omega$, or
$\mathcal{C}(\Gamma \times (0,+\infty))$ to denote the set of 
Carleson measures on $\Gamma \times (0,+\infty)$. 

We say that $\mu$ is a \underbar{Carleson measure} on $\Gamma \times (0,+\infty)$ 
\underbar{with vanishing trace} when $\mu$ is a Carleson measure on $\Gamma \times (0,+\infty)$ and 
\begin{equation} \label{1b5bis}
\lim_{r_0 \to 0} \  \sup_{0 < r \leq r_0} \sup_{x\in \Gamma} \ r^{-d}\mu(\Delta(x,r) \times (0,r)) = 0.
\end{equation}

For any pair $(x_0,r_0) \in \Gamma \times (0,+\infty)$, we use $\mathcal{C}(x_0,r_0)$ 
or sometimes $\mathcal{C}(\Delta_0)$ when we work with $\Delta_0 = \Delta(x_0,r_0)$,
to denote the set of Borel measures satisfying the Carleson condition restricted to 
$\Delta_0 \times (0,r_0)$, i.e., 
such that
\begin{equation} \label{1b6bis}
\norm{\mu}_{\mathcal{C}(x_0,r_0)} 
:=\underset{(x,r) \in \Gamma \times (0,r_0), B(x,r) \subset B(x_0,r_0)}{\sup}
\, r^{-d} \mu(\Delta(x,r) \times (0,r)) < +\infty.
\end{equation}
\end{defi}

\ms
We shall consider
operators in divergence form $L=-\divg(A\nabla)$,
where $A=\begin{bmatrix} a_{ij}(X)\end{bmatrix}: \Omega\to M_{d+1}(\Real)$ 
is a matrix-valued function, which is elliptic, as follows. 

\ms
\begin{defi}[Elliptic operators]\label{d1b7}
Let $\mu_0 \geq 1$ be given. 
We say that $L=-\divg(A\nabla)$ (and by extension the matrix-valued function $A$)
is \underbar{$\mu_0$-elliptic} on $\Omega$ when  
\begin{equation}
\langle A(X)\xi,\zeta\rangle \le \mu_0\abs{\xi}\abs{\zeta} \mbox{ and } \langle A(X)\xi,\xi\rangle \ge \mu_0^{-1}\abs{\xi}^2
\text{ for  }X \in \Omega \text{ and  }\xi, \eta\in\Real^{d+1}.
\label{cond ellp}
\end{equation}
We say that $L$ (or $A$) is elliptic when it is $\mu_0$-elliptic for some $\mu_0 \geq 1$.
Finally we denote by $\cA_0(\mu_0)$ the collection of all {\it constant} $\mu_0$-elliptic matrices.
\end{defi} 

\ms
We are 
ready to define the weak DKP condition. Let $\mu_0 \geq 1$ be given. 
We use the following quantity to measure the closeness of a $\mu_0$-elliptic matrix 
$A=A(X)$ to constant coefficient matrices. For $x\in\Gamma$ and $r>0$, we define
\begin{equation} \label{1a5}
\alpha_A(x,r) = \inf_{A_0 \in \cA_0(\mu_0)}
\bigg\{\fiint_{X \in W(x,r)} |A(X)-A_0|^2 dX\bigg\}^{1/2}.
\end{equation}
Notice that we only integrate on the Whitney box, but it turns out that since $\Gamma$ is nice,
a Carleson measure condition on the $\alpha_A(x,r)$ also yields a similar control on the larger 
numbers $\gamma_A(x,r)$ where you integrate on the full box $T(x,r)$; see 
Lemma~\ref{lem gammabdalpha} below. Also, we required the constant matrix
$A_0$ to be $\mu_0$-elliptic, but if we did not, we would still get an equivalent number 
$\wt\alpha_A(x,r)$; use Chebyshev or see \cite{david2021carleson}.

\ms
\begin{defi}[Weak DKP condition, vanishing weak DKP condition] \label{d1b10} 
Let $A$ be a $\mu_0$-elliptic matrix-valued function on $\Omega$. 
\begin{itemize}
    \item We say that $A$ satisfies the weak DKP condition with
constant $M > 0$, when the measure  $\mu$ defined by 
$d\mu(x,r) = \alpha_A(x,r)^2 \frac{d\sigma(x) dr}{r}$
is a Carleson measure on $\Gamma \times (0,+\infty)$, with norm
\begin{equation}\label{1a7}
\cN(A) : = \norm{\alpha_A(x,r)^2 \frac{d\sigma(x) dr}{r}}_{\mathcal{C}} \le M,
\end{equation}
and where we called $\sigma = \H^d_{\vert \Gamma}$ the surface measure on $\Gamma$.
\item We say that $A$ satisfies the vanishing weak DKP condition if $\mu$ defined by 
$d\mu(x,r) = \alpha_A(x,r)^2 \frac{d\sigma(x) dr}{r}$
is a Carleson measure on $\Gamma \times (0,+\infty)$ with norm vanishing trace.
\end{itemize}
\end{defi}

\ms
The weak DKP condition is a weaker version of the celebrated Carleson condition popularized
by Dahlberg, Kenig, and Pipher (DKP), which
instead demands that $\wt\alpha_A(x,r)^2\frac{dxdr}{r}$ satisfies a Carleson measure estimate, where
\[
\wt\alpha_A(x,r) = r \sup_{(y,s) \in W(x,r)} |\nabla A(y,s)|.
\]
In 1984, Dahlberg first introduced this condition, and conjectured  that such a Carleson condition guarantees 
the absolute continuity of the elliptic measure
with respect to the Lebesgue measure in the upper half-space. In 2001, Kenig and Pipher \cite{kenig2001dirichlet} proved Dahlberg's conjecture. More precisely, they show that the DKP condition guarantees that the corresponding elliptic measure is $A_\infty$ with respect to the surface measure in Lipschitz domains. Let $\mu$ and $\nu$ be two non-negative Borel measures on $\Gamma$. We say that $\mu\in A_\infty(\nu)$ if there exists $C,\eta,\theta>0$ such that for any surface ball $\Delta\subset\Gamma$ and any $E\subset \Delta$, we have 
\begin{equation}\label{def Ainfty}
    \frac{\mu(E)}{\mu(\Delta)}\le C\br{\frac{\nu(E)}{\nu(\Delta)}}^\theta \quad \text{and}\quad \frac{\nu(E)}{\nu(\Delta)}\le C\br{\frac{\mu(E)}{\mu(\Delta)}}^\eta.
\end{equation}
We refer the readers to \cite{coifman1974weighted} for equivalent definitions of $A_\infty$.

The weak DKP condition was first introduced in \cite{david2021carleson},
in connection with the regularity of the Green function.
In the aforementioned recent work of Bortz, Toro and Zhao \cite{bortz2021elliptic},
the authors show that the weak DKP condition 
is sufficient for $A_\infty$ of the elliptic measure in the upper half-space (see Lemma \ref{lem DKPtoAinfty}).

Next we recall our main condition on $\Gamma$. We use the simplest definition of
$CASSC$ for our purposes, and will provide further comments later in Section \ref{Scassc}.

\ms
\begin{defi}[Chord-arc surfaces with small constant] \label{d1b12} 
Let $\varepsilon \in (0,10^{-1})$ be given.
Let $\Gamma$ be a closed, unbounded set in $\R^{d+1}$. We say that
$\Gamma$ is a chord-arc surface with constant $\varepsilon$, and we write
$\Gamma \in CASSC(\varepsilon)$, when for $x\in \Gamma$
and $r > 0$, we can find an $\varepsilon$-Lipschitz graph $G_{x,r}$
that meets $B(x,r/2)$ and such that
\begin{equation} \label{1b12}
\H^d(\Gamma \cap B(x,r) \sm G_{x,r}) 
+ \H^d(G_{x,r} \cap B(x,r) \sm \Gamma) \leq \varepsilon r^d.
\end{equation}

By $\varepsilon$-Lipschitz graph, we mean a set $G = \big\{ x + A(x) \, ; \, x \in P \big\}$,
where $P \subset \R^{d+1}$ is a vector hyperplane, $A : P \to P^\perp$ is a Lipschitz function
with norm at most $\varepsilon$, and $P^\perp$ is the vector line orthogonal to $P$.
\end{defi}

\ms
In this paper, $\Omega$ will always be one of the two connected component of $\Gamma$,
where either $\Gamma$ is the hyperplane $\R^d \subset \R^{d+1}$, or 
$\Gamma \in CASSC(\varepsilon)$ for a small enough $\varepsilon > 0$.
In both cases, $\Gamma$ is uniformly rectifiable, $\Omega$ satisfies the (two-sided) NTA condition,
and for $X \in \Omega$ there is a standard definition of the $L$-elliptic measure $\omega^X$
with pole at $X$ and the Green function $G_L(X,Y)$, $Y \in \Omega$, as long as $L$
is elliptic. As we mentioned before, since in addition $A$ satisfies the weak DKP condition, $\omega^X$ is even 
absolutely continuous with respect to the surface measure $\sigma = \H^d_{\vert \Gamma}$,
and the density $\frac{\d \omega}{d \sigma}$ is an $A_\infty$ weight when $\Gamma=\rd$  by \cite{bortz2021elliptic}\footnote{In fact, they prove the result for $C^1$-square domains.}. We will see that this is also true when $\Gamma\in CASSC(\eps)$.
With the usual DKP condition, this is known and is stated in \cite{hofmann2021uniform}, but factually is a straightforward combination of \cite{kenig2001dirichlet} with \cite{david1990lipschitz} as the domains in question are NTA.  However, the main point of the paper is not the weaker DKP condition, but the fact that the $A_\infty$ 
constant is as small as we want.

Our main statement will be easier to state with the the elliptic measure $\omega^\infty$
with a pole at infinity which we will define in Section \ref{SGrHa} 
as a normalized limit of measures $\omega^X$, with $|X|$ tending to infinity.
One advantage, compared to the usual collection of measures $\{ \omega^X \}$, $X\in \Omega$,
is that $A_\infty$ is easier to define. Appropriately modified statements are valid for the case of a finite pole as well, and will be stated below. Recall that $\Gamma \in CASSC(\varepsilon)$ 
for a small enough $\varepsilon > 0$, and that $\sigma = \H^d_{\vert \Gamma}$.

\ms
\begin{defi}[$A_\infty$ with small constant] \label{d1b14} 
Let $\delta \in (0,10^{-1})$ be given; we say that the positive Borel measure $\omega$ on 
$\Gamma$ lies in $A_\infty(\sigma,\delta)$ when
\begin{equation}\label{1b15} 
  \abs{\frac{\omega(E)}{\omega(\Delta)}-\frac{\sigma(E)}{\sigma(\Delta)}}<\delta
 \end{equation}
for every surface ball $\Delta = B(x,r) \cap \Gamma$ and every Borel subset $E\subset \Delta$.
\end{defi}

\ms
We will see later that this is also equivalent (modulo the precise value of $\delta$) 
to saying that $\omega$ is absolutely continuous with respect to $\sigma$, and the
corresponding density $k = \frac{d\omega}{d\sigma}$ is such that 
$||\log(k)||_{BMO} \leq \tau$, where $\tau$ tends to $0$ with $\delta$.
Notice that none of the two conditions change when we multiply $\sigma$ or $\omega$
by a constant. 

The main result of the present paper can be formulated as follows.

\begin{thm}\label{t1b1}
For every $\delta > 0$, we can find $\varepsilon > 0$, that depends only on
$\delta$, $d$, and $\mu_0$, such that if $\Gamma \in CASSC(\varepsilon)$,
$L=-\divg{A\nabla}$ is an elliptic operator with ellipticity $\mu_0$, which satisfies
the weak DKP condition with norm $\cN(A)<\varepsilon$, then the associated harmonic measure
$\omega^\infty$ lies in $A_\infty(\sigma,\delta)$.
\end{thm}

\ms
As we shall see in Corollary \ref{cvmo}, 
we can replace the conclusion with the essentially equivalent fact that $\omega$ is absolutely continuous with respect to $\sigma = \H^d_{\vert \Gamma}$, and its density 
$ \frac{d\omega}{d\sigma}$ satisfies
\begin{equation} \label{1b17}
\norm{\log \frac{d\omega}{d\sigma}}_{BMO(\Gamma)}\le \delta.
\end{equation}

We also have analogous local results, where we consider the elliptic measure 
$\omega^X$ with a finite pole, but then need to restrict to small enough balls.

\begin{thm}\label{t1b2}
For every choice of $\delta > 0$ and $\kappa > 1$, we can find 
$\varepsilon = \varepsilon(\delta, d, \mu_0) > 0$
and $\tau = \tau(\delta, d, \mu_0, \kappa) \in (0,1)$ with the following properties.
Suppose $\Gamma \in CASSC(\varepsilon)$, and $L=-\divg{A\nabla}$ is an elliptic operator with 
ellipticity $\mu_0$ which satisfies the weak DKP condition with norm $\cN(A)<\varepsilon$.
Let $X \in \Omega$ be given, and denote by $\omega^X$ the elliptic measure on 
$\Gamma$ associated to $L$ and with pole at $X$. Then 
\begin{equation}\label{1b19} 
  \abs{\frac{\omega^{X}(E)}{\omega^{X}(\Delta)}-\frac{\sigma(E)}{\sigma(\Delta)}}<\delta
 \end{equation}
for every surface ball $\Delta = B(x,r) \cap \Gamma$ such that 
$\Delta \subset B(X,\kappa \dist(X,\Gamma))$ and $0 < r \leq \tau \dist(X,\Gamma)$,
and every Borel subset $E\subset \Delta$.
\end{thm}

Theorem \ref{t1b2} will be deduced from Theorem \ref{t1b1} 
at the end of Section~\ref{SGrHa}, and we will see in Corollary \ref{cvmo} 
that \eqref{1b19} can be replaced by the fact that $k = \frac{d\omega\SM{^X}}{d\sigma}$
is defined on $\Delta$, and $||\log(k)||_{BMO(\Delta)} \leq \delta$.

The rest of this paper is organized as follows. In Section~\ref{SGrHa}, we recall  definitions and some well-known properties of the elliptic measures and the Green function. We also deduce the local version (Theorem \ref{t1b2}) of Theorem \ref{t1b1} using Lemma \ref{lem polediff} that allows us to compare elliptic measures with different poles. In Section~\ref{Sexample}, we prove Theorem \ref{t1b1} for the upper half-space. In Section~\ref{SLip}, we define a change of variables that allows us to prove Theorem \ref{t1b1} for small Lipschitz graphs. In Section~\ref{Scassc}, we treat the $CASSC$ case, using the result for small Lispchitz graphs. In Section~\ref{Sbmo}, we define $BMO$ and $VMO$, and show the relations between $BMO$ (with a small semi-norm) and $A_\infty$ (with a small constant). In Section~\ref{sec emlimit_pf}, we give the proof of a lemma that we require in Section~\ref{Sexample}, which allows us to take the limit of elliptic measures.

\section{Green functions and harmonic measure with a pole at infinity}
\label{SGrHa}

Let $\Gamma \subset \R^{d+1}$ be a $CASSC$, and $L = -\div A \nabla$ 
be a divergence form elliptic operator on $\Omega$, one of the components of 
$\R^{d+1} \sm \Gamma$. In this section we only use the ellipticity of
$L$, the Ahlfors regularity of $\Gamma$, and the NTA character of $\Omega$
(see Section \ref{Scassc}), and recall the definition and basic properties of the 
Green function and ellipitc measure with a pole at infinity. 

We start with the $L$-elliptic measure with a pole in $\Omega$. Recall that 
for $X \in \Omega$, there exists a unique Borel probability measure $\omega^X = \omega^X_L$
on $\bdy\Omega$, such that for $f \in C_c^\infty(\bdy\Omega)$ the function
\[u(X) = \int_{\bdy\Omega} f(y) \, d\omega_L^X(y)\]
is the unique weak solution to the Dirichlet problem
\[(D)_L \begin{cases}
Lu = 0 \quad\text{in } \Omega, \\
u|_{\bdy\Omega} = f
\end{cases}\]
satisfying $u \in C(\overline{\Omega})$ and $u(X) \to 0$ as $|X| \to \infty$ in $\Omega$. We call 
$\omega^X_L$ the $L$-elliptic measure (or the elliptic measure corresponding to $L$) with pole at $X$.

We can also associate to $L$ a unique Green function in $\Omega$, 
$G_L(X,Y): \Omega \times \Omega \setminus  \text{diag}(\Omega) \to \Real$ 
with the following properties (see e.g. \cite{HOFMANN20176147}): 
$G_L(\cdot,Y)\in W^{1,2}_{\rm loc}(\Omega\setminus \{Y\})\cap C(\overline{\Omega}\setminus\{Y\})$, 
$G_L(\cdot,Y)\big|_{\bdy\Omega}\equiv 0$ for any $Y\in\Omega$, 
and $L G_L(\cdot,Y)=\delta_Y$ in the weak sense in $\Omega$, that is,
\begin{equation}\label{Greendef}
    \iint_\Omega A(X)\,\nabla_X G_{L}(X,Y) \cdot\nabla\vp(X)\, dX=\vp(Y), \qquad\text{for any }\vp \in C_c^\infty(\Omega).
\end{equation}
In particular, $G_L(\cdot,Y)$ is a weak solution to $L G_L(\cdot,Y)=0$ in $\Omega\setminus\{Y\}$. Moreover,
\begin{equation}\label{eq2.green}
G_L(X,Y) \leq C\,|X-Y|^{1-d}\quad\text{for }X,Y\in\Omega\,,
\end{equation}
\begin{equation}\label{eq2.green2}
c_\theta\,|X-Y|^{1-d}\leq G_L(X,Y)\,,\quad {\rm if } \,\,\,|X-Y|\leq \theta\, \dist(X,\bdy\Omega)\,, \,\, \theta \in (0,1)\,,
\end{equation}
\begin{equation}
\label{eq2.green3}
G_L(X,Y)\geq 0\,,\quad G_L(X,Y)=G_{L^\top}(Y,X), \qquad \text{for all } X,Y\in\Omega\,,\, 
X\neq Y.
\end{equation}
Here, and in the sequel, $L^\top$ is the 
adjoint of $L$ defined by 
$L^\top=-\divg A^\top\nabla$, where $A^\top$ denotes the transpose matrix of $A$.

We have the following Caffarelli-Fabes-Mortola-Salsa estimates 
(cf. \cite{CFMS81}, and \cite{jerison1982boundary} for NTA domains). 

\begin{lem} \label{lem cfms}
Let $\Omega$ be an NTA domain in $\ree$ and $L=-\divg{A\nabla}$ be a $\mu_0$-elliptic operator.
Let $x\in \bdy\Omega$, $0<r<\diam(\bdy\Omega)$. Then for
$X\in \Omega\setminus B(x,2r)$ we have
\begin{equation}\label{eq.CFMS}
\frac1C\omega_L^X(\Delta(x,r))
\leq
r^{d-1}G_L(X,X_{x,r}) \leq C \omega_L^X(\Delta(x,r)),
\end{equation}
where $X_{x,r}$ is a corkscrew point for $x$ at scale $r$ (see Definition \ref{tscs.def}).
The constant in \eqref{eq.CFMS} depends only on $\mu_0$,
dimension and on the constants in the NTA character.
\end{lem}

\begin{lem}[Bourgain's estimate \cite{bourgain1987hausdorff, kenig1994harmonic}]\label{lem Bourgain}
Let $\Omega$ and $L$ be as in Lemma \ref{lem cfms}. Let $x\in\pom$ and $0<2r<\diam(\pom)$. Then 
\[
\omega_L^{X_{x,r}}(\Delta(x,r))\ge\frac{1}{C},
\]
where $C>1$ depends on $d$, $\mu_0$ and the NTA constant of $\om$.
\end{lem}

\begin{lem}\label{lcomparison}
Let $\Omega$ be an NTA domain. There is a constant $M>1$, that depends on the dimension and the NTA constants for $\Omega$,
such that when $0 < Mr < \diam(\partial\Omega)$, and $u, v$ are non-trivial functions which vanish continuously
on $\Delta(x, Mr)$ for some $x\in\bdy\Omega$, $u, v \ge 0$ and $Lu = Lv = 0$ in $B(x, Mr)\cap\om$, then 
\[
\frac{u(X)}{v(X)} \approx \frac{u(X_{x,r})}{v(X_{x,r})}, \quad \text{for all } X \in B(x,r) \cap \om,
\]
where the implicit constants depend on $n$, $\mu_0$ and the NTA constants for $\om$.
\end{lem}

We now consider poles at infinity. With minor modifications only, 
one can prove the following lemma as in \cite{kenig1999free}, 
Lemma 3.7 and Corollary 3.2.

\begin{lem}\label{lem emGreen_infinity}
Let $L=-\divg{A\nabla}$ be an elliptic operator on $\Omega$ with Green function $G_L(X,Y)$. 
Then for any fixed $Z_0\in\Omega$, there exists a unique function $U\in C(\overline{\om})$ such that 
\[
\begin{cases}
L^\top U = 0 \quad \text{in } \Omega\\
U>0 \quad \text{in } \Omega\\
U=0 \quad \text{on }\bdy\Omega,\\
\end{cases}
\]
and $U(Z_0)=1$. In addition, for any sequence $\set{X_k}_k$ of points in $\Omega$
such that $\abs{X_k}\to \infty$, there exists a subsequence (which we still denote by 
$\set{X_k}_k$) such that if we set 
\[u_k(Y) := \frac{G_L(X_k, Y)}{G_L(X_k, Z_0)},\]
then 
\begin{equation}\label{Greeninfty_lmt}
    \lim_{k\to\infty}u_k(Y)=U(Y) \qquad\text{for any }Y\in\Omega.
\end{equation}
Moreover, there exists a unique locally finite positive Borel 
measure $\omega^\infty_L$ on $\bdy\Omega$ such that the Riesz formula 
\begin{equation}\label{Rieszinfty}
\int_{\bdy\Omega} f(y) \, d\omega^\infty_L(y) = - \iint A^\top(Y)\nabla_{Y}U(Y) \cdot \nabla_{Y} F(Y) \, dY 
\end{equation}
holds whenever $f \in C_c^\infty(\bdy\Omega)$ and $F \in C_c^\infty(\Omega)$ are such that $F|_{\bdy\Omega}=f$.
\end{lem}

\begin{defn}[Elliptic measure and Green function with pole at infinity]
Let $L$, $Z_0$, $U$ and $\omega_L^\infty$ be as in Lemma \ref{lem emGreen_infinity}. 
We call $\omega^\infty_L$ the elliptic measure with pole at infinity (normalized at $Z_0$), and $U$ the Green function with pole at infinity (normalized at $Z_0$). 
\end{defn}

We will systematically fix $Z_0$ at the beginning (and take $Z_0=(0,1)$ when $\Omega=\reu$),
so we drop the term ``normalized at $Z_0$" for the elliptic measure and the Green function with pole at infinity.  

Now we state estimates that allow to go from poles $X \in \Omega$ to poles at 
infinity. Here there is a single operator $L$, so we drop $L$ from the notation, and 
for instance write $\omega^X$ for $\omega_L^X$. 
Recall that the different measures $\omega^X$
are absolutely continuous with respect to each other, as a consequence of Harnack's 
inequality. Similarly, $\omega^X$ and $\omega^\infty$ are mutually absolutely continuous
too, as a result of the local H\^older continuity of solutions near $\Gamma$, and we even
have the estimates stated below on the density, which we take from \cite{bortz2021elliptic}.
For $X \in \Omega$, call 
\begin{equation}\label{def Hinfty}
    H_\infty(X,z):=\frac{d\omega^{X}}{d\omega^\infty}(z)
\end{equation}
the Radon-Nykodym density of $\omega^{X}$ with respect to $\omega^\infty$, 
evaluated at $z\in \Gamma$; the existence comes with the next lemma.

\begin{lem}[\cite{bortz2021elliptic}, Lemma 5.8]\label{lemBTZ kernel}
Let $\Gamma$, $\Omega$, and $L$ be as above; in particular $\Gamma \in CASSC(\varepsilon)$
for a small enough $\varepsilon > 0$, and $L-\divg{A\nabla}$ is an elliptic operator with ellipticity $\mu_0$. 
Given $X\in \Omega$, set $\delta(X) = \dist(X,\Gamma)$ and 
pick a point $X_1 \in \Omega \cap B(X, \dist(X,\Gamma))$
such that $\dist(X_1,\Gamma) = \delta(X)/4$. 
The density $H_\infty(X,z)$ of \eqref{def Hinfty} exists and is locally H\"older continuous  
of order $\gamma = \gamma(d,\mu_0)$. 
Moreover, for every $\kappa > 1$ there exists $C_\kappa = C_\kappa(\kappa, d, \mu_0)$ 
such that
\begin{equation}\label{HCinfkern.eq}
|H_\infty(X, z) - H_\infty(X, z')| 
\le C_\kappa \frac{G_L(X, X_1)}{U(X)} \left(\frac{|z - z'|}{\delta(X)}\right)^\gamma 
\end{equation}
for all $z,z' \in \Gamma \cap B(X,6\kappa \delta(X))$, $|z - z'| < \delta(X)/4$ and
\begin{equation}\label{compinfkern.eq}
(C_\kappa)^{-1} \frac{G_L(X, X_1)}{U(X)} \le H_\infty(X, z) 
\le C_\kappa \frac{G_L(X, X_1)}{U(X)}
\end{equation}
for all $z \in \Gamma \cap B(X,6\kappa \delta(X))$. 
Here $U$ is the Green function for $L$ with pole at infinity. 
\end{lem}

In fact this lemma is proved in \cite{bortz2021elliptic} in the special case when $\Gamma = \R^d$,
but the proof goes through. The existence of $X_1 \in \Omega \cap B(X, \dist(X,\Gamma))$
such that $\dist(X_1,\Gamma) = \delta(X)/4$ is easy because $\Gamma$ is a $CASSC$,
but a nearby corkscrew point (not so close to $X$) would do the job in a more general setting.
Notice also that we will only use the case when $\Gamma \neq \R^d$ to deduce
Theorem \ref{t1b2} from Theorem \ref{t1b1} at the end of the section.

We now use the estimates \eqref{HCinfkern.eq} and \eqref{compinfkern.eq} to  
compare $\omega^{X_0}(E)$ and $\omega^\infty(E)$.  

\begin{lem}\label{lem polediff}
Ler $\Gamma$, $\Omega$, $L$, and $X \in \Omega$ be as in Lemma \ref{lemBTZ kernel}
Let $L=-\divg{A\nabla}$ be an elliptic operator with ellipticity $\mu_0$.  
Then for any $\kappa>1$, for $x\in \Gamma \cap B(X,5\kappa \delta(X))$ and 
$0 < r \leq \delta(X)/4$, and any Borel set $E\subset \Delta = \Gamma \cap B(x,r)$,
there holds
\begin{equation}\label{eq polediff0}
    \abs{\frac{\omega^{X}(E)}{\omega^{X}(\Delta)}-\frac{\omega^{\infty}(E)}{\omega^{\infty}(\Delta)}}\le C_\kappa \br{\frac{r}{\delta(X)}}^\gamma,
\end{equation}
where $C_\kappa=C_\kappa(\kappa, d, \mu_0)$ and $\gamma=\gamma(d,\mu_0)$.
\end{lem}

\begin{proof}
By the definition \eqref{def Hinfty} of $H_\infty$, 
\[
\frac{\omega^{X}(E)}{\omega^{X_0}(\Delta)}
=\frac{\int_EH_\infty(X,z) d\omega^\infty(z)}{\int_\Delta H_\infty(X,z)d\omega^\infty(z)}.
\] 
Next 
\begin{equation*}
 \frac{\underset{z\in E}{\inf}H_\infty(X,z)}{\underset{z\in \Delta}{\sup}H_\infty(X,z)} \,\frac{\omega^\infty(E)}{\omega^\infty(\Delta)}\le \frac{\omega^{X}(E)}{\omega^{X}(\Delta)}\le \frac{\underset{z\in E}{\sup}H_\infty(X,z)}{\underset{z\in \Delta}{\inf}H_\infty(X,z)} \,\frac{\omega^\infty(E)}{\omega^\infty(\Delta)},
\end{equation*}
which implies that 
\begin{multline}\label{eq polediff}
    \br{\frac{\underset{z\in E}{\inf}H_\infty(X,z)}{\underset{z\in \Delta}{\sup}H_\infty(X,z)}-1}\frac{\omega^\infty(E)}{\omega^\infty(\Delta)}
    \le \frac{\omega^{X}(E)}{\omega^{X}(\Delta)}-\frac{\omega^\infty(E)}
    {\omega^\infty(\Delta)}\\
    \le \br{\frac{\underset{z\in E}{\sup}H_\infty(X,z)}{\underset{z\in \Delta}{\inf}H_\infty(X,z)}-1}\frac{\omega^\infty(E)}{\omega^\infty(\Delta)}.
\end{multline}
Notice that
\[
 \br{\frac{\underset{z\in E}{\sup}H_\infty(X,z)}{\underset{z\in \Delta}{\inf}H_\infty(X,z)}-1}\frac{\omega_L^\infty(E)}{\omega_L^\infty(\Delta)}
 \le\frac{\underset{z\in \Delta}{\sup}H_\infty(X,z)}{\underset{z\in \Delta}{\inf}H_\infty(X,z)}-1,
 \]
 and that 
 \[
 \frac{\underset{z\in E}{\inf}H_\infty(X,z)}{\underset{z\in \Delta}{\sup}H_\infty(X,z)}-1\le\br{\frac{\underset{z\in E}{\inf}H_\infty(X,z)}{\underset{z\in \Delta}{\sup}H_\infty(X,z)}-1}\frac{\omega_L^\infty(E)}{\omega_L^\infty(\Delta)}\le0,
 \]
 so \eqref{eq polediff} gives that 
 \begin{equation}\label{eq polediffH}
     \abs{\frac{\omega_L^{X}(E)}{\omega^{X}(\Delta)}-\frac{\omega^\infty(E)}{\omega^\infty(\Delta)}}
     \le\sup_{z,z'\in \Delta}\abs{\frac{H_\infty(X,z)}{H_\infty(X,z')}-1}.
 \end{equation}
But Lemma \ref{lemBTZ kernel} says that
for all $z,z' \in \Gamma \cap B(X,6\kappa \delta(X))$ such that $|z - z'| < \frac{\delta(X)}{4}$,
\[
\abs{\frac{H_\infty(X,z)}{H_\infty(X,z')}-1}=\frac{\abs{H_\infty(X,z)-H_\infty(X,z')}}{H_\infty(X,z')}\le C_\kappa\br{\frac{\abs{z-z'}}{\delta(X)}}^\gamma.
\]
Then the desired estimate \eqref{eq polediff0} follows from \eqref{eq polediffH} and our assumptions on $\Delta$. 
\end{proof}
\ms

We now use Lemma \ref{lem polediff} to deduce Theorem \ref{t1b2} from Theorem \ref{t1b1}.

\begin{proof}[Proof of Theorem \ref{t1b2} given Theorem \ref{t1b1}]
Let $\Gamma$, $\Omega$, and $L$ be as in both theorems. By Theorem \ref{t1b1}
we can find $\varepsilon >0$, depending on $\delta$, $d$, and $\mu_0$,
such that if $\Gamma \in CASSC(\varepsilon)$ and $\cN(A)<\varepsilon$,
we have 
\[
 \abs{\frac{\omega^{\infty}(E)}{\omega^{\infty}(\Delta)}-\frac{\sigma(E)}{\sigma(\Delta)}}<\frac{\delta}{2}
\]
for any surface ball $\Delta \subset \Gamma$ and any Borel set $E \subset \Delta$.
We now have to replace $\omega^{\infty}$ with $\omega^{X}$. Let 
$\Delta = \Gamma \cap B(x,r)$ be such that $\Delta \leq B(X,\kappa \delta(X))$
and $0 < r < \tau \delta(X)$ be as in the statement. Then we can apply 
Lemma \ref{lem polediff}, and we get \eqref{eq polediff0}, i.e., 
\[
\abs{\frac{\omega^{X}(E)}{\omega^{X}(\Delta)}-\frac{\omega^{\infty}(E)}{\omega^{\infty}(\Delta)}}\le C_\kappa \br{\frac{r}{\delta(X)}}^\gamma \leq C_\kappa \tau^\gamma.
\]
We now choose $\tau$ so small, depending on $\delta$, $d$, $\mu_0$, and $\kappa$,
that $C_\kappa \tau^\gamma \leq \delta/2$, and \eqref{1b19} follows from triangle inequality.
\end{proof}

\section{Construction of an example in $\reu$}
\label{Sexample}

In this section we prove Theorem \ref{t1b1} in the special case when $\Omega = \reu$.
As was said in the introduction, we shall assume that the theorem fails, and get the desired contradiction by constructing an operator $L$ that satisfies the assumption of the main theorem
of \cite{bortz2021elliptic}, but not the conclusion.
We recall the main theorem of \cite{bortz2021elliptic}. 
\begin{lem}[\cite{bortz2021elliptic}, Theorem 1.2]\label{BTZthm}
Let $L = -\divg A\nabla$ be a divergence form elliptic operator on $\reu$, whose coefficient matrix $A$ satisfies the vanishing weak DKP condition. If $k^\infty_L$ is the elliptic kernel associated to $L$ (in $\reu$) with pole at infinity then $\log k_L^\infty \in VMO(\rd)$ (see \eqref{6a3} for the definiton of $VMO$).
\end{lem}
The authors of \cite{bortz2021elliptic} also show the big constant version of the result above.
\begin{lem}[\cite{bortz2021elliptic}, Theorem 5.2]\label{lem DKPtoAinfty}
Let $L = -\divg A\nabla$ be a divergence form elliptic operator on $\reu$, whose coefficient matrix $A$ satisfies the weak DKP condition. Let $\omega^\infty_L$ be the elliptic measure with pole at infinity. Then $\omega^\infty_L \ll \mathcal{L}^d$, where $\mathcal{L}^d$ is the Lebesgue measure on $\rd$, $\omega_L^\infty \in A_\infty(dx)$, and $k^\infty_L(y) := \tfrac{d\omega^\infty_L}{dx}(y)$ has the property that $\log k^\infty_L \in BMO(\rd)$. The implicit constants in the statements $\omega_L^\infty \in  A_\infty(dx)$ and $\log k \in BMO(\rd)$ are each bounded by a constant depending on $d$, 
the ellipticity constant $\mu_0$, 
and 
$\left\|\alpha_A(x,r)^2 \, \frac{dx\, dr}{r}\right\|_{\mathcal{C}}$.
\end{lem}

In this section, $\Gamma = \R^d \subset \R^{d+1}$, $\Omega = \reu$, and we study operators
$L=-\divg(A\nabla)$ that are $\mu_0$-elliptic. Points of $\reu$ will be written
$X = (x,t)$, with $x\in \R^d$ and $t > 0$, and surface balls will be denoted
by $\Delta(x,r) = B_r(x) \cap \set{t=0} \subset \R^d$, with $x\in \R^d$ and $r > 0$.

If $\Delta = \Delta(x,r)$ is a surface ball, we may also denote by $T_\Delta$
the Carleson box $T(x,r)= B_r(x) \cap \reu$ over $\Delta(x,r)$, but our notation 
will be slightly simpler if we switch to
\begin{equation} \label{TT}
W(x,r):=\Delta(x,r)\times\Bigl(\frac{r}{2},r\Bigr] \subset \Real^{d+1}_+
\end{equation}
for the Whitney cube. This makes no real difference (compared to \eqref{1b2}),
and in particular the definition of Carleson measures and the weak DKP condition do not
change in any significant way.

So we start the proof of Theorem \ref{t1b1} in $\reu$ by contradiction: we assume 
the theorem to be false, so we can find an ellipticity constant $\mu_0$, a small
constant $\delta_0>0$, and for each integer $j \geq 0$, a $\mu_0$-elliptic matrix $A_j$
such that
\begin{equation}\label{AjDKP_epsj}
    \cN(A_j)\le \epsilon_j^2, \quad\text{where }\epsilon_j=2^{-j},
\end{equation}
but for which the conclusion of the theorem fails for $L_j=-\divg(A_j\nabla)$. That is,
$\omega^\infty_{L_j} \notin A_\infty(\sigma,\delta_0)$, where by habit we still denote 
by $\sigma$ the Lebesgue measure on $\R^d$. Notice that changing the normalization
point for $\omega^\infty_{L_j}$ would only multiply $\omega^\infty_{L_j}$ by a constant, and not affect 
the fact that $\omega^\infty_{L_j} \notin A_\infty(\sigma,\delta_0)$.

This means that for each $j$ we can find a surface ball $\Delta_j \subset \R^d$ and a Borel set 
$E_j \subset \Delta_j$ such that
\begin{equation}\label{bigAinfty}
    \abs{\frac{\omega^{\infty}_{L_j}(E_j)}{\omega_{L_j}^{\infty}(\Delta_j)}-\frac{\sigma(E_j)}{\sigma(\Delta_j)}}\ge\delta_0.
\end{equation}
We want to use this to derive a contradiction, 
but before we start for good, it will be good to record how our various objects
behave under dilations that preserve the boundary.

\subsection{Some scaling properties}
In this subsection we are only interested in the changes of variable induced by 
the linear transformations
\begin{equation}\label{def Tlz}
    \Tlz: \reu \to \reu,  \quad X\mapsto \lambda X+(z,0) = \lambda X+Z,
\end{equation}
where $\lambda > 0$, $z\in \R^d$, and we set $Z = (z,0) \in \R^{d+1}$ to prevent
confusion between $\R^d$ and $\R^{d+1}$.
In the following computations, $\lambda$ and $z$ are fixed, we are given an elliptic matrix
valued function $A$, and we consider 
\begin{equation} \label{3b5}
\wt A(X)=A(\Tlz(X)).
\end{equation}

\begin{lem}\label{lem rsc_DKP}
With the notation above,
\begin{equation}\label{rsc_alpha}
    \alpha_{\wt A}(x,r)=\alpha_A(\Tlz(x,r)) \qquad\text{for any }x\in\rd, r>0,
\end{equation}
and 
\begin{equation}\label{rsc_dkp}
    \cN(\wt A)=\cN(A).
\end{equation}
\end{lem}
\begin{proof}
Both \eqref{rsc_alpha} and \eqref{rsc_dkp} follow from direct computations; we omit the proof.
\end{proof}

\begin{lem}\label{lem rsc_eminfty}
Now let $\wt L=-\divg\wt A\nabla$. Then for any Borel set $E\in\rd$, 
\begin{equation}
    \omega_L^\infty(E)
    =\lambda^{d-1}U(z,\lambda) \ \omega^\infty_{\wt L}\br{\frac{E-(z,0)}{\lambda}}
    = \lambda^{d-1}U(z,\lambda) \ \omega^\infty_{\wt L}((\Tlz)^{-1}(E)),
\end{equation}
where $U$ is the Green function with pole at infinity associated to $L$ (and normalized at the point $Z_0=(0,1)$).
\end{lem}

\begin{proof}
By \eqref{Greendef}, one can show that 
\begin{equation}\label{rsc_Green}
    G_{\wt L}(X,Y)=\lambda^{1-d}G_L(\Tlz(X),\Tlz(Y)) \qquad \text{ for } X,Y\in\reu.
\end{equation}
Denote by $\wt U$ the Green function with pole at infinity for $\wt L$. 
Then by \eqref{Greeninfty_lmt} and \eqref{rsc_Green},
\[
\wt U(Y)=\lim_{k\to\infty}\frac{G_{\wt L}(X_k,Y)}{G_{\wt L}(X_k,Z_0)}=\lim_{k\to\infty}\frac{G_{L}(\Tlz(X_k),\Tlz(Y))}{G_{ L}(\Tlz(X_k),\Tlz(Z_0))},
\]
where $\set{X_k}$ is some sequence of points in $\reu$ such that $\abs{X_k}\to\infty$ as $k\to\infty$. On the other hand, we can write 
\begin{multline*}
    U(Y)=\lim_{k\to\infty}\frac{G_L(\Tlz(X_k),Y)}{G_L(\Tlz(X_k),Z_0)}
    =\lim_{k\to\infty}\frac{G_L(\Tlz(X_k),Y)}{G_L(\Tlz(X_k),\Tlz(Z_0))}\frac{G_L(\Tlz(X_k),\Tlz(Z_0))}{G_L(\Tlz(X_k),Z_0)}\\
    =\lim_{k\to\infty}\frac{G_L(\Tlz(X_k),Y)}{G_L(\Tlz(X_k),\Tlz(Z_0))}U(\Tlz(Z_0)).
\end{multline*}
Comparing it to $\wt U(Y)$, we obtain that
\begin{equation}\label{rsc_U}
    U(\Tlz(Y))=\wt U(Y)U(\Tlz(Z_0)) \qquad\text{for any }Y\in\reu.
\end{equation}
Recall that the Riesz formula \eqref{Rieszinfty} for $\wt L$ asserts that 
\begin{equation}\label{eq Riesz_wtL}
    \int_{\bdy\Omega} \wt f(y) \, d\omega^\infty_{\wt L}(y) = - \iint \wt A^\top(Y)\nabla_{Y}\wt U(Y) \cdot \nabla_{Y} \wt F(Y) \, dY
\end{equation}
for $\wt f \in C_c^\infty(\rd)$ and $\wt F \in C_c^\infty(\ree)$ such that $\wt F(y,0)=\wt f(y)$. Set $F(\Tlz(Y))=\wt F(Y)$, $f(y)=F(y,0)$. Then $f \in C_c^\infty(\rd)$ and $F \in C_c^\infty(\ree)$. 
Using \eqref{rsc_U}, the change of variable $Y' = \Tlz(Y)$, and then the Riesz formula for $L$, the right-hand side 
of \eqref{eq Riesz_wtL} is equal to
\begin{multline*}
   -\lambda^{1-d}U(\Tlz(Z_0))^{-1}\iint A^\top(Y')\nabla U(Y') \cdot \nabla F(Y') \, dY'\\  
   =\lambda^{1-d}U(z,\lambda)^{-1}\int_{\rd}f(y)d\omega_L^\infty(y),
\end{multline*}
where we have also used the fact that 
$U(\Tlz(Z_0))=U(z,\lambda)$. 
Since $\wt f(y)=\wt F(y,0)=F(\Tlz(y,0))=f(\lambda y+z)$,
   \eqref{eq Riesz_wtL} asserts that 
   \[
   \int_{\rd}f(\lambda y+z)d\omega^\infty_{\wt L}(y)
   =\lambda^{1-d}U(z,\lambda)^{-1}\int_{\rd}f(y)d\omega_L^\infty(y)
   \]
   for $f\in C_c^\infty(\rd)$. By a limiting argument, this implies that for any Borel set $E\subset\rd$, and if we set $E' = \lambda^{-1}(E-(z,0)) = (\Tlz)^{-1}(E)$, 
   \begin{multline*}
  \omega_{\wt L}^\infty(E)
  =\int_{\rd}\1_{E}(\lambda y+z)d\omega_{\wt L}^\infty(y)
  =\lambda^{1-d}U(z,\lambda)^{-1}\int_{\rd}\1_{E'}(y)d\omega_L^\infty(y) 
  \\
      =\lambda^{1-d}U(z,\lambda)^{-1}\omega_L^\infty(E'),
   \end{multline*}
   as desired.
\end{proof}

\subsection{We can modify $A_j$ away from a band}
Return to the bad operator $L_j=-\divg(A_j\nabla)$ and its bad set $E_j \subset \Delta_j$.
We claim that we can assume that for all $j$, $\Delta_j = \Delta_0$, the unit ball in 
$\R^d$. Write $\Delta_j = \Delta(z,\lambda)$, and notice that $\Tlz(\Delta_0) = \Delta_j$. 
Then replace $A_j$ with the function $\wt A_j = A \circ \Tlz$ given by \eqref{3b5}.
By Lemma \ref{lem rsc_DKP}, $\cN(\wt A_j)=\cN(A_j)$. In addition, set
$\wt E_j = (\Tlz)_{-1}(E_j) \subset (\Tlz)_{-1}(\Delta_j) = \Delta_0$; 
then by Lemma~\ref{lem rsc_eminfty},
\[
\frac{\omega_{L_j}^\infty(E_j)}{\omega_{L_j}^\infty(\Delta_j)}
=\frac{\omega_{\wt L_j}^\infty(\wt E_j)}{\omega_{\wt L_j}^\infty(\Delta_0)},
\]
where the extra factors $\lambda^{d-1}U(z,\lambda)$ are the same on the numerator and denominator and cancel. 
So let us assume that $\Delta_j = \Delta_0$ (and that \eqref{AjDKP_epsj} and \eqref{bigAinfty} hold).

Our goal is to use the $A_j$ to construct a matrix-valued function $A^*$
that satisfies the vanishing weak DKP condition and keeps the bad properties of each $A_j$.
The next lemma will allow us to modify $A_j$ outside of some strip, but essentially maintain
\eqref{bigAinfty}.

\begin{lem}\label{lem strip}
There exist $0<\rho_j<1$ and $M_j>1$ such that if $A_j^*$ is any $\mu_0$-elliptic matrix that satisfies $A_j^*(x,t)=A_j(x,t)$ for $t\in (\rho_j,M_j)$, and $\cN(A_j^*)\le C_0$ for some $C_0<\infty$, then 
\begin{equation}\label{bigAinfty2}
    \abs{\frac{\omega^{\infty}_{L_j*}(E_j)}{\omega_{L_j*}^{\infty}(\Delta_0)}-\frac{\sigma(E_j)}{\sigma(\Delta_0)}}\ge\frac{\delta_0}{2},
\end{equation}
where $\omega_{L_j*}^{\infty}$ corresponds to the operator $L_j*=-\divg{A_j^*\nabla}$.
\end{lem}

We could also have allowed modifications of $A_j$ for $|x|$ very large, but since this is not necessary for our construction, we will not do that.

\begin{proof}
Notice that we allow $\rho_j$ and $M_j$ to depend wildly on $j$. 
We prove the lemma by contradiction. Fix $j$ and suppose that the statement is false; then 
there exists a sequence of $\mu_0$-elliptic matrices $\set{A_{j,k}}_{k\in\NN}$, such that $A_{j,k}(x,t)=A_j(x,t)$ for $t\in(2^{-k},2^k)$, $\cN(A_{j,k})\le C_0$, and 
\begin{equation}\label{eq contradict1}
    \abs{\frac{\omega^{\infty}_{L_{j,k}}(E_j)}{\omega_{L_{j,k}}^{\infty}(\Delta_0)}-\frac{\sigma(E_j)}{\sigma(\Delta_0)}}<\frac{\delta_0}{2}, \qquad k\in\NN.
\end{equation}
Let $X=(0,t_0)$, with $t_0>4$ to be chosen soon. Then by Lemma \ref{lem polediff}
(applied with $\kappa = 2$, $x=0$, and $r = 1$), 
\begin{equation*}
    \abs{\frac{\omega_{L_j}^{X}(E_j)}{\omega_{L_j}^{X}(\Delta_0)}-\frac{\omega_{L_j}^{\infty}(E_j)}{\omega_{L_j}^{\infty}(\Delta_0)}}\le C\br{\frac{1}{t_0}}^\gamma
    = C {t_0}^{-\gamma},
\end{equation*}
where $C$ and $\gamma$ are positive constants that depend only on $d$ and $\mu_0$.
Choose $t_0$ so large (depending on $d$ and $\mu_0$) that 
$C {t_0}^{-\gamma} \leq \frac{\delta_0}{8}$; then 
\begin{equation}\label{eq poles_eps0}
    \abs{\frac{\omega_{L_j}^{X}(E_j)}{\omega_{L_j}^{X}(\Delta_0)}-\frac{\omega_{L_j}^{\infty}(E_j)}{\omega_{L_j}^{\infty}(\Delta_0)}}\le \frac{\delta_0}{8}.
\end{equation}
By the triangle inequality, \eqref{bigAinfty} (recall that $\Delta_j = \Delta_0$) 
and \eqref{eq poles_eps0} yield
\begin{equation}\label{eq X0_gteps0}
  \abs{\frac{\omega^{X}_{L_j}(E_j)}{\omega_{L_j}^{X}(\Delta_0)}-\frac{\sigma(E_j)}{\sigma(\Delta_0)}}\ge\frac{7}{8}\delta_0,  
\end{equation}
while \eqref{eq contradict1} and the analogue of \eqref{eq poles_eps0} 
for $L_{j,k}$ (which is valid with uniform bounds) give that 
\begin{equation}\label{eq contradict2}
    \abs{\frac{\omega^{X}_{L_{j,k}}(E_j)}{\omega_{L_{j,k}}^{X}(\Delta_0)}-\frac{\sigma(E_j)}{\sigma(\Delta_0)}}<\frac38\delta_0 \qquad \text{for } k\in\NN.
\end{equation}
Now we want to let $k$ tend to $+\infty$. We claim that for any bounded Borel set $E\subset\rd$, 
\begin{equation}\label{eq em_limitj}
    \lim_{k\to\infty}\omega^{X}_{L_{j,k}}(E)=\omega^{X}_{L_j}(E).
\end{equation}
Once we prove this, we let $k$ tend to $+\infty$ in \eqref{eq contradict2}
and get a contradiction with \eqref{eq X0_gteps0}. 
So Lemma \ref{lem strip} follows from the following lemma (proved later).
\end{proof}

\begin{lem}\label{lem em_limit}
Let $L=-\divg{A\nabla}$ be a $\mu_0$-elliptic operator, let  $L_k=-\divg{A_k\nabla}$ be 
a $\mu_0$-elliptic operator that satisfies $A_k(x,t)=A(x,t)$ for $t\in (2^{-k}, 2^k)$, 
and $\cN(A_k)\le C_0$ for some $C_0<\infty$ for all $k\in\ZZ_+$. 
Then for any bounded Borel set $E\subset\rd$, and any $X_0\in\reu$, there holds
\begin{equation}\label{eq em_limit}
    \lim_{k\to\infty}\omega^{X_0}_{L_k}(E)=\omega_L^{X_0}(E).
\end{equation}
\end{lem}
This lemma should not surprise the reader, but 
something like the uniform bound $\cN(A_k)\le C_0$  that guarantees $A_\infty$ is needed. 
The weak convergence of the elliptic measures, on the other hand, does not require anything for the operator apart from ellipticity. In fact, it is well known that assuming $A_k\to A$ a.e. as $k\to\infty$, $A_k$ and $A$ elliptic, then for any fixed $X_0\in\reu$, $\int f\,d\omega_{L_k}^{X_0}$ tends to $\int f\,d\omega_{L}^{X_0}$ as $k\to\infty$ for any continuous function $f$ with compact support on $\rd$; see e.g. \cite{kenig1993neumann}. However, we do not seem to find a proof written for the convergence \eqref{eq em_limit}, and so we give the proof in Section~\ref{sec emlimit_pf}.

\subsection{We neutralize the coefficients of $A_j$ away from a band} \label{S33}
Now we continue with a given $j \geq 0$ and construct a matrix $A_j^*$ that coincides with $A_j$ on the large band $\big\{ (x,t) \, ; \, \frac{\rho_j}{2}<t\le 2M_j \big\}$
and with the identity matrix outside of a much larger band (so that further gluing will be easier). The main point of the construction is to do it without increasing the norm $\cN(A_j)$ too much.

\begin{defn}[The matrix $A_j^*$]\label{def Aj*}
Let $A_j$ be the matrix that we fixed earlier. Let $\rho_j$ and $M_j$ be as in 
Lemma \ref{lem strip}. Set $N_j=\epsilon_j^{-1}=2^j$ and denote by $I$ is the identity 
matrix of size $d+1$. We define $A_j^*$ by
\[
A_j^*(x,t):=\begin{cases}
I, \qquad t>2^{N_j}M_j  \,\text{ or  } t\le 2^{-N_j}\rho_j,\\
\frac{l}{N_j}I+\br{1-\frac{l}{N_j}}A_j(x,t), \qquad 2^l M_j<t\le 2^{l+1}M_j, \text{ for } 1\le l\le N_j-1,\\
A_j(x,t), \qquad \frac{\rho_j}{2}<t\le 2M_j,\\
\frac{l}{N_j}I+\br{1-\frac{l}{N_j}}A_j(x,t), \qquad 2^{-l-1} \rho_j<t\le 2^{-l}\rho_j, \text{  for } 1\le l\le N_j-1.
\end{cases}
\]
\end{defn}

Notice that we only care about the values of $t$ here, and we do not modify $A_j(x,t)$ 
for $|x|$ large, because this is not needed for the construction. 
We need a very large interval around $[\frac{\rho_j}{2}, 2M_j]$ because we want the coefficients
to vary slowly, so that some $\ell^2$-norm is small. We now evaluate the $\alpha$-numbers for $A_j^*$.

\begin{lem}\label{lem alphaAj*}
For any $x\in\rd$, 
\[
\alpha_{A_j^*}(x,r)
\begin{cases}
=0, \qquad r\le 2^{-N_j}\rho_j \text{ or } r\ge 2^{N_j+1}M_j,\\
\le C\epsilon_j, \qquad 2^{-N_j}\rho_j\le r\le \rho_j\text{ or } 2M_j\le r\le 2^{N_j+1}M_j, \\
=\alpha_{A_j}(x,r), \qquad \rho_j<r\le 2M_j.
\end{cases}\]
Here, the constant $C$ depends only on $d$ and $\mu_0$.
\end{lem}

\begin{proof}
Let us discuss four cases. 

\noindent \textbf{Case 1.} $r\le 2^{-N_j}\rho_j \text{ or } r\ge 2^{N_j+1}M_j$. 

In this case, observe that for any $x\in\rd$, 
\[
W(x,r)=\Delta(x,r)\times(r/2,r]\subset \rd\times \set{(0, 2^{-N_j\rho_j}] \cup ( 2^{N_j}M_j,\infty)}.
\]
By the definition of $A_j^*$, $A_j^*(x,t)=I$ for $(x,t)\in \rd\times \set{(0, 2^{-N_j\rho_j}] \cup (2^{N_j}M_j,\infty)}$. So we can just take $A_0=I$ in the definition of $\alpha_{A_j^*}$ and get that
\[
\alpha_{A_j^*}(x,r)=0 \qquad\text{for }x\in\rd, \, r\le 2^{-N_j}\rho_j \text{ or } r\ge 2^{N_j+1}M_j.
\]

\noindent\textbf{Case 2.} $2^{-N_j}\rho_j\le r\le \rho_j$. 

Let $k$ be such that $2^{-k-1}\rho_j<r\le 2^{-k}\rho_j$; then $0\le k\le N_j-1$ and for 
$x\in\rd$,
\begin{multline*}
    \alpha_{A_j^*}(x,r)^2\\
    \le \frac{1}{\abs{W(x,r)}}\int_{\Delta(x,r)}\int_{r/2}^{2^{-k-1}\rho_j}\abs{\frac{k+1}{N_j}I+\br{1-\frac{k+1}{N_j}}A_j(y,t)-A_0}^2dtdy\\
    +\frac{1}{\abs{W(x,r)}}\int_{\Delta(x,r)}\int_{2^{-k-1}\rho_j}^r\abs{\frac{k}{N_j}I+\br{1-\frac{k}{N_j}}A_j(y,t)-A_0}^2dtdy\\
    =: I_1+I_2,
\end{multline*}
where we use the constant matrix
\begin{equation}\label{eq A0}
   A_0=\frac{k}{N_j}I+\br{1-\frac{k}{N_j}}\fiint_{T(x,2r)}A_j(Z)dZ.
\end{equation}
Inserting the definition of $A_0$ into the integrals, we obtain that
\begin{multline*}
    I_1\le \frac{1}{\abs{W(x,r)}}\int_{\Delta(x,r)}\int_{r/2}^{2^{-k-1}\rho_j}
    \br{\frac{1}{N_j}}^2\br{\abs{I}+\abs{A_j(y,t)}^2}dtdy\\
    +\frac{1}{\abs{W(x,r)}}\int_{\Delta(x,r)}\int_{r/2}^{2^{-k-1}\rho_j}\br{1-\frac{k}{N_j}}^2\Big|A_j(y,t)-\fiint_{T(x,2r)}A_j(Z)dZ\Big|^2dtdy\\
    \le CN_j^{-2}+\frac{1}{\abs{W(x,r)}}\int_{\Delta(x,r)}\int_{r/2}^{2^{-k-1}\rho_j}\Big|A_j(y,t)-\fiint_{T(x,2r)}A_j(Z)dZ\Big|^2dtdy\\
   \leq  C\epsilon_j^{2}+C \fiint_{Y \in T(x,2r)} \Big| A_j(Y)-\fiint_{T(x,2r)}A_j(Z)dZ\Big|^2dtdy.
\end{multline*} 
At this point we find it convenient to use an observation of \cite{david2021carleson}.
Set, for any elliptic matrix-valued function $A$ and $(x,r) \in \reu$,
\begin{equation} \label{def gamma}
\gamma_A(x,r) 
= \inf_{A_0 \in \cA_0(\mu_0)}\bigg\{\fiint_{(y,s) \in T(x,r)} |A(y,s) - A_0|^2 \bigg\}^{1/2}.
\end{equation}
Notice that we integrate on $T(x,r)$, while in the definition \eqref{1a5}
of $\alpha_A(x,r)$, we only integrated on the Whitney box $W(x,r)$, so it seems that 
$\gamma_A(x,r)$ should be harder to control. Yet we have the following.

\begin{lemma}[\cite{david2021carleson}, Remark 4.22]\label{lem gammabdalpha}
Suppose that $A$ is a $\mu_0$-elliptic matrix-valued function defined on $\mathbb{R}^{n+1}_+$. Then for every surface ball $\Delta_0$,
\[\left\|\gamma_A(x,r)^2 \, \frac{dx \, dr}{r} \right\|_{\mathcal{C}(\Delta_0)} \le C \left\|\alpha_A(x,r)^2 \, \frac{dx \, dr}{r} \right\|_{\mathcal{C}(3\Delta_0)},\]
and
\begin{equation} \label{eq gammabdd}
    \gamma_A(x,r)^2 \le C \left\|\alpha_A(x,r)^2 \, \frac{dx \, dr}{r} \right\|_{\mathcal{C}(3\Delta_0)} \text{for all } (x,r) \in T_{\Delta_0},
\end{equation}
where $C$ only depends on dimension.
\end{lemma}

\ms
Return to $I_1$; the estimate above yields 
$I_1 \leq  C\epsilon_j^{2}+C  \gamma_{A_j}(x,2r)^2$, and now 
\eqref{eq gammabdd} implies that $I_1 \leq C\epsilon_j^{2}+C \cN(A_j) \leq C\epsilon_j^2$
(by \eqref{AjDKP_epsj}).

The estimate for $I_2$ is easier, as most of the terms in the integral are cancelled. Again by \eqref{eq gammabdd}, \begin{multline*}
    I_2\le \frac{1}{\abs{W(x,r)}}\int_{\Delta(x,r)}\int_{2^{-k-1}\rho_j}^r 
  \bigg|  A_j(y,t)-\fiint_{T(x,2r)}A_j(Z)dZ \bigg|^2dtdy\\
    \le C\gamma_{A_j}(x,2r)^2\le C\epsilon_j^2.
\end{multline*}
So we have proved that $\alpha_{A_j^*}(x,r)^2\le C\epsilon_j^2$ for any 
$x\in\rd$ when 
$2^{-N_j}\rho_j\le r\le \rho_j$.

\medskip

\noindent\textbf{Case 3.} $2M_j\le r\le 2^{N_j+1}M_j$.

In this case, $\alpha_{A_j^*}(x,r)$ can be estimated as in Case 2. 
More precisely, write $2^kM_j<r\le 2^{k+1}M_j$ for some $1\le k\le N_j$. Then for any $x\in\rd$,
\begin{multline*}
    \alpha_{A_j^*}(x,r)^2\\
    \le \frac{1}{\abs{W(x,r)}}\int_{\Delta(x,r)}\int_{r/2}^{2^kM_j}\abs{\frac{k-1}{N_j}I+\br{1-\frac{k-1}{N_j}}A_j(y,t)-A_0}^2dtdy\\
    +\frac{1}{\abs{W(x,r)}}\int_{\Delta(x,r)}\int_{2^{k}M_j}^r\abs{\frac{k}{N_j}I+\br{1-\frac{k}{N_j}}A_j(y,t)-A_0}^2dtdy,
\end{multline*}
where $A_0$ is the same constant matrix as in \eqref{eq A0}.
A straightforward computation shows that the first term on the right-hand side is bounded by 
\[
\frac{C}{N_j^2}+C\gamma_{A_j}(x,2r)^2\le C\epsilon_j^2.
\]
The second term is bounded by $C\gamma_{A_j}(x,2r)^2$, and thus also bounded by $C\epsilon_j^2$.

\medskip 

\noindent \textbf{Case 4.} $\rho_j<r\le 2M_j$.

One only needs to observe that in this case, 
\[
W(x,r)\subset\rd\times\Big[\frac{\rho_j}{2},2M_j\Big] \qquad\text{for any }x\in\rd.
\]
But $A_j^*(x,t)=A_j(x,t)$ for $t\in\Big[\frac{\rho_j}{2},2M_j\Big]$. So $\alpha_{A_j^*}(x,r)=\alpha_{A_j}(x,r)$. 
\end{proof}

\subsection{We glue the $A_j^\ast$ on top of each other}

Now we are ready to a function $A^*$ that contains a dilated copy of each $A_j^*$
and contradicts Lemma \ref{BTZthm}.

\begin{defn}[The matrix $A^*$]\label{def A*} 
Let $\rho_j$ and $M_j$ be as in Lemma \ref{lem strip}. 
Let $A_j^*$ be as in Definition \ref{def Aj*}, $j\in\ZZ_+$. 
Choose the $\lambda_j$, $j \geq 1$, by induction, so that $\lambda_1=1$ and 
 \begin{equation}\label{def lambdaj}
     2\lambda_{j+1}2^{N_{j+1}}M_{j+1}< \lambda_j 2^{-N_j}\rho_j,
 \end{equation}
 and define $A^*$ by
 \begin{equation} \label{3b31}
 A^*(x,t)=\begin{cases}
 A_j^*\br{\frac{x}{\lambda_j},\frac{t}{\lambda_j}}, \qquad \lambda_j 2^{-N_j}\rho_j\le t\le \lambda_j 2^{N_j}M_j,\quad j\ge 1,\\
 I\qquad \text{elsewhere}.
 \end{cases}
\end{equation}
 Finally let $L_*:=-\divg A^*\nabla$.
\end{defn}

\begin{rem}
The different regions $\big\{ \lambda_j 2^{-N_j}\rho_j\le t\le\lambda_j 2^{N_j}M_j \big\}$ 
in \eqref{3b31} are disjoint by \eqref{def lambdaj};
furthermore, the gluing is smooth, because since $A_j^*(x,t)=I$ for $t>2^{N_j}M_j$ 
or $t\le 2^{-N_j}\rho_j$, we actually have that
\begin{equation}\label{eq A*equal}
    A^*(x,t)=A_j^*\br{\frac{x}{\lambda_j},\frac{t}{\lambda_j}} \qquad\text{for }t\in \Big[\lambda_j2^{-N_j-1}\rho_j, \lambda_j2^{N_j+1}M_j\Big], \quad j\ge1.
\end{equation}
Also, by tracking the region where $A^*=I$, one sees that for any $x\in\rd$,
\begin{multline}\label{eq alpha=0}
   \alpha_{A^*}(x,r)=\alpha_I(x,r)=0\\
   \text{when } \lambda_{j+1}2^{N_{j+1}+1}M_{j+1}\le r\le\lambda_j2^{-N_j}\rho_j 
   \quad \text{for some }j\ge1\quad
   \text{or } r\ge 2^{N_1+1}M_1.
\end{multline}
\end{rem}

\begin{lem}
The matrix $A^*$ satisfies the vanishing weak DKP condition. 
\end{lem}

\begin{proof}
Set
\begin{equation}\label{3b36}
\widehat\alpha(r_0) = \sup_{x_0\in\rd}\frac{1}{\abs{\Delta(x_0,r_0)}}
\iint_{T(x_0,r_0)}\alpha_{A^*}(x,r)^2\frac{dxdr}{r}
\end{equation}
for $r_0 > 0$. 
It suffices to show that 
\begin{equation}\label{eq A*vanishing}
\widehat\alpha \text{ is bounded and }
  \lim_{r_0\to 0^+} \widehat\alpha(r_0) = 0.
\end{equation}
Fix any $x_0\in\rd$ and $r_0>0$. Since $\alpha_A^*(x,r)=0$ for $r\ge 2^{N_1+1}M_1$ (see \eqref{eq alpha=0}), 
we can assume that $r_0\le 2^{N_1+1}M_1$. Let us first assume that 
\begin{equation}\label{r0case1}
    \lambda_k 2^{N_k+1}M_k\le r_0\le\lambda_{k-1}2^{-N_{k-1}}\rho_{k-1} \qquad\text{for some }k\in\ZZ, \,k\ge 2.
\end{equation}
We claim that 
\begin{equation}\label{eq A*DKP}
    \iint_{T(x_0,r_0)}\alpha_{A^*}(x,r)^2\frac{dxdr}{r}\le C 2^{-k}\abs{\Delta(x_0,r_0)},
\end{equation}
where the constant $C$ depends only on $d$.
To see this, we write
\begin{multline}\label{eq A*DKP split}
    \iint_{T(x_0,r_0)}\alpha_{A^*}(x,r)^2\frac{dxdr}{r}\le \sum_{j=k}^\infty\bigg\{\int_{\Delta(x_0,r_0)}\int_{\lambda_j2^{-N_j}\rho_j}^{\lambda_j\rho_j}\alpha_{A^*}(x,r)^2\frac{drdx}{r}\\
    +\int_{\Delta(x_0,r_0)}\int_{\lambda_j\rho_j}^{2\lambda_jM_j}\alpha_{A^*}(x,r)^2\frac{drdx}{r}+\int_{\Delta(x_0,r_0)}\int_{2\lambda_jM_j}^{\lambda_j2^{N_j+1}M_j}\alpha_{A^*}(x,r)^2\frac{drdx}{r}\bigg\}\\
    =:\sum_{j=k}^\infty\{J_1+J_2+J_3\},
\end{multline}
noticing \eqref{eq alpha=0}. 
Set \[\wt A_j(x,t):=A_j^*\br{\frac{x}{\lambda_j},\frac{t}{\lambda_j}}.\]
Then by \eqref{eq A*equal}, 
\begin{equation}\label{eq A*=wtAj}
    A^*(x,t)=\wt A_j(x,t) \qquad\text{for }t\in \Big[\lambda_j2^{-N_j-1}\rho_j, \lambda_j2^{N_j+1}M_j\Big], \quad j\ge1.
\end{equation}
Notice that \eqref{rsc_alpha} asserts that 
\begin{equation}\label{eq wtAj alpha}
    \alpha_{\wt A_j}(x,r)=\alpha_{A_j^*}\br{\frac{x}{\lambda_j},\frac{r}{\lambda_j}} \qquad\text{for any }x\in\rd,\, r>0.
\end{equation}
For any $x\in\rd$ and $\lambda_j2^{-N_j}\rho_j\le r\le \lambda_j2^{N_j+1}M_j$, the Whitney cube $W(x,r)\subset\rd\times\Big[\lambda_j2^{-N_j-1}\rho_j, \lambda_j2^{N_j+1}M_j\Big]$. So \eqref{eq A*=wtAj} and \eqref{eq wtAj alpha} imply that
\[
\alpha_{A^*}(x,r)=\alpha_{A_j^*}\br{\frac{x}{\lambda_j},\frac{r}{\lambda_j}} \qquad\text{for }x\in\rd,\, \lambda_j2^{-N_j}\rho_j\le r\le \lambda_j2^{N_j+1}M_j,\quad j\ge1.
\]
By Lemma \ref{lem alphaAj*}, for any $x\in\rd$,
\begin{equation}\label{eq alphaA*}
    \alpha_{A^*}(x,r)\begin{cases}
    \le C \epsilon_j, 
    \qquad \lambda_j2^{-N_j}\rho_j\le r\le \lambda_j\rho_j\text{ or } 2\lambda_jM_j\le r\le \lambda_j2^{N_j+1}M_j, \\
=\alpha_{A_j}\br{\frac{x}{\lambda_j},\frac{r}{\lambda_j}}, \qquad \rho_j<r\le 2M_j.
    \end{cases}
\end{equation}
Now we are ready to estimate \eqref{eq A*DKP split}. First, $J_1$ and $J_3$ can be estimated similarly. By \eqref{eq alphaA*}, 
\[
J_1\le C 
\int_{\Delta(x_0,r_0)}\epsilon_j^2N_jdx\le C\epsilon_j\abs{\Delta(x_0,r_0)},
\]
and $J_3\le C\epsilon_j\abs{\Delta(x_0,r_0)}$. For $J_2$, by a change of variables, 
\[
    J_2=\int_{\Delta(x_0,r_0)}\int_{\lambda_j\rho_j}^{2\lambda_jM_j}\alpha_{A_j}\br{\frac{x}{\lambda_j},\frac{r}{\lambda_j}}^2\frac{drdx}{r}
    =\lambda_j^d\int_{\Delta\br{\frac{x_0}{\lambda_j},\frac{r_0}{\lambda_j}}}\int_{\rho_j}^{2M_j}\alpha_{A_j}(y,s)^2\frac{dsdy}{s}.
\]
Recall that $j\ge k$. Our definition of $\lambda_j$'s and our assumption \eqref{r0case1} on $r_0$ imply that  $\lambda_j2^{N_j}M_j\le\lambda_k2^{N_k}M_k\le r_0$ for all $j\ge k$. Hence, $M_j\le 2^{-N_j}\frac{r_0}{\lambda_j}\le \frac{r_0}{\lambda_j}$ for all $j\ge k$, which implies that 
\[
J_2\le \lambda_j^d\int_{\Delta\br{\frac{x_0}{\lambda_j},\frac{r_0}{\lambda_j}}}\int_{0}^{\frac{2r_0}{\lambda_j}}\alpha_{A_j}(y,s)^2\frac{dsdy}{s}.
\]
Recall that $\cN(A_j)\le \epsilon^2_j$, so 
\[
J_2\le\lambda_j^d\epsilon_j^2\abs{\Delta\br{\frac{x_0}{\lambda_j},\frac{4r_0}{\lambda_j}}}\le C\epsilon_j^2\abs{\Delta(x_0,r_0)}.
\]
Combining these estimates with \eqref{eq A*DKP split}, we obtain that 
\begin{multline*}
    \iint_{T(x_0,r_0)}\alpha_{A^*}(x,r)^2\frac{dxdr}{r}\le C\sum_{j\ge k}\br{\epsilon_j+\epsilon_j^2}\abs{\Delta(x_0,r_0)}\le C\abs{\Delta(x_0,r_0)}\sum_{j\ge k}2^{-j}\\
    =C2^{-k}\abs{\Delta(x_0,r_0)},
\end{multline*}
as desired.

Next consider the case when    
\begin{equation}\label{r0case2}
    \lambda_k 2^{-N_k}\rho_k\le r_0\le\lambda_{k}2^{N_k+1}M_k 
    \qquad\text{for some } k > 0. 
\end{equation}
We claim that \eqref{eq A*DKP} also holds in this case. Again, splitting the left-hand side of \eqref{eq A*DKP}, we get that
\begin{multline*}
    \iint_{T(x_0,r_0)}\alpha_{A^*}(x,r)^2\frac{dxdr}{r}\le \sum_{j={k+1}}^\infty\int_{\Delta(x_0,r_0)}\int_{\lambda_j2^{-N_j}\rho_j}^{\lambda_j2^{N_j+1}M_j}\alpha_{A^*}(x,r)^2\frac{drdx}{r}\\
    +\int_{\Delta(x_0,r_0)}\int_{\lambda_k2^{-N_k}\rho_k}^{r_0}\alpha_{A^*}(x,r)^2\frac{drdx}{r}=:\br{\sum_{j=k+1}^\infty J_1'}+J_2'.
\end{multline*}
Observe that
\[
\sum_{j=k+1}^\infty J_1'=\int_{\Delta(x_0,r_0)}\int_0^{\lambda_{k+1}2^{N_{k+1}+1}M_{k+1}}\alpha_{A^*}(x,r)^2\frac{drdx}{r}.
\]
So the previous case applies, and \[\sum_{j=k+1}^\infty J_1'\le C2^{-k-1}\abs{\Delta(x_0,r_0)}.\] 
Turning to $J_2'$, one can discuss 3 situations:
\begin{enumerate}
    \item $\lambda_k 2^{-N_k}\rho_k\le r_0\le\lambda_k\rho_k$;
    \item $\lambda_k \rho_k< r_0\le2\lambda_kM_k$;
    \item $2\lambda_k M_k< r_0\le \lambda_k 2^{N_k+1}M_k$.
\end{enumerate}
Using our estimates \eqref{eq alphaA*} for $\alpha_{A^*}(x,r)$, one can show as in the previous case that in all situations, we have that $J_2\le C2^{-k}\abs{\Delta(x_0,r_0)}$. Therefore, \eqref{eq A*DKP} holds for any $r_0$ that satisfies \eqref{r0case1} or \eqref{r0case2}. Since $x_0\in\rd$ is arbitrary, this implies that $\cN(A^*)\le C$. Moreover, since 
\(
\lambda_{k-1}2^{-N_{k-1}}\rho_{k-1}\to 0\) and $\lambda_{k}2^{N_k+1}M_k\to 0
$ as $k\to\infty$, \eqref{eq A*DKP} implies \eqref{eq A*vanishing}, as desired.
\end{proof}

\begin{lem}
Let $A_j$, $L_j=-\divg{A_j\nabla}$, and $E_j\subset Q_0$ be as fixed before, such that \eqref{AjDKP_epsj} and \eqref{bigAinfty} hold. 
Let the $\lambda_j$,  $A^*$, and $L_*$ be 
as in Definition \ref{def A*}. Then
\begin{equation}\label{eq bigAinftyL*}
 \abs{\frac{\omega^{\infty}_{L_*}(\lambda_jE_j)}{\omega_{L_*}^{\infty}(\lambda_j\Delta_0)}
 -\frac{\sigma(\lambda_jE_j)}{\sigma(\lambda_j\Delta_0)}}\ge\frac{\delta_0}{2}, 
 \qquad j \geq 1. 
\end{equation}
\end{lem}

\begin{proof}
Recall that $\sigma$ is the Lebesgue measure on $\R^d$.
This lemma will be an immediate consequence of Lemma \ref{lem strip} and the 
rescaling property of the elliptic measure with pole at infinity. 
In fact, for any fixed $j \geq 1$, set 
\[
\wt A(x,t):=A^*(\lambda_jx,\lambda_jt)=A^* \big(T_{\lambda_j}^0(X)\big)
\]
where we use the notation \ref{def Tlz} for the linear transformation, and $X=(x,t)$.
Then for this fixed $j$, the definitions of $A^*$ and $A_j^*$ assert that
\[
\wt A(x,t)=A_j(x,t) \qquad\text{for }x\in\rd, \, \rho_j\le t\le M_j.
\]
Recall that our first move was to take $\Delta_j = \Delta_0$.
Applying Lemma \ref{lem strip} to $\wt A^*$, 
one gets that
\[
\abs{\frac{\omega^{\infty}_{\wt L}(E_j)}{\omega_{\wt L}^{\infty}(\Delta_0)}
-\frac{\sigma(E_j)}{\sigma(\Delta_0)}}\ge\frac{\delta_0}{2},
\]
where $\wt L=-\divg \wt A\nabla$. By Lemma \ref{lem rsc_eminfty}, 
\[
\frac{\omega_{\wt L}^\infty(E_j)}{\omega_{\wt L}^\infty(\Delta_0)}=\frac{\omega_{L_*}^\infty(\lambda_jE_j)}{\omega_{L_*}^\infty(\lambda_j \Delta_0)}.
\]
So \eqref{eq bigAinftyL*} follows. 
\end{proof}
\ms
To summarize, we showed that if Theorem \ref{t1b1} 
is not true in $\reu$, then we can construct an elliptic operator $L_*$ that satisfies the vanishing weak DKP condition, 
a sequence of numbers $\set{\lambda_j}_{j\in\ZZ_+}$ that approach 0 as $j\to\infty$, and a sequence of Borel sets 
$\set{E_j}_j$ contained in the unit ball $\Delta_0 \subset \R^d$, 
such that \eqref{eq bigAinftyL*} holds.

On the other hand, Lemma \ref{BTZthm} says that $\omega _{\wt L_*}^\infty$
is absolutely continuous with respect to $\sigma$, with a density $k_{L_*}^\infty$
such that $\log k_{L_*}^\infty\in VMO(\rd)$.
The desired contradiction follows at once from the following lemma, 
which will be proved in Section \ref{Sbmo}.

\begin{lem}\label{lem VMOtosmAinfty}
Let $\omega$ be a positive locally finite Borel measure on $\R^d$, suppose that it is absolutely
continuous with  respect to the Lebesgue measure $\sigma$, and denote by 
$k = \frac{d \omega}{d \sigma}$ its Radon-Nikodym density. 
Assume $\log k\in VMO(\rd)$. Then for any $\delta>0$, there exists $\gamma>0$ 
such that for any ball $\Delta = \Delta(x,r) \subset\rd$ with $0< r \le \gamma$, and any 
Borel set $E\subset \Delta$,
\begin{equation} \label{3b48}
\abs{\frac{\omega(E)}{\omega(\Delta)}-\frac{\sigma(E)}{\sigma(\Delta)}}<\delta.
\end{equation}
\end{lem}

This completes our proof of Theorem \ref{t1b1} in $\reu$, modulo 
Lemmas \ref{lem em_limit} and \ref{lem VMOtosmAinfty}.

\section{Small Lipschitz graphs}
\label{SLip}
In this section we use a change of variable to prove Theorem \ref{t1b1} in when 
$\Gamma$ is the graph of a Lipschitz function $\varphi : \R^d \to \R$, with small Lipschitz constant.
More precisely, we assume that 
\begin{equation} \label{4b1}
||\nabla \varphi ||_{\infty} \leq \varepsilon,
\end{equation}
set
\begin{equation} \label{5b2}
\Gamma = \big\{ (x,\varphi(x)) \in \R^{d+1}\, ; \, x\in \R^d \big\},
\end{equation}
and denote by $\Omega$ the  connected component of $\R^{d+1}\sm \Gamma$
that lies above $\Gamma$. We then let $A$ and $L$ be as in Theorem \ref{t1b1},
and we need to show that if $\varepsilon$ is small enough, depending  on $\delta$, $d$, 
and the ellipticity constant $\mu_0$, the associated elliptic measure $\omega_L^\infty$ lies in
$A_\infty(\sigma,\delta)$.

 We now define a mapping from $\reu$ to $\om$ which is the same as the one that is used in \cite{kenig2001dirichlet}, originally due to Dahlberg, Kenig, and Stein.
For $(x,t)\in\reu$, let $\eta_t(x)=t^{-d}\eta(x/t)$, where $\eta$ is a nonnegative radial 
$C^\infty$ function supported in $\set{\abs{x}\le 1/2}$ and such that $\int_{\R^d} \eta = 1$. 
Define a mapping 
\begin{equation}\label{def rho}
    \rho: \reu \to \Omega, \quad \rho(x,t)=(x,c_0t+F(x,t))
\end{equation}
with $F(x,t)=\eta_t*\vp(x)
= \int \eta(z) \vp(x-tz) dz$,
and  $c_0=1+c_d\norm{\nabla\vp}_\infty$, where \begin{equation}\label{def.cd}
    c_d=\int_{\rd}\abs{\eta(y)}\abs{y}dy.
\end{equation}
This choice of $c_0$ guarantees that $\rho$ is a one-to-one  
bi-Lipschitz mapping of $\reu$ onto $\Omega$. In fact, we have
\begin{equation} \label{4a4}
\nabla\rho(x,t)=\begin{pmatrix}
I & \nabla_x F(x,t)\\
0 & c_0+\partial_t F(x,t)
\end{pmatrix},
\end{equation}
and 
 \(   \det\nabla\rho(x,t)= c_0+\partial_t F(x,t)\).
Since 
\begin{equation} \label{4a5}
\abs{\partial_tF(x,t)}=\abs{-\int_{\rd}\eta(z)z\cdot(\nabla\vp)(x-tz)dz}
\le c_d\norm{\nabla\vp}_{L^\infty(\rd)},
\end{equation}
we get that 
\begin{equation}\label{eq detdrho}
    1\le\det\nabla\rho(x,t)\le 1+2c_d\norm{\nabla\vp}_\infty.
\end{equation}
Incidentally, we have that 
\begin{equation}\label{eq gradxF}
    \abs{\nabla_xF(x,t)}=\abs{\int_{\rd}\eta(z)(\nabla\vp)(x-zt)dz}\le c_d'\norm{\nabla\vp}_{L^\infty(\rd)},
\end{equation}
where $c_d'=\int\abs{\eta(y)}dy$.

Another important property of the mapping $\rho$ is that $\abs{\nabla^2\rho(x,t)}^2t\,dxdt$ 
is a Carleson measure on $\reu$. This property has been mentioned in \cite{kenig2001dirichlet}, but no proof 
nor an estimate for the Carleson norm has been given. 
We give the control of the Carleson norm and sketch the easy proof.

\begin{lem}\label{lem 2ndrho.Carl}
For $(x,t)\in\rd\times\Real_+$, the measure $\mu$ defined by the density 
\[
d\mu(x,t)=\abs{\nabla_{x,t}^2F(x,t)}^2t\,dxdt
\]
is a Carleson measure on $\reu$, and 
\[
\norm{\mu}_{\mathcal C}\le C\norm{\nabla\vp}_{BMO(\rd)}^2\le C\norm{\nabla\vp}_{L^\infty(\rd)}^2,
\]
where the constant $C$ depends only $\eta$ and $d$. 
\end{lem}
\begin{proof}[Sketch of Proof] The lemma follows from a Carleson measure characterization of $BMO$ functions 
(see e.g. \cite{grafakos2014modern} Theorem 3.3.8). In fact, all the second derivatives of $F(x,t)$ can be expressed 
in the form of $t^{-1}\theta_t*\nabla\vp(x)$, where $\theta_t(x)=t^{-d}\theta(x/t)$ satisfies $\int_\rd\theta(x)dx=0$, 
$\abs{\theta(x)}\le C(1+\abs{x})^{-n-\delta}$ for some $C$, $\delta\in(0,\infty)$, and 
\[
\sup_{\xi\in\rd}\int_0^\infty\abs{\hat\theta(s\xi)\frac{ds}{s}}<\infty.
\]
Therefore, 
\[
\norm{\abs{\theta_t*\nabla\vp(x)}^2\frac{dxdt}{t}}_{\mathcal{C}}\le C\norm{\nabla\vp}_{BMO(\rd)}^2,
\]
which gives the desired estimate.
\end{proof}

One can check 
(see for instance, in a slightly different context, (1.42) and its proof in Lemma 6.17 in \cite{david2019dahlberg})
that if $u$ is a solution of $Lu=-\divg{A\nabla u}=0$ in $\Omega$, then $v=u\circ\rho$ is a solution of $\wt Lv=-\divg{\wt A\nabla v}=0$ in $\reu$, where 
\begin{equation}\label{def wtA}
  \wt A(x,t) = \det(\nabla\rho(x,t)) (\nabla\rho(x,t)^{-1})^T A(\rho(x,t)) \nabla\rho(x,t)^{-1}.  
\end{equation}
It will be useful to write out the matrix $\wt A(x,t)$. A simple computation shows that
\begin{multline}\label{eq wtA}
    \wt A(x,t)=\begin{pmatrix}
    I & \mathbf{0}\\
    -\nabla_xF(x,t) & 1
    \end{pmatrix}
    A(\rho(x,t))
    \begin{pmatrix}
    I & -\frac{\nabla_xF(x,t)}{c_0+\d_tF(x,t)}\\
    0 & (c_0+\d_t F(x,t))^{-1}
    \end{pmatrix}\\
    =: P(x,t)A(\rho(x,t)) Q(x,t).
\end{multline}

\begin{lem}\label{lem cNwtA} If $\norm{\nabla\vp}_\infty<\epsilon<\frac{1}{100(1+c_d)}$, where $c_d$ is as in \eqref{def.cd}, then for $\wt A$ defined in \eqref{def wtA},
\[
\cN(\wt A)\le C\norm{\nabla\vp}_\infty^2 + C\cN(A),
\]
where $C$ depends on $d$ and the ellipticity constant of $A$. 
\end{lem}

\begin{proof}
For $x\in\rd$ and $r>0$, let us use the notation 
$W(x,r)=\Delta_{\rd}(x,r)\times(5r/6,r)$ for Whitney regions in $\reu$, where $\Delta_{\rd}(x,r)=\set{y\in\rd: \abs{y-x}<r}$. Notice that they are slightly smaller than the Whitney regions defined in \eqref{TT}, 
but as we said earlier, this change is insignificant in terms of Theorems \ref{t1b1} and \ref{t1b2} in $\reu$. 
We denote by $W_\Omega(x_\rho,r)$ the Whitney  
region $\set{Y\in \om\cap B(x_\rho,r):\dist(Y,\Gamma)>3r/10}$ in $\Omega$, where $x_\rho=(x,\vp(x))\in\Gamma$ for $x\in\rd$. Notice that we replace the constant $1/2$ in \eqref{1b2} with $3/10$, but we know that this does not matter. 
 It will be convenient to work with the following different Whitney regions in $\Omega$ : for $x\in\rd$ and $r>0$, set
\[
W_\rho(x,r):=\rho(W(x,r)).
\]
Since $1\le \d_t(c_0 t+F(y,t))\le 1+2c_d\norm{\nabla\vp}_\infty$ for all $(y,t)\in\reu$, 
\begin{equation}\label{eq Lipdist}
    t\le c_0t+F(y,t)-\vp(x)\le (1+2c_d\eps)t \quad \text{for }(y,t)\in\reu.
\end{equation}
Let $(y,s) = W_\rho(x,r)$ be given, and write $(y,s) = \rho(y,t) = (y, c_0 t + F(y,t))$ for some 
$(y,t) \in W(x,r)$. Thus $|y-x| < r$ and $5r/6 < t < r$. By \eqref{eq Lipdist}, $5r/6<s-\vp(y)\le (1+2c_d\eps)r$. This means that $\abs{(y,s)-y_\rho}>5r/6$, and hence $\dist((y,s),\Gamma)\ge 3r/4$ if we choose $\eps$ sufficiently small.
In addition, \[|(y,s)-x_\rho| \le |y-x|+|s-\varphi(y)|+\abs{\vp(y)-\vp(x)}\le \br{2+(2c_d+1)\eps}r
\leq 5r/2,\]
if we choose $(2c_d+1)\eps<1/2$. So 
\begin{equation} \label{eq Wrho_ss}
 W_\rho(x,r)\subset \set{Y\in\Omega\cap B(x_\rho,5r/2): \dist(Y,\Gamma)>3r/4)}
 = W_\Omega(x_\rho,5r/2).
\end{equation}

Now we fix $x\in\rd$ and $r>0$, and estimate $\alpha_{\wt A}(x,r)$. 
Let $A_0$ be a constant coefficient matrix which achieves the infimum for $\alpha_A(x_\rho,\frac{3r}{2})$,
 set 
 $$a = a(x,r) = \fiint_{W(x,r)}\nabla_y F(y,s)dyds$$ and 
$$b = b(x,r) = \fiint_{W(x,r)}\d_sF(y,s)dyds,$$ and define 
$\wt A_0$ to be the constant coefficient matrix
\begin{equation*}
    \wt A_0:=
    \begin{pmatrix}
 I & \mathbf{0}\\
    -a & 1
\end{pmatrix}
    A_0\begin{pmatrix}
 I & -a \br{c_0+b}^{-1}\\
    0 & \br{c_0+b}^{-1}
\end{pmatrix}
    =:P_0 A_0 Q_0, 
\end{equation*}
which should be compared to 
the expression of $\wt A(x,t)$ in
 \eqref{eq wtA}. Then 
\begin{multline*}
    \alpha_{\wt A}(x,r)^2\le\fiint_{W(x,r)}\abs{\wt A(y,s)-\wt A_0}^2dyds\\
    =\fiint_{W(x,r)}\abs{P(y,s)A(\rho(y,s))Q(y,s)-P_0 A_0 Q_0}^2dyds.
\end{multline*}
By the triangle inequality, 
\begin{multline*}
    \abs{P(y,s)A(\rho(y,s))Q(y,s)-P_0 A_0 Q_0}
    \le     \abs{(P(y,s)-P_0)A(\rho(y,s))Q(y,s)}\\
    +\abs{P_0A(\rho(y,s))(Q(y,s)-Q_0)} + \abs{P_0\br{A(\rho(y,s))-A_0}Q_0}.
\end{multline*}
Using 
the definitions of $P$, $Q$, $P_0$, $Q_0$, the ellipticity of $A$ and $A_0$, and 
\eqref{eq detdrho} \eqref{eq gradxF}, we obtain that
\begin{multline*}
    \abs{P(y,s)A(\rho(y,s))Q(y,s)-P_0 A_0 Q_0}
    \lesssim \big|\nabla_yF(y,s)- a 
    \big|\\+\big|\d_s F(y,s)- b 
    \big|+ 
    \abs{A(\rho(y,s))-A_0}.
\end{multline*}
Then
\begin{multline}\label{eq alphawtA}
    \alpha_{\wt A}(x,r)^2\\
    \lesssim \fiint_{W(x,r)}\bigg|\nabla F(y,s)-\fiint_{W(x,r)}\nabla F(z,\tau)dzd\tau\bigg|^2dyds + \fiint_{W(x,r)}\abs{A(\rho(y,s))-A_0}^2dyds\\
    \lesssim r^2\fiint_{W(x,r)}\abs{\nabla^2F(y,s)}^2dyds + \fiint_{\rho(W(x,r))}\abs{A(Y)-A_0}^2dY.
\end{multline}
We fix any $x_0\in\rd$ and $r_0>0$, and estimate 
$\iint_{T(x_0,r_0)}\alpha_{\wt A}(x,r)^2dxdr/r$. 
By Fubini's theorem 
(recall also that $r \sim s$ when $(y,s) \in W(x,r)$), 
\[
\iint_{T(x_0,r_0)}r^2\fint_{W(x,r)}\abs{\nabla^2F(y,s)}^2dyds\frac{dxdr}{r}
\le C\int_{\Delta_{\rd}(x_0,2r_0)}\int_0^{r_0}\abs{\nabla^2F(y,s)}^2s\,dyds,
\]
which is bounded by $C\norm{\nabla\vp}_\infty^2r_0^d$ thanks to Lemma \ref{lem 2ndrho.Carl}. 
By \eqref{eq Wrho_ss} and our choice of $A_0$, 
\begin{multline*} 
    \iint_{(x,r) \in T(x_0,r_0)}\fiint_{Y \in \rho(W(x,r))}\abs{A(Y)-A_0}^2dY  \frac{dxdr}{r}\\
    \le C\int_{x\in\Delta_{\rd}(x_0,r_0)}\int_{r=0}^{r_0}
   \fint_{Y \in  W_{\Omega}(x_\rho,5r/2)}
      \abs{A(Y)-A_0}^2dY\frac{dxdr}{r}\\ 
    \le C\int_{x\in\Delta_{\rd}(x_0,r_0)}\int_{r=0}^{r_0}
  \alpha_A(x_\rho,5r/2)^2\frac{dxdr}{r},
\end{multline*}
where in fact $\alpha_A(x,r)$ is defined as in \eqref{1a5}, but in terms of the larger 
Whitney boxes $W_{\Omega}(x,r)$; it is easy to see this does not alter the weak DKP condition  \eqref{1a7}.
Now observe that if $x \in \Delta_{\rd}(x_0,r_0)$, then
\[
x_\rho\in \Delta((x_0)_\rho,(1+2\epsilon)r_0)\subset 
\Delta((x_0)_\rho,3r_0/2)
\]
if $\varepsilon$ is small enough, and now, setting $\xi = x_\rho$, 
\begin{multline*}
\int_{x\in\Delta_{\rd}(x_0,r_0)}\int_{r=0}^{r_0}  \alpha_A(x_\rho,5r/2)^2\frac{dxdr}{r}
\\
\leq \int_{\xi\in \Delta((x_0)_\rho,3r_0/2)}\int_{r=0}^{r_0}  \alpha_A(\xi,5r/2)^2\frac{d\sigma(\xi)dr}{r}
\le C\cN(A) r_0^d
\end{multline*}
by the weak DKP condition on the $\alpha_A(\xi,r)$.
Recalling \eqref{eq alphawtA}, we have proved that for any $x_0\in\rd$ and $r_0>0$,
\[
\iint_{T(x_0,r_0)}\alpha_{\wt A}(x,r)^2\frac{dxdr}{r}\le C\br{\norm{\nabla\vp}_\infty^2+\cN(A)}r_0^d,
\]
which gives the desired the estimate for $\cN(\wt A)$.
\end{proof}

Let $\omega^{X_0}$ (or $\omega^\infty$) denote the elliptic measure with pole at $X_0\in\Omega$ (or at infinity) corresponding to the operator $L=-\divg{A\nabla}$ in $\Omega$. Let $\wt{\omega}^{X_0}$ (or $\wt\omega^\infty$) denote the elliptic measure with pole at $X_0\in\reu$ (or at infinity) corresponding to the operator $\wt L=-\divg{\wt A\nabla}$ in $\reu$. A change of variables shows that for any set $E\subset\Gamma$, and any $X_0\in\Omega$,
\begin{equation}\label{eq emcov.X0}
    \omega^{X_0}(E)=\wt\omega^{\rho^{-1}(X_0)}(\rho^{-1}(E)).
\end{equation}
Similarly, one can show that for any $X$, $Y\in\Omega$,
\(\wt G(\rho^{-1}(X),\rho^{-1}(Y))=G(X,Y)\), where $\wt G$ is the Green function for $\wt L$ in $\reu$, 
and $G$ is the Green function for $L$ in $\Omega$. 
This implies, by the definition of the elliptic measure with a 
pole at infinity in Lemma \ref{lem emGreen_infinity}, that for any $E\subset\Gamma$,
\begin{equation}\label{eq emcov.infty}
   \omega^\infty(E)=\wt\omega^\infty(\rho^{-1}(E)).
\end{equation}

Now we are ready to prove Theorems \ref{t1b1} and \ref{t1b2} for $\Gamma$ defined in \eqref{5b2} with small Lipschitz constant. 

\begin{proof}[Proof of Theorem \ref{t1b1} for small Lipschitz graphs] 
Let $\Delta=B\cap\Gamma$ be  
any surface ball on $\Gamma=\set{(x,\vp(x)):\, x\in\rd}$ and let $E\subset \Delta$ be any Borel set. 
We write $B = B(x_\rho,r)$, with $x \in \R^d$ and $x_\rho = (x, \varphi(x))$, and
set $\wt \Delta = B(x, r) \cap \R^d$. Also call $\pi: \Real^{d+1}\to\bdy\reu\cong\rd$ 
the projection onto $\rd$. Observe also that $\pi(\Delta) \subset \wt \Delta$ and, 
due to the definition of $\rho$ and the fact that $\Gamma$ is a graph, 
\begin{equation} \label{4b16}
\rho^{-1}(E)=\pi(E)
\ \text{ and } \ \rho^{-1}(\Delta)=\pi(\Delta).
\end{equation}
Next we compare the Hausdorff measures. Since $\Gamma$ is an $\varepsilon$-Lipschitz graph,
\begin{equation} \label{4b17}
\sigma(\pi(E)) \leq \sigma(E) \leq (1+C\varepsilon) \sigma(\pi(E)),
\end{equation}
and similarly 
\begin{equation} \label{4b17bis}
\sigma(\pi(\Delta)) \leq \sigma(\Delta) \leq (1+C\varepsilon) \sigma(\pi(\Delta)).
\end{equation}
In addition, $\pi(\Delta)$ contains $\R^d \cap B(x, (1-\varepsilon) r)$, so
\begin{equation} \label{4b18}
\sigma(\pi(\Delta)) \geq (1-C\varepsilon) \sigma(\wt \Delta),
\end{equation}
and hence
\begin{equation} \label{4b21}
1-C\varepsilon \leq \frac{\sigma(\pi(\Delta))}{\sigma(\wt\Delta)} \leq 1
\end{equation}
because $\pi(\Delta) \subset \wt\Delta$.

Lemma \ref{lem cNwtA} and Theorem \ref{t1b1} for $\reu$ say that
if $\varepsilon$ and $\cN(A)$ are small enough, 
\begin{equation} \label{4b19}
\abs{\frac{\wt\omega^\infty(\pi(E))}{\wt\omega^\infty(\wt \Delta)}
-   \frac{\sigma(\pi(E))}{\sigma(\wt\Delta)}
}<\frac{\delta}{10}
  \  \text{ and } 
  \abs{\frac{\wt\omega^\infty(\pi(\Delta))}{\wt\omega^\infty(\wt \Delta)}
-   \frac{\sigma(\pi(\Delta))}{\sigma(\wt\Delta)}
}<\frac{\delta}{10}.
\end{equation} 
Because of \eqref{eq emcov.infty} and \eqref{4b16}, this becomes
\begin{equation} \label{4b20}
\abs{\frac{\omega^\infty(E)}{\wt\omega^\infty(\wt \Delta)}
-   \frac{\sigma(\pi(E))}{\sigma(\wt\Delta)}
}<\frac{\delta}{10}
  \  \text{ and } 
  \abs{\frac{\omega^\infty(\Delta)}{\wt\omega^\infty(\wt \Delta)}
-   \frac{\sigma(\pi(\Delta))}{\sigma(\wt\Delta)}
}<\frac{\delta}{10},
\end{equation}
and by \eqref{4b21} the second part yields
\begin{equation} \label{4b22}
1-C\varepsilon-\frac{\delta}{10} 
\leq \frac{\omega^\infty(\Delta)}{\wt\omega^\infty(\wt \Delta)} 
\leq 1 + \frac{\delta}{10}.
\end{equation}
Since by \eqref{4b17} $|\sigma(\pi(E))-\sigma(E)| \leq C \varepsilon \sigma(\pi(E)) 
\leq C \varepsilon \sigma(\wt \Delta)$, we deduce from the first part of \eqref{4b20} that
\begin{equation} \label{4b23}
\abs{\frac{\omega^\infty(E)}{\wt\omega^\infty(\wt \Delta)}
-   \frac{\sigma(E))}{\sigma(\wt\Delta)}
}<\frac{\delta}{10}+C\eps.
\end{equation}
Finally, we may as well assume that we took $\delta < 10^{-1}$, and then we deduce from
\eqref{4b22} and \eqref{4b23} that 
\begin{equation} \label{4b24}
\abs{\frac{\omega^\infty(E)}{\omega^\infty(\Delta)}
-   \frac{\sigma(E))}{\sigma(\wt\Delta)}
}<\frac{\delta}{2}, 
\end{equation}
and since 
\begin{equation} \label{4b25}
\abs{\frac{\sigma(E))}{\sigma(\wt\Delta)} - \frac{\sigma(E))}{\sigma(\Delta)}}
= \frac{\sigma(E))}{\sigma(\Delta)} \, \frac{\sigma(\wt\Delta) - \sigma(\Delta)}{\sigma(\wt\Delta)} 
\leq C\varepsilon
\end{equation}
by \eqref{4b21} and \eqref{4b17bis}, we get that
\begin{equation} \label{4b26}
\abs{\frac{\omega^\infty(E)}{\omega^\infty(\Delta)}
-   \frac{\sigma(E))}{\sigma(\wt\Delta)} }<\delta ,
\end{equation}
as needed. 
\end{proof}

\begin{proof}[Proof of Theorem \ref{t1b2} for small Lipschitz graphs] 
It follows from Theorem \ref{t1b1} for small Lipschitz graph and Lemma \ref{lem polediff}. 
Or we can work directly and apply \eqref{eq emcov.X0} and Theorem \ref{t1b2} for $\reu$.
\end{proof}

\section{Chord-arc surfaces with small constants}
\label{Scassc}

\subsection{Geometric properties of the $CASSC$}

In Definition \ref{d1b12}, we 
defined the $CASSC$ by the fact that they have very big pieces of $\varepsilon$-Lipschitz
graphs. This is convenient for us (and also extends well to higher co-dimensions), 
but we could have used different definitions, 
a popular one being that the unit
normal to $\Gamma$ (say, with values in $\mathbb S / \{ \pm 1 \}$) has a small
BMO norm. We refer to the early papers of S. Semmes \cite{semmes1989criterion,semmes1990differentiable,semmes1990analysis} for details.

Let us first give some geometric definitions.
\begin{defn}[Two-sided Corkscrew condition \cite{jerison1982boundary}]\label{tscs.def}
We say a domain $\Omega \subset \ree$ satisfies the two-sided corkscrew condition if there exists a uniform constant $M \ge 2$ such that for all $x \in \bdy\Omega$ and $r \in (0, \diam \bdy\Omega)$ there exists $X_1,X_2 \in \ree$ such that
\[B(X_1, r/M) \subset B(x,r) \cap \Omega, \quad B(X_2, r/M) \subset B(x,r) \setminus \overline{\Omega}.\]
We write $X_{x,r} := X_1$, for the interior corkscrew point for $x$ at scale $r$.  We also use the notation  $X_\Delta:=X_{x,r}$ when $\Delta=\Delta(x,r)=B(x,r)\cap\bdy\Omega$.
\end{defn}

\begin{defn}[Harnack chain condition \cite{jerison1982boundary} ]
We say a domain $\Omega \subset \ree$ satisfies the Harnack chain condition if there exists a uniform constant $M \ge 2$ such that if $X_1, X_2 \in \Omega$ with $\dist(X_i, \bdy\Omega) > \epsilon > 0$ and $|X_1 - X_2|< 2^k \epsilon$ then there exists a `chain' of open balls $B_1, \dots, B_N$ with $N < M k$ such that $X_1 \in B_1$, $X_2 \in B_N$, $B_j \cap B_{j+1} \neq \emptyset$ for $j = 1, \dots, N-1$ and $M^{-1} \diam B_j \le \dist(B_j, \bdy\Omega) \le M \diam B_j$ for $j = 1, \dots, N$.
\end{defn}

\begin{defn}[NTA domains \cite{jerison1982boundary}]
We say a domain $\Omega \subset \ree$ is an NTA domain if it satisfies the two-sided corkscrew condition and the 
Harnack chain condition. 
\end{defn}

Let $\Gamma \in CASSC(\varepsilon)$ be given. We systematically assume that $\varepsilon$
is small enough, depending on $d$ when needed. We start with simple geometric consequences of the definition the $CASSC$.

\begin{lemma} \label{l5a1}
The set $\Gamma$ is Ahlfors regular and Reifenberg flat. More precisely, there is a constant 
$C \geq 1$, that only depends on $d$, such that for $x\in \Gamma$ and $r > 0$,
\begin{equation} \label{5b2bis}
(1-C\varepsilon^{1/d}) c_d r^d \leq \H^d(\Gamma \cap B(x,r)) 
\leq (1+C\varepsilon) c_d r^d
\end{equation}
and there is a hyperplane $P$ through $x$ such that 
\begin{equation} \label{5b3}
\dist(y,P) \leq C \varepsilon^{1/d} r \text{ for } y \in \Gamma \cap B(x,r)
\text{ and } 
\dist(y,\Gamma) \leq C \varepsilon^{1/d} r \text{ for } y \in P \cap B(x,r).
\end{equation}
\end{lemma}

\begin{proof}
We first prove the upper bound in \eqref{5b2bis}. Let $x\in \Gamma$ and $r > 0$ be given,
and apply the definition to find an $\varepsilon$-Lipschitz graph $G = G_{x,r}$ that meets 
$B(x,r/2)$ and such that \eqref{1b12} holds. Notice that 
$\H^d(G \cap B(x,r)) \leq (1+C\varepsilon) c_d r^d$, and then the right-hand side of 
\eqref{5b2bis}, with a slightly larger $C$, follows. Since $G$ meets $B(x,r/2)$ at
some point $y$, we also have that $\H^d(G \cap B(x,r)) \geq 
\H^d(G \cap B(y,r/2))\ge (1-C\varepsilon) c_d (r/2)^d$ and then, if $\varepsilon$ is small enough,
\begin{equation} \label{5b4}
\H^d(\Gamma \cap B(x,r)) \geq \H^d(G \cap B(x,r)) - \varepsilon r^d \geq c_d (r/3)^d.
\end{equation}
This already gives the Ahlfors regularity, but we want to improve the lower bound
and prove the Reifenberg-flatness. Set $\delta = \dist(x,G)$; then 
$\Gamma \cap B(x,\delta) \subset \Gamma \sm G$, so by \eqref{1b12}
$\H^d(\Gamma \cap B(x,\delta)) \leq \varepsilon r^d$, and \eqref{5b4} (for $(x,\delta)$)
implies that $\delta \leq C \varepsilon^{1/d} r$. We may now revise our proof of 
\eqref{5b4}, because if we pick $y \in G \cap \overline B(x,\delta)$, then
$\H^d(G \cap B(x,r)) \geq \H^d(G \cap B(y, (1-C\varepsilon^{1/d}) r))
\geq (1-C\varepsilon) (1-C\varepsilon^{1/d}) c_d r^d$, and the proof of 
\eqref{5b4} yields the lower bound in \eqref{5b2bis}.

Incidentally, we added the constraint that $G = G_{x,r}$ that meets $B(x,r/2)$
in Definition \ref{d1b12} to avoid the case of a very small $\Gamma \cap B(x,r)$
and a $G_{x,r}$ that does not meet $B(x,r)$. The $\varepsilon^{1/d}$ is a little
ugly, but requiring $G$ to go through $x$ sounded a little too much. And  also we could
use the fact that we deal with $\varepsilon$-Lipschitz graphs to prove \eqref{5b2bis} 
and \eqref{5b3} with the better constants $C\varepsilon$, but we won't care.

Return to $G$. Since both $G$ and $\Gamma$ are Ahlfors-regular, 
\eqref{1b12} now implies that
\begin{multline} \label{5b5}
\dist(y,G) \leq C \varepsilon^{1/d} r \text{ for } y \in \Gamma \cap B(x,r/2)\\
\text{ and } 
\dist(y,\Gamma) \leq C \varepsilon^{1/d} r \text{ for } y \in G \cap B(x,r/2)
\end{multline}
(otherwise we can find a ball of size $C \varepsilon^{1/d} r$ centered on  $\Gamma$
(respectively $G$) that is contained in $B(x,r) \cap \Gamma \sm G$ 
(respectively $B(x,r) \cap G \sm \Gamma$). But we can also find a hyperplane $P$ 
that contains $x$ and which is $C\varepsilon^{1/d} r$-close to $G$ in $B(x,r)$, and then 
\begin{multline} \label{5b6}
\dist(y,P) \leq C \varepsilon^{1/d} r \text{ for } y \in \Gamma \cap B(x,r/3)\\
\text{ and } 
\dist(y,\Gamma) \leq C \varepsilon^{1/d} r \text{ for } y \in P \cap B(x,r/3).
\end{multline}
This is the same as \eqref{5b3}, but for $B(x,r/3)$; the lemma follows.
\end{proof}

\ms
Recall that since $\Gamma$ is Reifenberg flat, it separates
exactly two domains (there is even a 
bi-H\"older mapping of $\R^{d+1}$ that maps $\R^d$ to $\Gamma$), and it is easy to check that
both domains are NTA. Thus 
Lemma \ref{l5a1} allows us to apply the usual estimates for elliptic operators.
We let $\Omega$ be one of these domains.

Similarly, the Ahlfors regularity of $\Gamma$ allow us to define $A_\infty$
(with respect to $\sigma = \H^d_{\vert \Gamma}$, which is doubling)
and $A_\infty(\sigma,\delta)$ (as in Definition \ref{d1b14}).

Finally, we can even construct an analogue of dyadic cubes on $(\Gamma, d\sigma)$,
which is pleasant because it simplifies the theory of $A_\infty$ weights,
and it implies that the John-Nirenberg inequality on the exponential integrability of 
BMO functions holds on $(\Gamma, d\sigma)$ essentially as on $\R^d$.

\ms
We now approximate $\Omega$ by small Lipschitz domains that are contained in $\Omega$.
\begin{lemma} \label{l5a7}
Suppose as above that $\Gamma \in CASSC(\varepsilon)$, with $\varepsilon$ small enough (depending on $d$). 
Set $\eta = \varepsilon ^{\frac{1}{2d}}$. 
Then for each $x\in \Gamma$ and $r > 0$, we can find an $\eta$-Lipschitz graph $G$ such that for one 
of the connected components $U$ of $\R^{d+1} \sm G$, we have
\begin{equation} \label{5b8}
U \cap B(x,r) \subset \Omega 
\end{equation} 
and 
\begin{equation} \label{5b9}
\H^d(\Gamma \cap B(x,r)\sm G) + \H^d(G \cap B(x,r)\sm \Gamma) 
\leq C \varepsilon^{1/2} r^d. 
\end{equation}
As usual, $C$ depends only on $d$.
\end{lemma}

\begin{proof}
Again we make no attempt to get the best constants here.
Let $x\in \Gamma$ and $r > 0$ be given, and let $G_0 = G_{x,2r}$ be the
$\varepsilon$-Lipschitz graph given by Definition~\ref{d1b12}. 
After a rotation if needed, we can assume that $G_0$ is the graph of $A_0 : \R^d \to \R$
(even if $A$ was not initially defined on the whole $\R^d$, we could extend it).

Write $x =(x_0,t_0)$. Notice that $\dist(x,G_0) \leq C \varepsilon^{1/d} r$ 
by \eqref{5b5}, and since $G_0$ is almost horizontal (and if $\varepsilon$ is small enough), 
this means that 
\begin{equation} \label{5b10}
|A_0(x_0)- t_0| \leq C  \varepsilon^{1/d} r.
\end{equation}
Then consider the points $X_\pm = (x_0, t_0 \pm r/2)$; 
since $X_\pm$ is far from $G_0$, \eqref{5b5} says that $X_\pm$ is also far form $\Gamma$; 
without loss of generality, we can assume that 
$X_+ \in \Omega$, because otherwise $X_- \in \Omega$ (recall that $\Gamma$
is Reifenberg flat, so $\Omega$ lies at least on one side) and we could replace 
$U$ below with the lower component of $\R^{d+1} \sm G_0$. 

Set ${\mathcal Z} = \Gamma \cap \overline B(x, r) \sm G_0$; we want to hide
${\mathcal Z}$ below the new graph $G$. The simplest is to take
\begin{equation} \label{5b11}
A(y) = \max\big(A_0(y), \sup_{(z,t) \in {\mathcal Z}} t - \eta |y-z| \big).
\end{equation}
Obviously $A$ is $\eta$-Lipschitz. 
Call $G$ the graph of $A$, and $U$ the component of 
$\R^{d+1} \sm G$ above $G$. 

Our next step is to prove that $U \cap B(x,r) \subset \Omega$, as in \eqref{5b8}. 
First we claim that $U \cap B(x,r)$ does not meet $\Gamma$. 
Clearly it does not meet $G_0$, because $A \geq A_0$, so 
it is enough to exclude points of $B(x,r) \cap \Gamma \sm G_0 \subset {\mathcal Z}$.
But if $(z,t)$ is such a point, then by \eqref{5b11}, 
$A(z) \geq t - \eta |z-z| = t$, so $(z,t) \notin U$; this proves our claim. 

Let us now check that $X_+ =(x_0, t_0 + r/2) \in U$, or in other words
$A(x_0) < t_0 + r/2$. But
$A_0(x_0) \leq t_0 + C  \varepsilon^{1/d} r < t_0 + r/2$ by \eqref{5b10}, 
and for $(z,t) \in {\mathcal Z}$, we have that $t \leq A(x_0) + C \varepsilon^{1/d} r$ because 
$\dist((z,t),G_0) \leq \varepsilon^{1/d} r$ by \eqref{5b5}, and $G_0$
is almost horizontal. Then $t - \eta |y-z| \leq t \leq A(x_0) 
+ C \varepsilon^{1/d} r < t_0 + r/2$, so $X_+ \in U$.
At this point $U \cap B(x, r)$ is a connected set that contains $X_+ \in \Omega$ and that 
does not meet $\Gamma$, and this implies that $U \cap B(x, r) \subset \Omega$,
as needed for \eqref{5b8}.

Finally we estimate the size of the two bad sets in \eqref{5b9}.
For each $Z = (z,t) \in {\mathcal Z}$, consider the ball $B_{Z} = B(Z, r(Z))$,
with $r(Z) = \dist(Z,G_0)/100$.
By the $5$-covering lemma of Vitali, we can find a countable collection 
of balls $B_{Z}$, $Z \in I$, such that the $B_Z$ are disjoint but the 
$5B_Z$ cover $\mZ$. First consider the set
\begin{equation} \label{5b12}
H = \big\{ y \in \R^d \cap B(x_0,r) \, ; \, A(y) \neq A_0(y) \big\}. 
\end{equation}
We want to cover $H$ by multiples of the $B_Z$, so let $y \in H$ be given.
By definition, $A(y) \neq A_0(y)$ and so we can find $Z'  = (z',t') \in {\mathcal Z}$
such that $A_0(y) < t' - \eta |y-z'|$ and hence, since $A_0$ is $\varepsilon$-Lipschitz,
$A_0(z') < t' - \eta |y-z'| + \varepsilon |y-z'| < t' - \eta |y-z'|/2$.
Then we can find $Z = (z,t) \in I$ such that $(z',t') \in 5 B_Z$.
\begin{equation} \label{5b13}
|y-z| \leq |y-z'| + |z'-z| \leq 2\eta^{-1}(t'-A_0(z')) + 5 r(Z).
\end{equation}
Now let $W=(w,A_0(w)) \in G_0$ minimize the distance to $Z$.
Then $|w-z| \leq |W-Z| = 100 r(Z)$, 
so $|A_0(z)-A_0(w)| \leq \varepsilon |z-w| \leq 100 \varepsilon r(Z)$.
Now 
\begin{eqnarray} \label{5b14}
t'-A_0(z') &\leq& |t'-t| + |A_0(z')- A_0(w)| + |A_0(w) - t| 
\nonumber\\
&\leq&  5 r(Z) + \varepsilon |z'- w| + |W - Z| 
\\
&\leq&  5 r(Z) + \varepsilon (|Z'-Z|+|Z-W|) + 100 r(Z)
\leq 106 r(Z),
\nonumber
\end{eqnarray}
so \eqref{5b13} implies that $|y-z| \leq 150 \eta^{-1} r(Z)$, and altogether
\begin{equation} \label{5b15}
H \subset \bigcup_{Z = (z,t) \in I} B(z, 200 r(Z)).
\end{equation}
This yields
\begin{eqnarray} \label{5b16}
\H^d(H) &\leq& C \eta^{-d} \sum_{Z \in I} r(Z)^d \leq C \eta^{-d}\sum_{Z \in I} 
\H^d(\Gamma \cap B_Z)
\\
&\leq& C \eta^{-d} \H^d(\Gamma \cap B(x,2 r) \sm G_0)
\leq C \eta^{-d} \varepsilon  r^d = C \varepsilon^{1/2} r^d
\nonumber
\end{eqnarray}
because the $B_Z$ are disjoint by construction, contained in $B(x,2 r)\sm G_0$
because $r(Z) = \dist(Z,G_0)/100 << r$, by the Ahlfors regularity of $\Gamma$,
by \eqref{1b12}, and because $\eta = \varepsilon^{1/2d}$.
So we control the measure of the difference between the two graphs $G$ and $G_0$, and at
this point \eqref{5b9} follows from \eqref{1b12}.
\end{proof}

As a corollary of Lemma \ref{l5a7}, the inclusion \eqref{5b8} can be reversed. That is, given any ball $B(x,r)$ centered on $\Gamma$, $\Omega$ can be approximated by a Lipschitz domain that contains $\Omega\cap B(x,r)$. We state this result as follows.
\begin{cor}\label{cor U}
Suppose as above that $\Gamma \in CASSC(\varepsilon)$, with $\varepsilon$ small enough (depending on $d$). Set $\eta = \varepsilon ^{\frac{1}{2d}}$. 
Then for each $x\in \Gamma$ and 
$r > 0$, we can find an $\eta$-Lipschitz graph $G$ such that for one of the connected
components $U$ of $\R^{d+1} \sm G$, we have
\begin{equation} \label{5c1}
\Omega \cap B(x,r) \subset U 
\end{equation} 
and 
\begin{equation} \label{5c2}
\H^d(\Gamma \cap B(x,r)\sm G) + \H^d(G \cap B(x,r)\sm \Gamma) 
\leq C \varepsilon^{1/2} r^d. 
\end{equation}
As usual, $C$ depends only on $d$.
\end{cor}

\begin{proof}
Let $G$ be the $\eta$-Lipschitz graph as in Lemma \ref{l5a7}. Then we only need to prove \eqref{5c1} since \eqref{5c2} is the same as \eqref{5b9}. Let the two connected components of $\ree\setminus G$ be $U_\pm$, and the two connected components of $\ree\sm\Gamma$ be $\Omega_\pm$. Then by Lemma \ref{l5a7}, $U_-\cap B(x,r)\subset\Omega_-$. Taking complement, we get that 
$\Omega_+\subset U_+\cup\stcomp{ B(x,r)}$. Therefore, $\Omega_+\cap B(x,r)\subset U_+\cap B(x,r)\subset U_+$, as desired.
\end{proof}

\subsection{Small $A_\infty$ result in the case of $CASSC$}\label{SCASSCmain}

In this section we deduce Theorem \ref{thmCAS} from the case of small Lipschitz graphs, 
through a standard comparison argument. 

Throughout the section, we let $\Gamma\in CASSC(\epsilon)$ with $\epsilon$ small enough, and $\Omega$ be one of the components of $\ree\sm\Gamma$. We shall approximate $\Omega$ from outside by a Lipschitz graph as in Corollary \ref{cor U}.

\begin{thm}\label{thmCAS}
For every choice of $\delta > 0$, there exist 
$\eps_0 = \varepsilon_0(\delta, d, \mu_0) > 0$
with the following properties. Let $0<\eps\le\eps_0$. 
Suppose $\Gamma \in CASSC(\varepsilon)$, and $L=-\divg{A\nabla}$ is an elliptic operator with 
ellipticity $\mu_0$ which satisfies the weak DKP condition with norm $\cN(A)<\varepsilon$. 
 Then for every surface ball $\Delta = B(x,r) \cap \Gamma$ and and every Borel subset $E\subset \Delta$,
\[
  \abs{\frac{\omega_\Omega^{X_0}(E)}{\omega_\Omega^{X_0}(\Delta)}-\frac{\sigma(E)}{\sigma(\Delta)}}<\delta,
\]
whenever
$X_0=X_{x,\eps^{-\beta}r}\in\om$ is a corkscrew point for $x$ at scale $\eps^{-\beta}r$, with $0<\beta\le\beta_0$ for some $\beta_0$ that depends only on $d$ and $\mu_0$.
\end{thm}

\begin{rem}
Although Theorem \ref{thmCAS} is stated differently from Theorem \ref{t1b2}, one can check that the latter can be deduced from the former. In fact, let $\delta>0$, $\kappa>1$ and $X\in\om$ be fixed. Take $\Delta=\Delta(x,r)\subset B(X,\kappa\,\delta(X))$ with $r\le\tau\,\delta(X)$ and $\tau>0$ to be determined, where $\delta(X)=\dist(X,\Gamma)$, and let $E\subset\Delta$. By Theorem \ref{thmCAS}, there exists an $\epsilon_0>0$ such that 
\begin{equation}\label{eqX0delta}
    \abs{\frac{\omega^{X_0}(E)}{\omega^{X_0}(\Delta)}-\frac{\sigma(E)}{\sigma(\Delta)}}<\delta/2
\end{equation} whenever 
$X_0\in\om$ is a corkscrew point for $x$ at scale $\eps_0^{-\beta_0}r$, for some $\beta_0=\beta_0(d,\mu_0)>0$. By a change-of-pole argument similar to Lemma \ref{lem polediff}, one can show that 
\begin{equation}\label{eqXX0delta}
    \abs{\frac{\omega^{X_0}(E)}{\omega^{X_0}(\Delta)}-\frac{\omega^{X}(E)}{\omega^{X}(\Delta)}}\le C\br{\frac{r}{d}}^\gamma \quad \text{for }r<\frac{d}{10}
\end{equation}
for some $\gamma=\gamma(d,\mu_0)>0$, where $d=\min\set{\delta(X),\delta(X_0)}$. Note that we can assume that $d=\delta(X_0)$ by taking $\tau$ sufficiently small (for instance, let $\tau=\epsilon_0^{\beta_0}/2$). Now we take $\epsilon_0$ sufficiently small so that $\epsilon_0^{\beta_0\gamma}<\frac{\delta}{10C}$, then Theorem \ref{t1b2} follows from \eqref{eqX0delta}, \eqref{eqXX0delta}, and the triangle inequality.
Theorem \ref{t1b1} can be deduced from Theorem \ref{t1b2} using Lemma \ref{lem polediff}.
\end{rem}

 We shall need the following two lemmata that enable us 
 to localize the elliptic measures.
\begin{lem}\label{lem loc1}
There exists a constant $M\ge1$ depending on the dimension so that the following holds. Let $x\in\Gamma$ and $r>0$ be given, and $L=-\divg A\nabla$ be a $\mu_0$-elliptic operator. Let $M'\ge M^2$. Let $\omega$ be the $L$-elliptic measure of $\Omega$ with pole at $X_0=X_{x,2M'r}$
and let $\wt\omega$ be  the $L$-elliptic measure of $\Omega\cap B(x,4M'r)$ with pole at $X_0$. Then $\omega|_{\Delta(x,r)}$ and $\wt\omega|_{\Delta(x,r)}$ are mutually absolutely continuous, and 
\begin{equation}\label{eq5d1}
    \frac{d\wt\omega}{d\omega}(y)\approx 1, \quad \omega\text{ a.e. }y\in \Delta(x,r),
\end{equation}
where the implicit constants depend on $d$, $\mu_0$, and the NTA constant. Therefore,  
\begin{equation}\label{5d1}
    \wt\omega(E)\approx\omega(E) \quad \text{for any Borel set $E\subset\Delta(x,r)$},
\end{equation}
where the implicit constants depend on $d$, $\mu_0$, and the NTA constant.
\end{lem}
\begin{proof}
Let $y\in\Gamma$ and $s>0$ be such that $B(y,s)\subset B(x,r)$. By Lemma \ref{lem cfms}, and \eqref{eq2.green3}, 
\[
\frac{\wt\omega(\Delta(y,s))}{\omega(\Delta(y,s))}\approx\frac{\wt G(X_0,X_{y,s})}{G(X_0,X_{y,s})}\approx\frac{\wt G^\top(X_{y,s},X_0)}{G^\top(X_{y,s},X_0)},
\]
where $\wt G(\cdot,Y)$ is the $L$-Green function for $\Omega\cap B(x,4M'r)$ with pole at $Y$, and $G(\cdot, Y)$ is the $L$-Green function for $\Omega$ with pole at $Y$. Since $\abs{X_0-x}\ge\dist(X_0,\pom)\ge 2M'r/M$ and we have chosen $M'\ge M^2$, $\wt G^\top(\cdot, X_0)$ and $G^\top(\cdot, X_0)$ are both solutions to $L^\top u=0$ in $B(x,2Mr)\cap\Omega$ that vanish on $\Delta(x,2Mr)$. By the comparison principle (Lemma \ref{lcomparison}) and the estimates \eqref{eq2.green} \eqref{eq2.green2} for the Green function, 
\[
\frac{\wt G^\top(X_{y,s},X_0)}{G^\top(X_{y,s},X_0)}\approx\frac{\wt G^\top(X_{x,r},X_0)}{G^\top(X_{x,r},X_0)}\approx1,
\]
and thus 
\[
\frac{\wt\omega(\Delta(y,s))}{\omega(\Delta(y,s))}\approx 1.
\]
Since the elliptic measures are regular and doubling, this implies that $\wt\omega|_{\Delta(x,r)}$ and $\omega|_{\Delta(x,r)}$ are mutually absolutely continuous. Then the Lebesgue differentiation theorem asserts that 
\[
\frac{d\wt\omega}{d\omega}(y)=\lim_{s\to 0}\frac{\wt\omega(\Delta(y,s))}{\omega(\Delta(y,s))}\approx 1 \quad \omega\text{ a.e. } y\in\Delta(x,r),
\]
which proves \eqref{eq5d1}. Then 
\eqref{5d1} is an immediate consequence of \eqref{eq5d1} since we can write
\(
\wt\omega(E)=\int_E\frac{d\wt\omega}{d\omega}(y)d\omega(y)\).
\end{proof}

\begin{lem}\label{lem loc2}
There exists a constant $M\ge1$ depending on the dimension so that the following holds. Let $M'\ge M^2$, $r>0$, $x\in\Gamma$, and $L$ be as in Lemma \ref{lem loc1}. Let $\omega$ be the $L$-elliptic measure of $\Omega$ with pole at $X_0= X_{x,M'r}$, and let $\wt\omega$ be the $L$-elliptic measure of $\Omega\cap B(x,M'r)$ with pole at $X_0$. Then 
\begin{equation}\label{5d2}
    \frac{d\wt\omega}{d\omega}(y)=\lim_{Y\to y}\frac{\wt G(X_0,Y)}{G(X_0,Y)} 
    \quad \text{for }\omega \text{ a.e. } y\in \Delta(x,r),
\end{equation}
where $\wt G(\cdot,Y)$ is the $L$-Green function for $\Omega\cap B(x,M'r)$, $G(\cdot,Y)$ is the $L$-Green function for $\Omega$, and the limit is taken in $B(x,M'r)\cap\Omega$. Moreover, for $E$, $E'\subset\Gamma\cap B(x,r)$,
\begin{equation}\label{5d3}
    \frac{1-\frac{C}{M'^\alpha}}{1+\frac{C}{M'^\alpha}}
  \,\,  \frac{\wt\omega(E)}{\wt\omega(E')}\le\frac{\omega(E)}{\omega(E')}\le 
  \frac{1+\frac{C}{M'^\alpha}}{1-\frac{C}{M'^\alpha}}
  \, \, \frac{\wt\omega(E)}{\wt\omega(E')},
\end{equation}
where $C$ and $\alpha$ are positive constants that depend on $d$, $\mu_0$, and the NTA constant of $\Omega$.
\end{lem}
\begin{proof}[Sketch of proof] 
By Lemma \ref{lem loc1}, $\wt\omega|_{\Delta(x,r)}$ and $\omega|_\Delta(x,r)$ are mutually absolutely continuous, 
and \[
\frac{d\omega}{d\wt\omega}(y)=\lim_{s\to 0}\frac{\omega(\Delta(y,s))}{\wt\omega(\Delta(y,s))} \quad \wt\omega\text{ a.e. } y\in\Delta(x,r).
\]
Then 
\eqref{5d2} can be proved using the 
Riesz formula, CFMS estimates, 
the 
ellipticity of the operator and the boundary Caccioppoli inequality. We refer the readers to \cite{bortz2020optimal} Lemma 5.1 for details. Now 
\eqref{5d3} is a consequence of \eqref{5d2}, as one can write
\[
\omega(E)=\int_E\frac{d\omega}{d\wt\omega}(y)d\wt\omega(y)=\int_E\mathcal{L}(y)d\wt\omega(y),
\]
where $\mathcal{L}(y)=\lim_{Y\to y}\frac{G(X_0,Y)}{\wt G(X_0,Y)}$. By the comparison principle, one can show that \[
\br{1-\frac{C}{M'^\alpha}}\mathcal{L}(x)\le\mathcal{L}(y)\le \br{1+\frac{C}{M'^\alpha}}\mathcal{L}(x) \quad\text{for }y\in\Delta(x,r).
\]
From here, \eqref{5d3} follows. Details can be found in e.g. \cite{kenig1997harmonic} Corollary 4.2.
\end{proof}
\ms
Now we start our proof of Theorem \ref{thmCAS}. Let $\delta\in(0,\frac{1}{10})$, $x_0\in\Gamma$ and 
$r_0>0$ be given. Let $\Delta_0=B(x_0,r_0)\cap\Gamma$, and let $E\subset\Delta_0$.
Let $M$ be the largest of the 
two constants $M$ in Lemmata \ref{lem loc1} and \ref{lem loc2}, and take $K\ge M^2$ to be determined. 
Set $X_0=X_{x_0,Kr_0}\in\Omega\cap B(x_0, Kr_0)$ to be a 
corkscrew point for $x_0$ at scale $Kr_0$. 

By Corollary \ref{cor U}, there exists an $\epsilon^{\frac{1}{2d}}$-Lipschitz graph $G=\set{(y,A(y)): y\in\rd}$ 
such that for one of the connected components $U$ of $\ree\sm G$, we have 
\begin{equation}\label{eq Om<U}
    \Omega\cap B(x_0, 10 K r_0)\subset U
\end{equation}
and 
\begin{equation}\label{eq GammaG sm}
    \H^d(\Gamma \cap B(x_0, 10 K r_0)\sm G) + \H^d(G \cap B(x_0,10Kr_0)\sm \Gamma)
\leq C K^{d}\varepsilon^{1/2} r_0^d. 
\end{equation}

Similar to \eqref{5b5}, we have that
\begin{align}\label{distGGamma}
\begin{split}
       & \dist(y,G) \leq C K\varepsilon^{\frac1{2d}} r_0 \text{ for } y \in \Gamma \cap B(x_0,10Kr_0/2),\\
&\dist(y,\Gamma) \leq C K\varepsilon^{\frac1{2d}}r_0 \text{ for } y \in G \cap B(x_0,10Kr_0/2).
\end{split}
\end{align}
We shall choose
\begin{equation}\label{KK}
K = K(\varepsilon) = \varepsilon^{-\beta},
\end{equation}
with a small constant $\beta > 0$ that will be chosen near the end of the argument, and then $\varepsilon
=\epsilon(\delta,d,\mu_0)>0$ sufficiently small. Thus $K$ is as large as we want, and in particular
we shall choose $\varepsilon$ so small that $K \geq M$.
We intend to 
show that for $\epsilon=\epsilon(\delta,d,\mu_0)>0$ sufficiently small,
\begin{equation}\label{eq goal}
    \frac{\omega^{X_0}_\Omega(E)}{\omega^{X_0}_\om(\Delta_0)}\ge\frac{\sigma(E)}{\sigma(\Delta_0)}-\delta.
\end{equation}
Once we have \eqref{eq goal} for any set $E\subset\Delta_0$, we obtain that
\(
\frac{\omega^{X_0}_\Omega(E)}{\omega^{X_0}_\om(\Delta_0)}\le\frac{\sigma(E)}{\sigma(\Delta_0)}+\delta
\) by taking $E$ to be $\Delta_0\sm E$ in \eqref{eq goal}. Therefore, it suffices to show \eqref{eq goal}.
\ms

Set $B = B(x_0,Kr_0)$ to lighten the notation. We first transfer elliptic measures from 
$\Omega$ to $\Omega\cap B$; this will be convenient because $\Omega\cap B$ is contained in $U$, 
while $\Omega$ may not be.

Recall that $L= -\divg A\nabla$. In Section~\ref{SExt}, we will show that there is a
$\mu_0$-elliptic operator $\wt L=-\divg \wt A\nabla$, that satisfies the weak DKP condition in $U$
with constant $\cN(\wt A)\le C\epsilon$, and such that 
\begin{equation}\label{adda}
\wt A(X) = A(X) \text{ on } \Omega\cap B(x_0,2Kr_0)\subset U.
\end{equation}
Denote by $\omega$ the $L$-elliptic measure of $\Omega\cap B$ with pole at $X_0$
and by $\omega'$ the $\wt L$-elliptic measure of $U\cap B$ with pole at $X_0$. 
Define 
\begin{equation}\label{def Z}
    \mathcal{Z}:= \Gamma\cap B\sm G.
\end{equation}
Observe that by \eqref{eq GammaG sm}, 
\begin{equation}\label{Hd(Z)}
    \H^d(\mZ)\le CK^{d}\eps^{1/2}r_0^d. 
\end{equation}
Our first claim is that for any set $H\subset \Gamma\cap G\cap B(x_0,r_0)$,
\begin{equation}\label{eq ZH}
    \omega'(H)\le \omega(\mZ)+\omega(H).
\end{equation}
\begin{proof}[Proof of \eqref{eq ZH}]
Let $H_n\subset H$ be closed sets in $G\cap\Gamma$ such that $H_n\uparrow H$ as $n\to\infty$. Let $O_n\subset G\cap B(x_0,2r_0)$ be open sets such that $O_n\downarrow H$. Let $g_n\in C^\infty_c(O_n)$ be such that $\1_{H_n}\le g_n\le \1_{O_n}$. For $X\in U\cap B$, define $u_n(X)=\int g_n(y)d\omega'^X(y)$, where $\omega'^X$ is the $\wt L$-elliptic measure of $U\cap B$ with pole at $X$. Then by the definition of elliptic measures, $\wt Lu_n=0$ in $U\cap B$ and $u$ is continuous in $\overline{U\cap B}$. Since $\bdy(\om\cap B)\subset \overline{U\cap B}$, $u_n$ is continuous on $\bdy(\om\cap B)$. Let $v_n(X)=\int u_n(y)d\omega^X(y)$, where $\omega^X$ is the $L$-elliptic measure of $\Omega\cap B$ with pole at $X$. Then $Lv_n=0$ in $\om\cap B$, with $v_n=u_n$ on $\bdy(\om\cap B)$. Since $\wt L= L$ in $\Omega\cap B$, the maximum principle implies that $v_n=u_n$ in $\om\cap B$. Observe that 
\(
\bdy(\om\cap B)=(\Gamma\cap G\cap B)\cup \mZ \cup (\bdy B\cap \om)
\),
and that $u_n=g_n$ on $(\Gamma\cap G\cap B)\cup (\bdy B\cap \om)$. So for $X\in\om\cap B$,
\begin{multline*}
    u_n(X)=v_n(X)=\int_{\Gamma\cap G\cap B}g_n(y)d\omega^X(y)+\int_\mZ u_n(y)d\omega^X(y)+\int_{\bdy B\cap\om}g_n(y)d\omega^X(y)\\
    =\int_{\Gamma\cap O_n}g_n(y)d\omega^X(y)+\int_\mZ u_n(y)d\omega^X(y)\le \omega^X(O_n\cap\Gamma)+\omega^X(\mZ),
\end{multline*}
where in the last inequality we have used $g_n\le \1_{O_n}$ and $u_n\le 1$. On the other hand, since $g_n\ge \1_{H_n}$,
\[
u_n(X)\ge \int_{H_n}d\omega'^X(y)=\omega'^X(H_n).
\]
So we get that 
\[
\omega'^X(H_n)\le \omega^X(O_n\cap\Gamma)+\omega^X(\mZ).
\]
Then \eqref{eq ZH} follows from taking $n\to\infty$ and by the 
regularity of elliptic measures. 
\end{proof}
Now we define another set $F$, which we think of as the 
large shadow of $\mZ$ on $G$, which is defined as 
\begin{equation}\label{def F}
    F:=\bigcup_{Z\in\mZ} G  \cap B(Z, 20 \,\delta(Z)).
\end{equation}
Here, and in the sequel, $\delta(Z) = \dist(Z, G)$. Notice that by \eqref{distGGamma}, 
$\delta(Z)\le C\eps^{\frac{1}{2d}}K r_0$ 
for all $Z\in\mZ$. We show that 
\begin{equation}\label{eq Fsm}
    \H^d(F)\le CK^{d}\epsilon^{1/2}r_0^d. 
\end{equation}
\begin{proof}[Proof of \eqref{eq Fsm}]
For each $Z\in\mZ$, let $B_Z=B(Z,r(Z))$ with $r(Z)=\dist(Z,G)/100$. Then by the Vitali covering lemma, we can find a countable collection 
of balls $B_{Z}$, $Z \in I$, such that the $B_Z$ are disjoint but $\mZ\subset\underset{Z\in I}{\bigcup}5B_Z$.
Recall that $G$ is the graph of $A: \rd\to\Real$. Write $x_0=(x,t_0)$, where $x\in\rd$ and $t_0\in\Real$. By \eqref{distGGamma}, for $\eps$ sufficiently small, we have that \(F\subset \set{(y,A(y)): y\in\rd\cap B(x,2Kr_0)}\). Set 
\[
H:=\set{y\in\rd\cap B(x,2Kr_0): (y, A(y))\in F}.
\]
Observe that for any fixed $y\in H$, there exists a $Z=(z,t)\in\mZ$ such that $\abs{(y,A(y))-(z,t)}\le 20\delta(Z)$. But there must be a $Z'=(z',t')\in I$ such that $Z\in 5B_{Z'}$. So 
\[
\delta(Z)\le \delta(Z')+\abs{Z-Z'}\le \delta(Z')+5r(Z')\le 2\delta(Z').
\]
Then
\[
\abs{y-z'}\le\abs{y-z}+\abs{z-z'}\le 20\delta(Z)+5r(Z')\le 21\delta(Z'),
\]
which implies that $y\in B(z',21\delta(Z'))$. Since $y\in H$ is arbitrary, we obtain that 
\[
H\subset\bigcup_{Z=(z,t)\in I}B(z,21\delta(Z)).
\]
Since $\Gamma$ is $d$-Ahlfors regular, 
\[
\H^d(H)\le C\sum_{Z\in I}\delta(Z)^d\le C\sum_{Z\in I}\H^d(\Gamma\cap B_Z).
\]
Recall that the 
$B_Z$ are disjoint and that $B_Z\subset B(x_0,2Kr_0)\sm G$, so 
\[
\sum_{Z\in I}\H^d(\Gamma\cap B_Z)\le C\H^d(\Gamma\cap B(x_0,2Kr_0)\sm G) 
\le CK^{d}\eps^{1/2}r_0^d, 
\]
where the last inequality follows from \eqref{eq GammaG sm}.
Therefore, 
\[
\H^d(F)\le\int_{H}\sqrt{1+\abs{\nabla A}^2}dy\le C(1+2\eps^{1/d})K^{d} 
\eps^{1/2}r_0^d\le CK^{d}
\eps^{1/2}r_0^d,
\]
as desired.
\end{proof}
We now consider the $\wt L$-ellipitc measure of $U$. We claim that for some $C>1$ depending on $d$ and $\mu_0$,
\begin{equation}\label{eq omUF}
    \omega_U^Z(F)\ge \frac{1}{C} \quad \text{for all }Z\in\mZ.
\end{equation}
\begin{proof}[Proof of \eqref{eq omUF}]
Let $Z=(z,t)\in\mZ$ be given. We first show that
\begin{equation}\label{5d4}
    (y, A(y))\in F \quad \text{for any $y\in\rd$ that satisfies $\abs{y-z}\le 10\delta(Z)$}.
\end{equation}
Let $Z'=(z', A(z'))\in G$ be such that $\abs{Z-Z'}=\delta(Z)$. Then 
\[
\abs{z-z'}+\abs{t-A(z')}\le\sqrt{2\abs{Z-Z'}^2}<2\delta(Z).
\]
For any $y\in\rd$ with $\abs{y-z}\le 10\delta(Z)$, we have that 
\begin{multline*}
    \abs{(y,A(y))-Z}\le\abs{y-z}+\abs{A(y)-t}\\
    \le\abs{y-z}+\abs{A(y)-A(z)}+\abs{A(z)-A(z')}+\abs{A(z')-t}\\
    \le \br{1+\eps^{\frac{1}{2d}}}\abs{y-z}+\eps^{\frac{1}{2d}}\abs{z-z'}+\abs{A(z')-t}\\
    \le\br{12+3\eps^{\frac{1}{2d}}}\delta(Z)\le 20\delta(Z),
\end{multline*}
which proves \eqref{5d4}. 

By \eqref{5d4}, $Z'\in F$ and $B(Z',5\delta(Z))\cap G\subset F$. Therefore, 
\[
\omega_U^Z(F)\ge\omega^Z_U(B(Z',5\delta(Z))\cap G)\ge\frac{1}{C},
\]
where in the last inequality we have used Bourgain's estimate (Lemma \ref{lem Bourgain}).
\end{proof}
With \eqref{eq omUF}, we are ready to show that 
\begin{equation}\label{eq UFgeZ}
    \omega_U^{X_0}(F)\ge \frac{1}{C}\omega(\mZ).
\end{equation}
\begin{proof}[Proof of \eqref{eq UFgeZ}]
Observe that $\mZ\subset\Gamma$ is open in $\Gamma$, and that $F\subset G$ is open in $G$. 
Let $F_n\subset F$ be a sequence of closed sets such that $F_n\uparrow F$, and let $\mZ_m\subset\mZ$ be a sequence of closed sets such that $\mZ_m\uparrow \mZ$. Let $f_n\in C_c^\infty(F)$ and $\1_{F_n}\le f_n\le \1_{F}$. Let $g_m\subset C_c^\infty(\mZ)$ and $\1_{\mZ_m}\le g_m\le \1_{\mZ}$. For $X\in U$, set $v_n(X)=\int_Gf_n(y)d\omega_U^X$. Then by the definition of elliptic measures, $\wt L v_n=0$ in $U$ and $v_n=f_n$ on $G$. Also, 
\begin{equation}\label{5d5}
\omega_U^X(F)\ge v_n(X)\ge\omega_U^X(F_n).     
\end{equation}
For $X\in \om\cap B$, set 
\[
u_m^{(n)}(X)=\int_{\bdy(\om\cap B)}g_m(y)v_n(y)d\omega_{\om\cap B}^X(y).
\]
Then $Lu_m^{(n)}=0$ in $\om\cap B$ and $u_m^{(n)}=g_m\,v_n$ on $\bdy(\om\cap B)$. Since $\wt L=L$ in $\om\cap B$, the maximum principle implies that 
\[
v_n(X)\ge u_m^{(n)}(X) \quad\text{for any }X\in\om\cap B, \text{ for any }m,n\in\NN.
\]
Then by $g_m\ge\1_{\mZ_m}$ and \eqref{5d5},
\[
    \omega_U^X(F)\ge\int_{\mZ_m}v_n(y)\,d\omega_{\om\cap B}^X(y)\ge \int_{\mZ_m}\omega_U^y(F_n)\,d\omega_{\om\cap B}^X(y).
\]
Letting $n\to\infty$, by the regularity of $\omega_U$ and \eqref{eq UFgeZ}, we get that
\[
\omega_U^X(F)\ge\int_{\mZ_m}\omega_U^y(F)\,d\omega_{\om\cap B}^X(y)\ge \frac{1}{C}\omega_{\om\cap B}^X(\mZ_m).
\]
Now letting $m\to\infty$, the regularity of $\omega_{\om\cap B}$ gives that 
\[
\omega_U^X(F)\ge \frac{1}{C}\omega_{\om\cap B}^X(\mZ) \quad\text{for any }X\in\om\cap B.
\]
This proves \eqref{eq omUF} since $X_0\in\om\cap B$.
\end{proof}
\ms
We are now ready to prove \eqref{eq goal}.  By \eqref{5d3}, 
\begin{equation}\label{5d6}
    \frac{\omega_\om^{X_0}(E)}{\omega_\om^{X_0}(\Delta_0)}
    \ge \frac{1-\frac{C}{K^\alpha}}{1+\frac{C}{K^\alpha}}
    \,\, \frac{\omega(E)}{\omega(\Delta_0)}\ge\frac{\omega(E)}{\omega(\Delta_0)}-C\epsilon^{\alpha\beta},
\end{equation}
where, according to \eqref{KK}, 
we have chosen $K=\epsilon^{-\beta}$ for some $\beta>0$ to be determined later.
We now compute a lower bound for  
$\frac{\omega(E)}{\omega(\Delta_0)}$. Since $E$ and $\Delta_0$ might not be contained in $G$, we write
\[
\frac{\omega(E)}{\omega(\Delta_0)}\ge \frac{\omega(E\cap G)}{\omega(\Delta_0\cap G)+\omega(\Delta_0\sm G)}.
\]
By \eqref{eq ZH}, 
\[
\omega(E\cap G)\ge\omega'(E\cap G)-\omega(\mZ).
\]
Since $\om\cap B\subset U\cap B$ and $L=\wt L$ on $\om\cap B$, the maximum principle 
implies that $\omega(\Delta_0\cap G)\le \omega'(\Delta_0\cap G)$. 
Also, from the definition of $\mZ$, it follows that $\omega(\Delta_0\sm G)\le\omega(\mZ)$. So
\begin{equation}\label{5d7}
    \frac{\omega(E)}{\omega(\Delta_0)}\ge\frac{\omega'(E\cap G)-\omega(\mZ)}{\omega'(\Delta_0\cap G)+\omega(\mZ)}
    \ge \frac{\omega'(E\cap G)}{\omega'(\Delta_0\cap G)}-\frac{2\omega(\mZ)}
    {\omega'(\Delta_0\cap G)}=: I_1-2I_2.
\end{equation}
  Let us first show that $I_2$ can be arbitrary small. By \eqref{eq UFgeZ} and \eqref{5d1}, 
  \begin{equation}\label{5d7.1}
      I_2\le C \,\frac{\omega_U^{X_0}(F)}{\omega_U^{X_0}(\Delta_0\cap G)}.
  \end{equation}
  By Section~\ref{SLip}, the $\wt L$-elliptic measure of $U$ is in $A_\infty(d\H^d|_{G})$ (we even have $A_\infty$ with small constant). Therefore, 
  \[
  \omega^{X_0}_U(F)\le C\omega_U^{X_0}(B\cap G)\br{\frac{\H^d(F)}{\H^d(B\cap G)}}^\theta,
  \]
  and 
  \[
  \omega^{X_0}_U(\Delta_0\cap G)\ge \frac{1}{C}\omega_U^{X_0}(B(x_0,r_0)\cap G)\br{\frac{\H^d(\Delta_0\cap G)}{\H^d(B(x_0,r_0)\cap G)}}^\eta,  \]
  where $\theta,\eta>0$ are $A_\infty$ constants, which are independent of the sets. Again by $A_\infty$, we have that 
  \[
  \frac{\omega_U^{X_0}(B\cap G)}{\omega_U^{X_0}(B(x_0,r_0)\cap G)}\le C\br{\frac{\H^d(B(x_0,r_0)\cap G)}{\H^d(B\cap G)}}^{-\eta}.
  \]
  Substituting these estimates in \eqref{5d7.1}, we get that 
  \begin{equation}\label{5d8}
      I_2\le C\br{\frac{\H^d(F)}{\H^d(B\cap G)}}^\theta\Big/\br{\frac{\H^d(\Delta_0\cap G)}{\H^d(B\cap G)}}^\eta
  \end{equation}
  Observe that 
  \[
  \H^d(\Delta_0\cap G)=\H^d(\Gamma\cap B(x_0,r_0)\cap G)\ge\H^d(\Gamma\cap B(x_0,r_0))-\H^d(\mZ).
  \]
Thus, if we choose 
$K=K(\eps)$ according to \eqref{KK}
and $\eps$ so small that
  \begin{equation}\label{Kreq1}
      CK^{d}\eps^{1/2}<\frac{(1-C\epsilon^{1/d})c_d}{2}, 
  \end{equation}
  then by \eqref{5b2bis} and \eqref{Hd(Z)}, 
  \begin{equation}\label{5d9}
      \H^d(\Delta_0\cap G)\ge \frac{(1-C\epsilon^{1/d})c_d}{2} \,\, r_0^d\ge C^{-1}       r_0^d.  
  \end{equation}
  We now check that 
  \begin{equation}\label{H(B.G)}
         \frac{c_d}{2}K^dr_0^d  \le \H^d(B\cap G)\le c_d(1+C \eps^{\frac{1}{2d}})K^dr_0^d. 
     \end{equation}
 The second inequality follows directly from the fact that $G$ is an $\eps^{\frac{1}{2d}}$-Lipschitz graph. 
 To see the first inequality, 
  we want to use a ball centered on $\Gamma$. By \eqref{distGGamma}, there exists a point $x^*\in G$ 
 such that $\delta(x_0)=\abs{x_0-x^*}\le C\eps^{\frac{1}{2d}} K r_0$. 
 Since $B(x^*, K r_0 - \delta(x_0)) \subset B$, we get that 
  \[
  \H^d(B\cap G) \geq \H^d(B(x^*, K r_0 - \delta(x_0)))
\geq (1-C\varepsilon^{\frac{1}{2d}})( K r_0 - \delta(x_0))^{d} \geq \frac{c_d}{2}K^dr_0^d
  \]
because $G$ is an $\eps^{\frac{1}{2d}}$-Lipschitz graph and if $\varepsilon$ is small enough;
\eqref{H(B.G)} follows.
 Using \eqref{eq Fsm}, \eqref{5d9} and \eqref{H(B.G)} in \eqref{5d8}, we obtain that
 \begin{equation}\label{5dI2}
      I_2\le CK^{d\eta}\eps^{\frac{\theta}{2}}.
 \end{equation}
We return to $I_1$ in \eqref{5d7}. By \eqref{5d3}, 
\begin{equation}\label{5d7.2}
    I_1\ge \frac{1-\frac{C}{K^\alpha}}{1+\frac{C}{K^\alpha}}\, \, 
    \frac{\omega^{X_0}_U(E\cap G)}{\omega_U^{X_0}(\Delta_0\cap G)}.
\end{equation}
 We are in a position to apply the $A_\infty(\sigma,\delta)$ result for small Lipschitz graph, which says that 
 for $\eps=\eps(d,\mu_0,\delta)$ sufficiently small and 
 \begin{equation}\label{Kreq2}
     K\ge\frac{M}{\tau(d,\delta,\mu_0)},
 \end{equation}
 where $\tau=\tau(d,\delta,\mu_0)$ is as in Theorem \ref{t1b2} for small Lipschitz graphs, 
 \begin{equation}\label{5d10}
     \frac{\omega^{X_0}_U(E\cap G)}{\omega_U^{X_0}(\Delta_0\cap G)}
     \ge\frac{\H^d(E\cap G)}{\H^d(\Delta_0\cap G)}-\frac{\delta}{8}.
 \end{equation}
Notice that we do 
not have a zero denominator on the right-hand side of \eqref{5d10},
and even $\H^d(\Delta_0\cap G)\ge c_dr_0^d$,
because $B(x^*,r_0/2)\subset B(x_0,r_0)$ for the same $x^*$ as above, as soon as we choose
$\beta < \frac{1}{4d}$ in \eqref{KK} and 
  $\eps$ so small that 
\begin{equation}\label{Kreq2.5}
    CK \eps^{\frac{1}{2d}}\le 1/2. 
\end{equation}
Since we want $\frac{\H^d(E)}{\H^d(\Delta_0)}$ in \eqref{eq goal}, we need to compare 
$\frac{\H^d(E\cap G)}{\H^d(\Delta_0\cap G)}$ and $\frac{\H^d(E)}{\H^d(\Delta_0)}$. 
We use \eqref{Hd(Z)} and the Ahlfors
regularity of $\Gamma$ to get that 
\[
    \frac{\H^d(E\cap G)}{\H^d(\Delta_0\cap G)}\ge \frac{\H^d(E)-\H^d(\mZ)}{\H^d(\Delta_0)}
    \ge\frac{\H^d(E)}{\H^d(\Delta_0)}-\frac{CK^{d}\eps^{1/2}}{(1+C\eps)c_d} 
    \ge \frac{\H^d(E)}{\H^d(\Delta_0)}-\frac{\delta}{8},
\]
where we have chosen $\beta < \frac{1}{4d}$ in \eqref{KK} 
and $\eps=\eps(d,\delta,\mu_0)$ 
so small that
 \begin{equation}\label{Kreq3}
 \frac{CK^{d}\eps^{1/2}}{(1+C\eps)c_d}\le\frac{\delta}{8} 
\end{equation}
in the last inequality. Returning to \eqref{5d7.2}, we obtain that 
for $\varepsilon$ small enough,
\begin{equation}\label{5dI1}
    I_1
    \ge \frac{1-\frac{C}{K^\alpha}}{1+\frac{C}{K^\alpha}}\br{\frac{\H^d(E)}{\H^d(\Delta_0)}
          -\frac{\delta}{4}}
         \ge \frac{\H^d(E)}{\H^d(\Delta_0)}-\frac{\delta}{4}- C\eps^{\alpha\beta}\ge \frac{\H^d(E)}{\H^d(\Delta_0)}-\frac{\delta}{2}.  , 
\end{equation}

Altogether, by \eqref{5dI1}, \eqref{5dI2}, \eqref{5d7} and \eqref{5d6}, we get that
\[
\frac{\omega_\om^{X_0}(E)}{\omega_{\om}^{X_0}(\Delta_0)}\ge 
\frac{\H^d(E)}{\H^d(\Delta_0)}-\frac{\delta}{2}-CK^{d\eta} 
\eps^{\frac{\theta}{2}}-C\epsilon^{\alpha\beta}\ge \frac{\H^d(E)}{\H^d(\Delta_0)}-\delta
\]
if we choose $\beta$ so that $d\eta \beta < \frac{\theta}{2}$ and $\varepsilon$ so small that
\begin{equation}\label{Kreq5}
    CK^{d\eta}\eps^{\frac{\theta}{2}}+C\epsilon^{\alpha\beta}\le\delta/2
\end{equation}
in the last inequality.

Finally, we check that all the conditions (\eqref{Kreq1},\eqref{Kreq2},\eqref{Kreq2.5},\eqref{Kreq3},
\eqref{Kreq5}) on $K$ and $\eps$ can be satisfied if we choose 
$0<\beta<\min\set{\frac{1}{2d},\frac{\theta}{2d\eta}}$ 
and $\eps=\eps(d,\mu_0,\delta)$ sufficiently small. This completes the proof of \eqref{eq goal}.
\qed

\subsection{We extend to $U$ the matrix of coefficients}\label{SExt}

In this subsection, 
we extend the matrix $A$ of coefficients 
of $L=-\divg(A\nabla)$ to $U$ so that the extension 
satisfies the weak DKP condition on $G\times(0,\infty)$, with a constant
smaller than $C\cN(A) \leq C \varepsilon$. 

We shall use the setup in Section~\ref{SCASSCmain}, that is, given $x_0\in\Gamma$ and $r_0>0$, 
we have an $\epsilon^{\frac{1}{2d}}$-Lipschitz graph $G$ such that 
\eqref{eq Om<U} and \eqref{eq GammaG sm} hold. 

We have seen in Section~\ref{SCASSCmain} that if we can find a $\mu_0$-elliptic 
operator $\wt L = -\divg(\wt A\nabla)$ on $U$, that satisfies the weak DKP condition in $U$ 
with constant  $\cN(\wt A)\le C\epsilon$, and such that 
\begin{equation}\label{adda2}
\wt A = A \text{ on } \Omega\cap B(x_0,2Kr_0)\subset U,
\end{equation}
as in \eqref{adda}, then we can prove of Theorem \ref{thmCAS}. 
We also required that \eqref{eq Om<U} and \eqref{eq GammaG sm} hold in the larger ball $B(x_0, 10K r_0)$, 
and we shall use the extra space to build the desired extension $\wt A$ with a small enough norm $\cN(\wt A)$. 
The construction will be similar to the usual proof of the Whitney extension theorem with partitions of unity.

Set $B=B(x_0,Kr_0)$ as before. Because of \eqref{adda2}, we want to keep 
$\wt A = A$ on the closure of $\om\cap 2B$, and we are free to define $\wt A$ as we please on 
$V =\ree\sm \overline{\om\cap 2B}$.
So we cover $V$ by Whitney cubes, which we construct as in \cite{stein1970singular}, Chapter VI.
That is, let $\{ Q_i \}$, $i\in I$, denote the the collection of maximal dyadic cubes $Q_i \subset V$ such that, say,
\begin{equation} \label{w63}
\diam(Q_i) \leq 10^{-1}\dist(Q_i, \om\cap 2B).
\end{equation}
The $Q_i$, $i\in I$, cover $V$, and they have overlap properties that we will recall when we need them. We can also construct a partition of unity 
$\set{\chi_i}_{i\in I}$ on $V$, adapted to $\set{Q_i}_{i\in I}$, such that 
$\sum_i\chi_i=\1_{V}$ and for each $i$, 
$\chi_i\in C_c^\infty(\frac65 Q_i)$, $0\le \chi_i\le 1$, and
$\abs{\nabla\chi_i}\le C\diam(Q_i)^{-1}$. See for instance \cite{stein1970singular} for details.
As usual in these instances, we shall keep $\wt A = A$ on $\overline{\om\cap 2B}$ (as suggested above)
and set 
\begin{equation} \label{w64} 
\wt A(X) = \sum_{i\in I} A_i\, \chi_i(X) \ \text{ for }  X \in V,
\end{equation}
where we take 
\begin{equation} \label{w65}
A_i = \fint_{W_i} A(X) dX,
\end{equation}
 where the average is componentwise, for some set $W_i \subset \Omega$ (so that $A$ is defined on $W_i$), that we shall now choose carefully. 
Notice that $A_i$ and $\wt A(x)$ are $\mu_0$-elliptic matrices, because they are averages of $\mu_0$-elliptic matrices.
Set $d_i = \dist(Q_i, \om\cap 2B)$. We start with the case when
\begin{equation} \label{w66}
\dist(Q_i, \Gamma) \leq d_i \leq Kr_0
\end{equation}
(and then we say that $i \in I_1$). Then we pick $\xi_i \in \Gamma$ such that
$\dist(\xi_i,Q_i) = \dist(Q_i, \Gamma) \leq d_i$, 
denote by $W_i$ the Whitney cube $W(\xi_i,d_i) \subset \Omega$ associated to $\Omega$ as in \eqref{1b2}.
When 
\begin{equation} \label{w67}
d_i < \dist(Q_i, \Gamma) \text{ and } d_i \leq Kr_0,
\end{equation}
we say that $i \in I_2$ and we choose $W_i = Q_i$. Notice that $Q_i$ does not meet $\Gamma$, so $Q_i$ is either contained in $\Omega$ or in $\R^{d+1} \sm \Omega$.
The second case is impossible because $\dist(Q_i,\Omega \cap B) \leq d_i<\dist(Q_i,\Gamma)$,
so $W_i \subset \Omega$
as needed. We are left with the case when 
\begin{equation} \label{w68}
d_i > Kr_0.
\end{equation}
Then we say that $i \in I_3$, and we decide to take $W_i = W(x_0, Kr_0)$, which is contained in $\Omega \cap B$
by definition. This will conveniently kill the variations of $\wt A$ far from $B$.

Notice that $W_i\subset\Omega\cap 5B$ for all $i \in I$. This is obvious by definition when $i \in I_3$ and because 
$d_i \le K r_0$ otherwise, so that $Q_i \subset 3 B$, which proves the case when $i\in I_2$. When $i \in I_1$, $\xi_i \in \frac72 B$, so $W_i=W(\xi_i,d_i)\subset{\frac92 B}$.

It will also be good to know that
\begin{equation} \label{w69}
\dist(X, Q_i) \leq 3 d_i = 3\dist(Q_i, \om\cap 2B)
\ \text{ for $X \in W_i$}, \quad i\in I.
\end{equation}
When $i \in I_1$, this is clear because $\dist(\xi_i,Q_i) \leq d_i$. 
When $i \in I_2$ this is trivial. Finally, when
$i \in I_3$, set $D = \dist(Q_i, x_0)$, observe that $W_i = W(x_0, Kr_0) \subset B$ lies in a 
$D+Kr_0$-neighborhood of $Q_i$, while $D \leq \dist(Q_i, \om\cap 2B) + 2K r_0 = d_i + 2Kr_0
\leq 3d_i$ by definition of $I_3$.

So we have a function $\wt A$ defined on the whole $\reu$, and our next task is to evaluate the numbers 
$\overline\alpha(y,r)$ associated to (the restriction to $U$ of) $\wt A$, defined for $y \in G$ and $r > 0$.
Recall that they are defined by 
\begin{equation} \label{bad}
\overline\alpha(y,r) = \inf_{A_0} \Big\{\fint_{\overline W(y,r)} |\wt A-A_0|^2 \Big\}^{1/2},
\end{equation}
where the infimum is taken over constant $\mu_0$-elliptic matrices $A_0$, and we use the Whitney boxes 
\begin{equation} \label{bac}
\overline W(y,r) = \big\{ X \in U \cap B(y,r)\, ; \, \dist(X, G) \geq  r/2\big\}.
\end{equation}

We claim that 
\begin{equation} \label{w82}
\overline\alpha(y,r) = 0 \ \text{ when } r > 6 K r_0,
\end{equation}
and 
\begin{equation}\label{w82'}
    \overline\alpha(y,r) = 0 \ \text{ when } \dist(y,\Gamma\cap 2B)>2r \text{ and }\dist(y,\Gamma\cap 2B)>6Kr_0.
\end{equation}
To see this, let us first prove that 
\begin{equation} \label{w84}
\wt A(X) = A_{00} : =\fint_{W(x_0, K r_0)} A(Y) dY  
\ \text{ for } X \in \R^{d+1} \sm 3B.
\end{equation}
Indeed for any Whitney cube $Q_i$, $i\in I_1 \cup I_2$, 
we have that $\dist(Q_i, \Omega \cap 2B) = d_i \leq K r_0$,
and then $\frac65 Q_i \subset 3B$ because $\diam(Q_i) \leq d_i/10$. 
Now if $X \in \R^{d+1} \sm 3B$, all the indices $i$ such that $\chi_i(X) \neq 0$
lie in $I_3$, and since $A_i = A_{00}$ for $i \in I_3$, \eqref{w64} yields $\wt A(X) = A_{00}$.

If $r > 6 K_0$, we have that $\dist(X,G) > r/2 \geq 3K r_0$ for $X \in \overline W(y,r)$, but
then $\dist(X,x_0) \geq \dist(X,G) > 3K r_0$ and $\wt A(X) = A_{00}$. That is, 
$\wt A$ is constant on $\overline W(y,r)$, and \eqref{w82} follows. If $\dist(y,\Gamma\cap 2B)>6Kr_0$ and $r<\dist(y,\Gamma\cap 2B)/2$, then for any $Z\in \overline W(y,r)$, 
\(
|Z-x_0|\ge\abs{y-x_0}-\abs{Z-y}\ge \dist(y,\Gamma\cap 2B)/2>3Kr_0,
\)
which shows that $\overline W(y,r)\subset\ree\sm 3B$. Then \eqref{w82'} follows from \eqref{w84}.

Let $y \in G$ and $r > 0$ be given.
We shall distinguish between cases, and the most interesting is probably when 
\begin{equation} \label{w72}
\dist(y,\Gamma \cap 2B) \leq 2r \text{ and }r\le 6Kr_0. 
\end{equation}
We claim that in this case
\begin{equation} \label{bae}
\overline\alpha(y,r) \leq C \gamma_A(x, 80 r) 
\text{ for every } x\in \Gamma \cap B(y, 4 r),
\end{equation}
where $\gamma_A(x,r)$ is the variant of $\alpha_A(x,r)$, but defined with balls. That is, 
\begin{equation} \label{baf}
\gamma_A(x,r) = \inf_{A_0} \Big\{\fint_{\Omega \cap B(x,r)} |A-A_0|^2\Big\}^{1/2}.
\end{equation}
Let $x \in \Gamma \cap B(y, 4 r)$ be as in the claim. We will use the same constant matrix 
$A_0$ as in the definition of $\gamma_A(x,80r)$ to evaluate $\overline\alpha(y,r)$. First write
\[
\overline{\alpha}(y,r)^2
\le \abs{\overline W(y,r)}^{-1}\br{\int_{\overline{W}(y,r)\cap \overline{\om \cap 2B}}\abs{A-A_0}^2dX
+\int_{\overline{W}(y,r)\cap V}\abs{\wt A-A_0}^2dX},
\]
and notice that the first part is in order, because $\Omega \cap \overline W(y,r) \subset 
\Omega \cap B(y,r) \subset \Omega \cap B(x,80r)$ and the denominators $\abs{\overline W(y,r)}$ and $\abs{B(x,80r)}$
are comparable. We are left with the integral $\displaystyle J =\int_{\overline W(y,r) \cap V} |\wt A-A_0|^2$.
Let $I(y,r)$ be the collection of indices such that $Q_i$ meets $\overline{W}(y,r)\cap V$; 
then $J \leq \sum_{i \in I(y,r)} \big|\wt A-A_0\big|^2dX$, and the definition \eqref{w64} of $\wt A$ yields
\[
 J \le\sum_{i \in I(y,r)}\int_{Q_i}
 \Big|\sum_j A_j \,\chi_j(X)-A_0\Big|^2dX
=\sum_{i \in I(y,r)}\int_{Q_i}
 \Big|\sum_j (A_j-A_0) \,\chi_j(X)\Big|^2dX
\]
because $\sum_j \chi_j(X) = 1$.
Observe that for $X\in Q_i$, we only sum over the set $J_i$ of those $j \in I$ such that $\frac65 Q_j$ contains $X$ and hence meets
$Q_i$. There are only finitely many of such indices $j$, and they are all such that 
\(
\frac{1}{4}\diam(Q_j)\le\diam(Q_i)\le 4\diam(Q_j)\). 
We use Minkowski's inequality to get that 
\[
\Big(\int_{Q_i}\Big|\sum_j(A_{j}-A_0)\chi_j(X)\Big|^2dX \Big)^{1/2}
\le\sum_{j\in J_i}\br{\int_{Q_i}\abs{(A_{j}-A_0)\chi_j(X)}^2dX}^{1/2}.
\]
Since $\chi_j\le 1$ and by definition of $A_{j}$, the right-hand side is at most 
\[
    \sum_{j\in J_i}\abs{Q_i}^{1/2}\abs{A_{j}-A_0}
    \le C\sum_{j\in J_i}\bigg(\frac{|Q_j|}{|W_j|}\int_{W_j}\abs{A(Z)-A_0}^2dZ \bigg)^{1/2},
\]
where in the last inequality we have used the fact that $|Q_j| \approx |Q_i|$
when $\frac65 Q_j$ meets $Q_i$. By construction $|W_j| \approx |Q_j|$ for $j\in I_1\cup I_2$. 
Recall from \eqref{w72} that $\dist(y, \Gamma\cap 2B) \leq 2r\le 12 Kr_0$,  
hence for $j\in I_i$ where $i\in I_3\cap I(y,r)$, $\diam(Q_j) \approx C K r_0$. Therefore $|W_j| \approx |Q_j|$ holds for all $j\in I_i$, $i\in I(y,r)$, and so we may drop $\frac{|Q_j|}{|W_j|}$. Altogether 
\begin{multline}\label{5e3}
  J \le C\sum_{i\in I(y,r)}\bigg\{\sum_{j\in J_i}\Big(\int_{W_j}\abs{A(Z)-A_0}^2dZ\Big)^{1/2}\bigg\}^2\\
    \le C\sum_{i\in I(y,r)}\sum_{j\in J_i}\int_{W_j}\abs{A(Z)-A_0}^2dZ,
\end{multline}
where the last inequality follows from Jensen's inequality. We claim that 
\begin{equation}\label{5e4}
    W_j\subset B(x,80r) \  \text{ for 
    $j\in J_i$}.
\end{equation}
To see this, let $Z \in W_j$ be given; by \eqref{w69}, $\dist(Z,Q_j) \leq 3d_j$.
But $d_j = \dist(Q_j, \Omega \cap 2B) \leq 20 \diam(Q_j)$ by definition of our 
Whitney cubes, so $\dist(Z,Q_j) \leq 60\diam(Q_j)$. Next $\frac65 Q_j$ meets $Q_i$,
so $\dist(Z,Q_i) \leq 62\diam(Q_j) \leq 244 \diam(Q_i)$. Pick $X \in Q_i \cap \overline{W}(y,r)$;
it exists because $i \in I(y,r)$; then $|X-x| \leq |X-y|+|y-x| \leq 5r$, and so
$|Z-x| \leq 244 \diam(Q_i) + 5r$. But the definition of our cubes yields
$\diam(Q_i) \leq 10^{-1} \dist(Q_i, \Omega \cap 2B)$, and 
$\dist(Q_i, \Omega \cap 2B) \leq \dist(Q_i, \Gamma \cap 2B)\le 3r$ by \eqref{w72}. 
So $244 \diam(Q_i) \leq 244 \frac{3r}{10} < 75r$, and \eqref{5e4} follows.

Return to \eqref{5e3}; by \eqref{5e4} and the fact that the $W_j$ have finite overlap, 
\[
J\le C\int_{B(x,80r)\cap\om}\abs{A(Z)-A_0}^2dZ\le C r^{d+1}\gamma_A(x,80r)^2,
\]
which completes the proof of our claim \eqref{bae}. 

Let us now use the claim to do the part of our Carleson measure estimate that
comes from \[\Xi = \big\{ (y,r) \in G \times (0,+\infty)\, ; \, \dist(y,\Gamma \cap 2B) \leq 2r \big\}.\] 
We need to know that the set $\Gamma \cap B(y,r)$ of the claim is not too small.
We use \eqref{w72} to pick $y^\ast \in \Gamma \cap 2B$ such that 
$|y-y^\ast| \leq \dist(y,\Gamma \cap 2B) \leq 2r$; then by \eqref{5b2bis} for balls centered at $y^\ast$
and since $B(y^\ast,2r) \subset B(y,4r) \subset B(y^\ast,6r)$,
\begin{equation}\label{5e2}
   c_d r^d \leq  (1-C\eps^{1/d})c_d (2r)^d  \leq \H^d(B(y,4r)\cap\Gamma) 
    \le (1+C\eps^{1/d}) c_d (6r)^d \leq 7^d c_d r^d.
\end{equation}

Let $y_0 \in G$ and $s_0 > 0$ be given, and estimate
\begin{equation*}
J(\Xi) = \int_{y\in G \cap B(y_0,s_0)} \int_{0 < r \leq s_0} \, \1_{\Xi}(y,r)
\,\,\overline\alpha(y,r)^2
\frac{dr}{r}d\H^d(y).
\end{equation*}
Denote $\min\set{s_0, 6Kr_0}$ by $\widehat{s_0}$.  We use \eqref{w82}, \eqref{bae}, \eqref{5e2},
 the Ahlfors regularity of $\Gamma$ and $G$, and a Fubini argument, to get that
\begin{multline} \label{bag}
J(\Xi)
\leq C \int_{y\in G\cap B(y_0,s_0)}\int_{0 < r \leq\widehat{s_0}} 
\Big\{r^{-d}\int_{x \in \Gamma \cap B(y, 4 r)} \gamma_A(x,80 r)^2 d\H^d(x) \Big\}
\frac{dr}{r}d\H^d(y) 
\\ \le C \int_{x \in \Gamma \cap B(y_0, 5 s_0)} \int_{0 < r \leq\widehat{s_0}}
\gamma_A(x,80 r)^2 \frac{dr}{r}d\H^d(x)\\=  C \int_{x \in \Gamma \cap B(y_0, 5 s_0)} \int_{0 < r \leq 80s_0}
\gamma_A(x,r)^2 \frac{dr}{r}d\H^d(x).
\end{multline}
In \cite{david2021carleson}, we proved that the weak DKP condition implies the
stronger estimate that $\gamma(x,r)^2 \frac{dxdr}{r}$ is a Carleson measure in $\rd\times(0,\infty)$. 
The same thing holds here, with almost the same proof, even though the geometry is a bit different because 
we work with $\Gamma$ and $\Omega$. That is, for any $x_0\in\Gamma$ and $r_0>0$, we have that 
\[
    \norm{\gamma_A(x,r)^2\frac{d\H^d(x)dr}{r}}_{\C(x_0, r_0)}\le C\norm{\alpha_A(x,r)^2\frac{d\H^d(x)dr}{r}}_{\C(x_0,r_0)}.
\]
Therefore, 
\eqref{bag} now implies that
\begin{equation} \label{w80}
J(\Xi) \leq C s_0^d \cN(A) = C s_0^d \varepsilon
\end{equation}
and concludes this part of the argument.

Return to our main Carleson estimate. 
We are left with the pairs $(y,r)$ 
that belong to the set
\begin{equation} \label{w85}
\Xi' = \big\{ (y, r) \in G \times (0,+\infty) \, ; \, \dist(y,\Gamma \cap 2B) > 2r \big\}. 
\end{equation}
Let $(y,r) \in \Xi'$ be given, and assume additionally that $\dist(y,\Gamma \cap 2B)\le 6Kr_0$. Set $d(y)=\dist(y,\Gamma \cap 2B)$, then $2r<d(y)\le 6Kr_0$.
Obviously 
\begin{equation} \label{w86}
\dist(\overline W(y,r), \Gamma \cap 2B) \geq \dist(B(y,r), \Gamma \cap 2B) > d(y)- r \geq d(y)/2.
\end{equation}

Let $Q_i$ be any of the Whitney cubes such that $\frac65 Q_i$ meets $\overline W(y,r)$, and pick 
$X \in \overline W(y,r) \cap \frac65 Q_i$. Then by \eqref{w86}
$d(y)/2 \leq \dist(X,\Gamma \cap 2B) \leq \dist(Q_i,\Gamma \cap 2B) + \frac65 \diam(Q_i) \leq 12 \diam(Q_i)$,
so $\diam(Q_i) \geq d(y)/24$. Next pick any $X \in \overline W(y,r)$ and any constant matrix $A_0$, and observe that
by \eqref{w64} 
\begin{equation} \label{w87} 
\nabla \wt A(X) = \sum_{i\in I} A_i\, \nabla \chi_i(X) = \sum_{i\in I} [A_i - A_0] \, \nabla \chi_i(X)
\end{equation}
because $\sum \nabla \chi_i = 0$. We sum only over the set $I(X)$ of indices $i$ such that $X \in \frac65 Q_i$, 
and since $\abs{\nabla\chi_i(X)}\le C \diam(Q_i)^{-1} \leq C d(y)^{-1}$ for $i \in I(X)$, we get that 
\begin{equation} \label{w88} 
|\nabla \wt A(X)| \leq C d(y)^{-1} \sum_{i\in I(X)}  |A_i - A_0|.
\end{equation}
Recall that $A_i$ is the average of $A$ on $W_i$; we want to compare $A_i$ to some average
of $A$ on a large ball centered on $\Gamma$, so we choose $y^\ast \in \Gamma \cap \overline{2B}$
such that $|y^\ast - y| \leq d(y)$ and we check that 
\begin{equation} \label{w89}
W_i \subset B(y^\ast, 12 d(y)).
\end{equation}
Indeed, for $Z \in W_i$, \eqref{w69} says that
$\dist(Z,Q_i) \leq 3 \dist(Q_i,\Omega \cap 2B) \leq 60 \diam(Q_i)$, and all points of $Q_i$
lie within $2 \diam(Q_i)$ from $\overline W(y,r)$ (because $\frac65 Q_i$ meets $\overline W(y,r)$)
hence within $2 \diam(Q_i)+ r$ from $y$. So $|Z-y^\ast| \leq |Z-y| + d(y) \leq 60 \diam(Q_i) + 3d(y)/2$.

Now $\frac65 Q_i$ meets $\overline W(y,r)$, and if $X \in \overline W(y,r) \cap \frac65 Q_i$
we both have that $\dist(X,\Gamma \cap 2B) \leq d(y) + r \leq 3d(y)/2$ by \eqref{w85},
and $\dist(X,\Gamma \cap 2B) \geq \dist(\frac65 Q_i,\Gamma \cap 2B) \geq \dist(Q_i,\Gamma \cap 2B)-\frac15\diam(Q_i) \geq 9 \diam(Q_i)$, so $60 \diam(Q_i) \leq 10 d(y)$ and \eqref{w89} follows.
Therefore, choosing $A_0 = \fint_{\Omega \cap B(y^\ast, 12 d(y))}A(Y)dY$ (which does not depend on $i \in I(X)$,
and not even on $X$),
\begin{multline*} 
|A_i - A_0| \leq |W_i|^{-1} \int_{W_i} |A(Y) - A_0|dY \leq  |W_i|^{-1} \int_{B(y^\ast, 12 d(y))} |A(Y) - A_0|dY 
\\
\leq C |W_i|^{-1} |B(y^\ast, 12 d(y))| \fint_{B(y^\ast, 12 d(y))} |A(Y) - A_0|dY. 
\end{multline*}
Observe that $|W_i|^{-1} |B(y^\ast, 12 d(y))| \leq C$. For $i\in I_1\cup I_2$, this follows from $\diam(Q_i) \geq d(y)/24$. For $i\in I_3$, this is because $d(y)\le 6Kr_0$. Hence,
\begin{equation} \label{w90}
|A_i - A_0| \leq C d(y)^{-1} \gamma_A(y^\ast, 12 d(y))
\end{equation}
by Cauchy-Schwarz. 
We notice that all the $W_i$, $i \in I(Z)$, stay at distance at least $C^{-1}\diam(Q_i)
\geq C^{-1}d(y)$ from $\Gamma$, so we can replace $\gamma(y^\ast, 12 d(y))$ with 
$\wh \alpha(y^\ast,d(y))$,
the larger variant
of $\alpha(x,r)$, defined by 
\begin{equation*}
\wh \alpha(x,r) = \Big\{ \inf_{A_0} |\wh W(x,r)|^{-1} \int_{\wh W(x,r)} |A-A_0|^2 \Big\}^{1/2},
\end{equation*}
with $\wh W(x,r) = \big\{ X \in \Omega \cap B(x,20r) \, ; \, \dist(X,\Gamma) \geq C^{-1} r \big\}$.
Altogether, 
\begin{equation} \label{w91} 
|\nabla \wt A(X)| \leq C d(y)^{-1} \sum_{i\in I(X)}  |A_i - A_0| \leq C d(y)^{-1} \wh \alpha(y^\ast,d(y)).
\end{equation}
Because of this, the oscillation of $\wt A(X)$ on $\overline W(y,r)$ is at most 
$C d(y)^{-1} r \wh \alpha(y^\ast,d(y))$, and hence 
$\overline\alpha(y,r) \leq C d(y)^{-1} r \wh \alpha(y^\ast,d(y))$. By this and \eqref{w82'}, 
\begin{multline} \label{bagi}
J(\Xi') := \int_{y\in G \cap B(y_0,s_0)}\int_{0 < r \leq s_0} \, \1_{\Xi'}(y,r) 
\,\1_{d(y)\le 6Kr_0}(y)\,\, \overline\alpha(y,r)^2 \frac{dr}{r}d\H^d(y) 
\\
\leq C  \int_{y\in G \cap B(y_0,s_0)}\int_{0 < r \leq d(y)/2} 
\wh\alpha(y^\ast,d(y)) \, d(y)^{-2} r dr\,d\H^d(y)
\\
\leq C\int_{y\in G \cap B(y_0,s_0)} 
\wh \alpha(y^\ast,d(y))^2 \, d\H^d(y)\\
\le C s_0^d\sup_{y\in G\cap B(y_0,s_0)}\wh\alpha(y^\ast,d(y))^2.
\end{multline}
Now the weak DKP condition implies that $\wh\alpha(z,r)^2 \leq C \cN(A) \leq C \varepsilon$
for all $z\in \Gamma$ and $r>0$, where the Carleson norm $\cN(A)$ is as in \eqref{1a7}.
The verification is also easy (and is done in \cite{david2021carleson}).
Thus we also get that $J(\Xi') \leq C s_0^d \varepsilon$, and since \eqref{w80} provides the 
estimate for the other piece $J(\Xi')$, we finally get that the extension $\wt A$ satisfies
the weak DKP condition on $G\times (0,\infty)$, with a norm at most $C\varepsilon$.

\section{Small and vanishing $A_\infty$ and BMO}
\label{Sbmo}

In this section we check some of the relations between $A_\infty$ weights and 
$BMO$ estimates for their logarithm. We want to do this in $(\Gamma, d\sigma)$,
where either $\Gamma \subset \R^{d+1}$ is a chord-arc surface with small constant and 
$\sigma = \H^d_{\vert \Gamma}$, or $\Gamma = \R^d$ and $\sigma$ is the Lebesgue measure on  $\Gamma$. Because of the first case, we define BMO with balls.

We first define the mean oscillation on a set $E$ such  that $0 < \sigma(E) < +\infty$, of the function $f \in L^1(E,d\sigma)$ by
\begin{equation} \label{6a1}
mo(f,E) = \sigma(E)^{-1} \int_E \bigg| f(y) - \Big(\fint_{E} f d\sigma \Big)  \bigg| d\sigma(y).
\end{equation}
We say that $f \in BMO(\Gamma)$ when $f\in L^1_{loc}(\Gamma,d\sigma)$ and
\begin{equation} \label{6a2}
\|f\|_{BMO} := \sup_{r > 0} \sup_{x \in \Gamma} \, mo(f,\Delta(x,r)) < +\infty,
\end{equation}
and that $f \in VMO(\Gamma)$ when $f \in BMO(\Gamma)$ and in addition 
\begin{equation} \label{6a3}
\lim_{r \to 0_+}  \sup_{x \in \Gamma} \, mo(f,\Delta(x,r)) = 0.
\end{equation}
It is easy to see that in the case of $\R^d$, we can replace the collection of balls $\Delta(x,r)$
with the collection of cubes $Q$ with faces parallel to the axes, and get the same space
BMO, with a slightly different but equivalent semi-norm
\begin{equation} \label{6a4}
\|f\|_{\ast} := \sup_{Q}  \, mo(f,Q);
\end{equation}
also $VMO(\Gamma)$ could be defined with cubes. Finally, we can localize: if
$\Delta = \Delta(x_0,r_0)$ is a surface ball in $\Gamma$, we say that 
$f \in BMO(\Delta)$ when $f \in L^1(\Delta, d\sigma)$ and 
\begin{equation} \label{6a5}
\|f\|_{BMO(\Delta)} := \sup_{(x,r) \in \Delta \times (0,r_0), \Delta(x,r) \subset \Delta} 
\, mo(f,\Delta(x,r)) < +\infty.
\end{equation}
This may be a little ugly near the boundary, but often we only care about the restriction
of $f$ to $\Delta(x_0,r_0/2)$ anyway. In $\R^d$, the version with cubes is a little better.
If $Q_0 \subset \R^d$ is a cube (with faces parallel to the axes, we won't repeat), 
we say that $f \in BMO_\ast(Q_0)$ when $f \in L^1(Q_0, d\sigma)$ and
\begin{equation} \label{6a6}
\|f\|_{BMO_\ast(Q_0)} := \sup_{Q \subset Q_0} \, mo(f,Q) < +\infty,
\end{equation}
where the supremum is over the cubes $Q \subset Q_0$.

With all these definitions at hand, we are ready to prove Lemma \ref{lem VMOtosmAinfty}.
We start with the simplest version of the John and Nirenberg theorem, with cubes in $\R^d$.

\begin{lem}\label{ljn}
There exist constants $c = c_d > 0$ and $C = C_d \geq 1$, that depend only on $d$, 
such that if $Q_0$ is a cube and $f \in BMO_\ast(Q_0)$, then
\begin{equation} \label{6a8}
\int_{Q_0} \exp\bigg( c \|f\|_{BMO_\ast(Q_0)}^{-1} 
\bigg| f(x) - \Big(\fint_{Q_0} f d\sigma \Big) \bigg| \bigg) d\sigma(x) \leq C \sigma(Q_0).
\end{equation}
\end{lem}

\ms
In fact, the proof only requires a control on $mo(f,Q)$ when $Q$ is a dyadic subcube of
$Q_0$ (i.e., obtained from $Q_0$ by the usual dyadic partitions).

\begin{proof}[Proof of Lemma \ref{lem VMOtosmAinfty}]
Let $\omega$ and $k = \frac{d\omega}{d\sigma}$ be as in the statement,
and let $\delta > 0$ be given. Let $\varepsilon > 0$ be given, and choose
$r_0 > 0$ so small that $mo(\log k,\Delta(x,r)) < \varepsilon$ for every surface ball
$\Delta(x,r)$ such that $0 < r \leq r_0$. By standard manipulations of covering
cubes with balls of roughly the  same sizes, we also see that 
$\| \log k \|_{BMO_\ast(Q_0)} \leq C \varepsilon$ for every cube $Q_0$ such that
${\rm diam}(Q_0) \leq r_0$. 

We choose $\gamma = d^{-1/2} r_0$, consider a surface ball $\Delta = \Delta(x,r)$
such that $r \leq \gamma$, and try to prove \eqref{3b48} for any Borel set 
$E \subset \Delta$. Call $Q_0$ the smallest cube that contains $\Delta$; we chose
$\gamma$ so that ${\rm diam}(Q_0) \leq r_0$, and so 
$\| \log(k) \|_{BMO_\ast(Q_0)} \leq C \varepsilon$. This allows us to apply 
Lemma \ref{ljn} and get 
\begin{equation} \label{6a9}
\int_{Q_0} \exp\bigg( a \varepsilon^{-1} 
\bigg| \log(k(x)) - \Big(\fint_{Q_0} \log(k) d\sigma \Big) \bigg| \bigg) d\sigma(x) \leq C,
\end{equation}
with a new constant $a = C^{-1}c > 0$ that depends only on $d$. 
Since \eqref{3b48} does not change when we multiply $\omega$ by a positive constant,
we may assume that $\fint_{Q_0} \log(k) d\sigma = 0$, which will simplify our computations slightly. We write
\begin{equation} \label{6a10}
\omega(Q) = \int_Q k(x) d\sigma(x) \ \text{ and } \omega(E) = \int_E k(x) d\sigma(x) 
\end{equation}
and decide to cut each of these integrals in three, corresponding to the three regions
\begin{multline} \label{6a11}
Q_1 = \big\{ x\in Q_0 \, ; \, \log(k(x)) \leq - A \varepsilon \big\}, \\ 
Q_2 = \big\{ x\in Q_0 \, ; \, |\log(k(x))| \leq A \varepsilon \big\}, \\
Q_3 = \big\{ x\in Q_0 \, ; \, \log(k(x)) \geq A \varepsilon \big\},
\end{multline}
where the large constant $A$ will be chosen soon. Notice that 
$a \varepsilon^{-1} |\log(k)| \geq aA$ on $Q_1 \cup Q_3$, so by Chebyshev
\begin{equation} \label{6a12}
\omega(Q_1) \leq  \sigma(Q_1) \leq  \sigma(Q_1) + \sigma(Q_3) \leq C e^{-aA} \sigma(Q_0).
\end{equation}
Similarly, we cut $Q_3$ into the regions 
$D_j = \big\{ x\in Q_3\, ; \, j + A\varepsilon \leq \log(k(x)) < j+1+ A\varepsilon\big\}$
and get that
\begin{multline} \label{6a13}
\omega(Q_3) \leq \int_{Q_3} k(x) d\sigma(x)
\leq \sum_{j \geq 0} \int_{D_j} k(x) d\sigma(x) \\
\leq \sum_{j \geq 0} e^{j+1+ A\varepsilon} \exp\big(- a \varepsilon^{-1}(j+A\varepsilon)\big)
\int_{D_j} \exp\big(a \varepsilon^{-1}\log(k(x))\big) d\sigma(x) \\
\leq C \sigma(Q_0) e^{1+ A\varepsilon} e^{-aA} \sum_{j \geq 0} e^j e^{-a \varepsilon^{-1} j} \\
\leq C \sigma(Q_0) e^{1+ A\varepsilon} e^{-aA} \leq C e^{-aA/2}
\end{multline}
if we choose $A \geq 2 a^{-1}$ and $\varepsilon \leq A^{-1}$.
Now let $F$ be a Borel subset of $Q_0$; we think about $E$ or $\Delta$. Then
\begin{multline} \label{6a14}
|\omega(F) - \sigma(F)| \leq \omega(F \cap (Q_1 \cup Q_3)) + \sigma(F \cap (Q_1 \cup Q_3))
+ \int_{F \cap Q_2} |k(x)-1| d\sigma \\
\leq \omega(Q_1 \cup Q_3)+ \sigma(Q_1 \cup Q_3) 
+ \int_{Q_2} \big|e^{A\varepsilon}-1 \big| d\sigma
\leq (C e^{-aA/2}  + 2 A\varepsilon)  \sigma(Q_0)
\end{multline}
by \eqref{6a11}-\eqref{6a13}. By definition of $Q_0$, we also have that
$\sigma(\Delta) \geq \alpha \sigma(Q_0)$ for some constant $\alpha > 0$ that depends only
on $d$.  Then we choose $A$, and then $\varepsilon$, so that
$C e^{-aA/2}  + 2 A\varepsilon \leq 10^{-2}\alpha \delta$ in \eqref{6a14}.
Then 
\begin{equation} \label{6a15}
|\omega(\Delta) - \sigma(\Delta)| \leq 10^{-2}\alpha \delta \sigma(Q_0) 
\leq 10^{-3} \delta \sigma(\Delta) 
\end{equation}
and similarly
\begin{equation} \label{6a16}
|\omega(E) - \sigma(E| \leq 10^{-2} \delta \sigma(\Delta), 
\end{equation}
from which it is easy to see that 
$\abs{\frac{\omega(E)}{\omega(\Delta)}-\frac{\sigma(E)}{\sigma(\Delta)}}<\delta$
as in \eqref{3b48}.
\end{proof}

\begin{rem}
We shall not need the fact that Lemma \ref{lem VMOtosmAinfty} is still valid when $\Gamma$ is a chord-arc surface with small enough constant, but this is the case, with essentially the same proof. The point is that we can cut $\Gamma$ into
collections of ``dyadic cubes'' that have roughly the same properties as the dyadic cubes 
in $\R^d$. Then the analogue of Lemma \ref{ljn} holds when $Q_0$ is one of these cubes
(because we only need to estimate $mo(f, Q)$ for dyadic subcubes), and then we can  proceed as above (but cover the surface ball $\Delta$ by a finite collection of ``dyadic cubes'').
This would even work when $\Gamma$ is simply Ahlfors regular 
(in fact the doubling property of $\sigma$ is the important feature).

Also, here we stated the lemma in terms of $VMO$ and vanishing $A_\infty$ condition, 
but the same proof shows that for each $\delta > 0$, we can find 
$\varepsilon = \varepsilon (\delta, d) >0$ such that if $||\log k||_{BMO} \leq \varepsilon$, 
then $\omega \in A_\infty(\sigma,  \delta)$.
\end{rem}

There is a converse to \ref{lem VMOtosmAinfty} and its small constant  variant, 
which we discuss now. 

\begin{lem}\label{lvmo}
There is a constant $\varepsilon > 0$, that depends only on $d$, and for each 
$\delta > 0$, a constant $\varepsilon_1 = \varepsilon_1(\delta, d)$ with the following
property. Let $\Gamma \in CASSC(\varepsilon)$ be a chord-arc surface 
with small enough constant in $\R^{d+1}$ (as in Definition \ref{d1b12}), 
and let $\omega$ be another measure on
$\Gamma$ such that $\omega \in A_\infty(\sigma,\varepsilon_1)$ (see Definition \ref{d1b14}). 
Then $\omega$ is absolutely continuous with respect to $\sigma$, and 
its density $k = \frac{d\omega}{d\sigma}$ satisfies $||\log(k)||_{BMO} \leq \delta$.
\end{lem}

\ms
\begin{proof}
 Let $\Gamma \in CASSC(\varepsilon)$ and $\omega \in A_\infty(\sigma,\varepsilon_1)$
 satisfy the assumption. It is a general fact about $A_\infty$ that then 
$\omega$ is absolutely continuous with respect to $\sigma$, so $k$ is well defined
and locally integrable. Let $a > 0$ be small, to be chosen later (depending on $d$).
We want to show that (if $\varepsilon$ and $\varepsilon_1$ are small enough), 
for every surface ball $\Delta$, there is a constant $C_\Delta$ such that 
\begin{equation} \label{6a19}
\sigma\big(\big\{ x\in \Delta \, ; \, |\log(k(x)) - c_\Delta| \geq \delta \big\}\big)
\leq a\sigma(\Delta).
\end{equation}
In fact, the simplest is to take $c_\Delta = \log\Big(\frac{\omega(\Delta)}{\sigma(\Delta)}\Big)$.
Write $\Delta = \Delta_1 \cup \Delta_2 \cup \Delta_3$, with
\begin{multline} \label{6a20}
\Delta_1 = \big\{ x\in \Delta \, ; \, k(x) < e^{c_\Delta -\delta} \big\}, \\ 
\Delta_2 = \big\{ x\in \Delta \, ; \, e^{c_\Delta -\delta} \leq k(x) \leq e^{c_\Delta +\delta} \big\}, \\
\Delta_3 = \big\{ x\in \Delta \, ; \, k(x) > e^{c_\Delta +\delta} \big\};
\end{multline}
we want to show that $\sigma(\Delta_1) + \sigma(\Delta_3) \leq \sigma(\Delta)/100$.
We take $E = \Delta_1$ in the definition \eqref{1b15}; we get that
$ \abs{\frac{\omega(\Delta_1)}{\omega(\Delta)}-\frac{\sigma(\Delta_1)}{\sigma(\Delta)}}<\varepsilon_1$ or, since $\omega(\Delta) = e^{c_\Delta} \sigma(\Delta)$, 
\begin{equation} \label{6a21}
| e^{- c_\Delta}\omega(\Delta_1) - \sigma(\Delta_1)| \leq \varepsilon_1 \sigma(\Delta).
\end{equation}
But $\omega(\Delta_1) = \int_{\Delta_1} k d\sigma \leq e^{c_\Delta -\delta} \sigma(\Delta)$
so $\sigma(\Delta_1) - e^{- c_\Delta}\omega(\Delta_1)
\geq (1-e^{-\delta}) \sigma(\Delta_1)$, so the comparison with \eqref{6a21} yields
$\sigma(\Delta_1) \leq \varepsilon_1 (1-e^{-\delta})^{-1} \sigma(\Delta) < a\sigma(\Delta)/2$ if $\varepsilon_1$ is chosen small enough.
Similarly, 
\begin{equation} \label{6a22}
| e^{- c_\Delta}\omega(\Delta_3) - \sigma(\Delta_3)| \leq \varepsilon_1 \sigma(\Delta)
\end{equation}
by the proof of \eqref{6a21}, and the definition yields
$\omega(\Delta_3) = \int_{\Delta_3} k d\sigma \geq e^{c_\Delta +\delta} \sigma(\Delta)$,
so that $e^{- c_\Delta}\omega(\Delta_3) - \sigma(\Delta_3) 
\geq (e^{-\delta}-1) \sigma(\Delta_3)$ and the comparison yields
$\sigma(\Delta_3) \leq \varepsilon_1 (e^{\delta}-1)^{-1} \sigma(\Delta) < a\sigma(\Delta)/2$. This proves \eqref{6a19}.

Now we claim that a well known result proved by F. John \cite{john1965quasi},
improved by J. O. Str\"omberg \cite{stromberg1979bounded}, allows one to
deduce from \eqref{6a19} that $\log(k) \in BMO$, with a norm less than $C\delta$.
Here $C$ depends on $d$, and of course the difference between $C\delta$ and $\delta$ 
could easily compensated by taking $\varepsilon_1$ smaller. 
The advantage of the result of John and Str\"omberg is that it is directly valid for a doubling measure. 
For the sake of the reader that would know the usual proof of the John-Nirenberg theorem
(with simple stopping times), let us say two words about how the underlying ideas
could be applied.
In $\R^d$, if you have \eqref{6a19} for balls, you also get it for cubes $Q$, by the usual
trick of using \eqref{6a19} for the smallest ball containing $Q$, and 
at the price of starting with the smaller $\wt a = a \sigma(Q)/\sigma(\Delta)$.
Then, if $a$ is small enough, we deduce from \eqref{6a19} that if $Q$ and $R$
are cubes such that $Q \subset R$ and $\rm{diam}(Q) \geq \rm{diam}(R)/2$,
then $|c_Q-c_R| \leq 2\delta$. Then the proof of the John-Nirenberg theorem can 
then be applied, and this yields $||\log(k)||_{BMO} \leq C\delta$. 
So the result is valid when $\Gamma \in CASSC(\varepsilon)$,
$\varepsilon$ small enough (so that $\sigma$ is doubling), with the usual fake dyadic cubes 
on $\Gamma$. We need to take $a$ even smaller, so that in particular we can control 
$|c_Q-c_R|$ for fake cubes, when $Q$ is a child of $R$, but the proof goes through.
\end{proof}

\begin{cor}\label{cvmo}
In Theorem \ref{t1b1}, we can replace the conclusion that 
$\omega^\infty \in A_\infty(\sigma,\delta)$ with the conclusion that
$\omega^\infty$ is absolutely continuous with respect to $\sigma$, with a
density $k = \frac{d\omega}{d\sigma}$ such that $||\log(k)||_{BMO} \leq \delta$.

In Theorem \ref{t1b2}, we can replace the conclusion by the fact that 
for each surface ball $\Delta = B(x,r) \cap \Gamma$ such that 
$\Delta \subset B(X,\kappa \dist(X,\Gamma))$ and $0 < r \leq \tau \dist(X,\Gamma)$,
$\omega^X$ is absolutely continuous with respect to $\sigma$ on $\Delta$,
and the density $k = \frac{d\omega^X}{d\sigma}$ is such that
$||\log(k)||_{BMO(\Delta)} \leq \delta$.
\end{cor}

\begin{proof}
The first part follows at once from Theorem \ref{t1b1} and Lemma \ref{lvmo}.
The second part follows Theorem \ref{t1b2} (applied with a larger $\kappa$ for security) 
and the proof of Lemma \ref{lvmo}.
\end{proof}

\section{Proof of Lemma \ref{lem em_limit} on the elliptic measure for a limit}
\label{sec emlimit_pf}

In this section, we prove Lemma \ref{lem em_limit}.

We first show that \eqref{eq em_limit} holds when $E=U$ is a bounded open set in $\rd$. 

Let $\set{F_i}_{i\ge1}$ be a sequence of closed sets with $F_i\subset U$ and $F_i\uparrow U$ as $i\to \infty$. 
For each $i\ge1$, let $g_i\in C^\infty_c(U)$, with $0\le g_i\le 1$ and $g_i\equiv 1$ on $F_i$. 
For $X\in\reu$, define
\[
v_i^{(k)}(X):=\int_{\rd}g_i(y)d\omega_{L_k}^X(y),
\]
and 
\[
v_i(X):=\int_{\rd}g_i(y)d\omega_{L}^X(y).
\]
By definition of the elliptic measures, $v_i^{(k)}\in W_{\loc}^{1,2}(\reu)$ is the unique weak solution to 
\begin{equation}\label{eq vki_sol}
\begin{cases}
L_k v_i^{(k)}=0 \qquad \text{in }\reu,\\
v_i^{(k)}(x,0)=g_i(x) \qquad\text{for }x\in\rd
\end{cases}\end{equation}
that satisfies $v_i^{(k)}\in  C(\overline{\reu})$ and $\lim_{\abs{X}\to\infty,\, X\in\reu}v_i^{(k)}(X)=0$. Similarly, $v_i\in W_{\loc}^{1,2}(\reu)$ is the unique weak solution to 
\begin{equation}\label{eq vi_sol}
\begin{cases}
L v_i=0 \qquad \text{in }\reu,\\
v_i(x,0)=g_i(x) \qquad\text{for }x\in\rd
\end{cases}\end{equation}
that satisfies $v_i\in C(\overline{\reu})$ and $\lim_{\abs{X}\to\infty,\, X\in\reu}v_i(X)=0$.
Here, and in the sequel, 
$\overline{\reu}=\rd\times[0,\infty)$.
Let us check that 
for any fixed $i$,
\begin{equation}\label{unifcvg_infty}
    \lim_{\abs{X}\to\infty,\, X\in\reu}v_i^{(k)}(X)=0 \quad\text{uniformly in }k.
\end{equation}
In fact, since $\supp g_i\subset U$ and $0\le g_i\le 1$,
\[
0\le v_i^{(k)}(X)=\int_Ug_i(y)d\omega_{L_k}^X(y)\le \omega_{L_k}^X(\Delta_r),
\]
where $\Delta_r\subset\rd$ is a ball with radius $r$ that contains $U$. By Lemma \ref{lem cfms} and the estimate \eqref{eq2.green} for the Green function, 
\[
\abs{v_i^{(k)}(X)}\le C\, r^{d-1}\abs{X-X_{\Delta_r}}^{1-d},
\]
where $C$ depends only on $\mu_0$ and $d$. 
So \eqref{unifcvg_infty} holds.

By our choice of $F_i$ and $g_i$, 
\begin{equation} \label{9c4}
 \abs{v_i^{(k)}(X)-\omega_{L_k}^X(U)}=\abs{\int_{\rd}(g_i(y)-\1_U(y))d\omega_{L_k}^X(y)}\le \omega_{L_k}^X(U\setminus F_i).
\end{equation}
Since $\cN(A_k)<C_0$ for all $k\in\ZZ_+$, Lemma \ref{lem DKPtoAinfty} asserts that 
$\omega_{L_k}\in A_\infty(d\sigma)$, and the $A_\infty$ constants depend only on $d$ and $C_0$. 
Therefore
we can find constants $\theta$, $C$, that may depend on $X$ and $U$, but not on $k$
or the $F_i$, such that 
\[
\omega_{L_k}^X(U\setminus F_i)\le C \sigma(U\setminus F_i)^{\theta}.
\]
The regularity of the
Lebesgue measure now implies that 
the right-hand side of \eqref{9c4} 
tends to 0 as $i\to\infty$. So for any $X\in\reu$, 
the limit $\lim_{i\to\infty}v_i^{(k)}(X)$ exists, and 
\begin{equation}\label{eq cvg1}
 \lim_{i\to\infty}v_i^{(k)}(X) =\omega_{L_k}^X(U) \qquad\text{uniformly in }k.
\end{equation}
Similarly, 
\[    \abs{v_i(X)-\omega_{L}^X(U)}=\abs{\int_{\rd}(g_i(y)-\1_U(y))d\omega_{L}^X(y)}\le \omega_{L}^X(U\setminus F_i)\to 0
\]
as $i\to\infty$, by the regularity of $\omega_L^X$. So 
\begin{equation}\label{eq cvg2}
 \lim_{i\to\infty}v_i(X) =\omega_{L}^X(U).
\end{equation}

We claim that for any fixed $i$ and 
$X\in\reu$,
\begin{equation}\label{eq cvg3}
    \lim_{k\to\infty}v_i^{(k)}(X)=v_i(X). 
\end{equation}
We shall justify the claim by first showing that as $k\to\infty$, $v_i^{(k)}$ converges to a weak solution $v_i^\infty$ of $Lv_i^\infty=0$ in $\reu$ with boundary data $g_i$, $v_i^\infty\in C(\overline{\reu})$, 
and $v_i^\infty$ converges to 0 uniformly at infinity in $\reu$. 
Then by the maximum principle (or by the definition of elliptic measures), 
$v_i^\infty$ must be the unique weak solution to the Dirichlet problem \eqref{eq vi_sol}. 
Thus, $v_i^\infty=v_i$, and the claim \eqref{eq cvg3} will follow.

Let $i\in\ZZ_+$ be fixed 
for the moment. Since the 
$v_i^{(k)}$ are solutions with the same (smooth) boundary data $g_i$ and the $L_k$ 
have the same ellipticity constant for all $k$, the maximum principle, the H\"older regularity of solutions, 
and the Arzela-Ascoli theorem imply that 
there is a subsequence of $\big\{v_i^{(k)}\big\}_{k\ge1}$,
which we still denote the same way, that converges uniformly on compact subsets  of $\overline\reu$.
That is, there is a continuous function $v_i^\infty$ on $\overline\reu$ such that 
\begin{equation}\label{viinfty}
    \lim_{k\to\infty}v^{(k)}_i(X)=v_i^{\infty}(X), \ \text{ uniformly on compact subsets of } \overline{\reu}
\end{equation}
Since $v_i^{(k)}$ is continuous on $\overline\reu$ and equal to $g_i$ on $\rd$, we get that 
\begin{equation}\label{viinfty bdy}
    v_i^\infty\in C(\overline{\reu}), \quad\text{and } v_i^\infty(x,0)=g_i(x) \ \text{ for } x\in\rd.
   \end{equation}
  By \eqref{viinfty} and \eqref{unifcvg_infty}, 
\begin{equation}\label{viinfty unif}
    \lim_{\abs{X}\to\infty,\,X\in\reu}v_i^\infty(X)=0.
\end{equation}

Next let us show that 
(still for $i$ fixed) 
\begin{equation}\label{vki_W12conv}
    \nabla v_i^{(k)}\to \nabla v_i^\infty \text{ in } L^2_{\loc}(\reu) \text{ as }k\to\infty.
\end{equation}
Let $K$ be a compact subset of $\reu$, and pick a smooth cut-off function
$\eta$, with $0 \leq \eta \leq 1$ everywhere, $\eta = 1$ on $K$, 
and compact support in $\reu$. Let $k, \ell \geq 0$ be given; we want to estimate
\begin{equation} \label{9c12}
I = \mu_0 \int_{\reu} \big|\nabla(v_i^{(k)}-v_i^{(\ell)})\big|^2 \eta^2.
\end{equation}
The ellipticity of $L_k$ gives that
\begin{multline}\label{eq gradvk-vk'} 
I \leq \int_{\reu} A_k\nabla(v_i^{(k)}-v_i^{(\ell)})\cdot\nabla(v_i^{(k)}-v_i^{(\ell)}) \eta^2 \\
\hskip-2cm = \int_{\reu} A_k\nabla(v_i^{(k)}-v_i^{(\ell)})\cdot\nabla\Big((v_i^{(k)}-v_i^{(\ell)}) \eta^2\Big) \\
-  2 \int_{\reu} A_k\nabla(v_i^{(k)}-v_i^{(\ell)})\cdot \nabla \eta \Big( \eta (v_i^{(k)}-v_i^{(\ell)})\Big) 
=: J_1 -2 J_2.
\end{multline}
Set $\varphi = (v_i^{(k)}-v_i^{(\ell)}) \eta^2$, and use it as a test function the equations 
$L_kv_i^{(k)}=0$ and $L_{\ell}v_i^{(\ell)}=0$; we get that 
\begin{multline*}
 J_1 = \int_{\reu} A_k\nabla(v_i^{(k)}-v_i^{(\ell)})\cdot\nabla \varphi
 = - \int_{\reu} A_k\nabla v_i^{(\ell)}\cdot\nabla \varphi
 = \int_{\reu} (A_{\ell}-A_k)\nabla v_i^{(\ell)}\cdot\nabla \varphi \\
 = \int_{\reu} (A_{\ell}-A_k)\nabla v_i^{(\ell)}\cdot\nabla(v_i^{(k)}-v_i^{(\ell)}) \eta^2
 + 2 \int_{\reu} (A_{\ell}-A_k)\nabla v_i^{(\ell)}\cdot\nabla \eta \Big( \eta (v_i^{(k)}-v_i^{(\ell)})\Big)\\
 =: J_3 + 2 J_4.
\end{multline*}
Then by \eqref{eq gradvk-vk'}, $I \leq J_1-2J_2 \leq J_3 - 2J_2 + 2 J_4$.
Now we use Cauchy Schwarz and Young's inequality for each term. That is,
\begin{multline} \label{9a14}
J_3 \leq 
\Big\{\int_{\reu} \big|\nabla(v_i^{(k)}-v_i^{(\ell)})\big|^2 \eta^2 \Big\}^{1/2}
\Big\{\int_{\reu} \Big|(A_k-A_{\ell})\nabla v_i^{(\ell)}\Big|^2 \eta^2 \Big\}^{1/2}
 \\
\leq \frac{1}{4} I + C \int_{\reu} \Big|(A_k-A_{\ell})\nabla v_i^{(\ell)}\Big|^2 \eta^2,
\end{multline}
where $C$ depends also on $\mu_0$, then 
\begin{multline} \label{9a15}
|2J_2| \leq 
C \Big\{\int_{\reu} \big|\nabla(v_i^{(k)}-v_i^{(\ell)})\big|^2 \eta^2 \Big\}^{1/2}
\Big\{\int_{\reu} \Big|v_i^{(k)}-v_i^{(\ell)}\Big|^2 |\nabla \eta|^2 \Big\}^{1/2}
 \\
\leq \frac{1}{4} I + C \int_{\reu} \Big|v_i^{(k)}-v_i^{(\ell)}\Big|^2 |\nabla \eta|^2,
\end{multline}
\begin{multline} \label{9a16}
|2J_4| \leq 
C \Big\{\int_{\reu} |A_{\ell}-A_k|^2 |\nabla v_i^{(\ell)}|^2 \eta^2 \Big\}^{1/2}
\Big\{\int_{\reu} \Big|v_i^{(k)}-v_i^{(\ell)}\Big|^2 |\nabla \eta|^2 \Big\}^{1/2}
 \\
\leq C\int_{\reu} |A_{\ell}-A_k|^2 |\nabla v_i^{(\ell)}|^2 \eta^2 
+ C \int_{\reu} \Big|v_i^{(k)}-v_i^{(\ell)}\Big|^2 |\nabla \eta|^2,
\end{multline}
Altogether, 
\begin{equation} \label{9a17}
I \leq C  \int_{\reu} |A_{\ell}-A_k|^2 |\nabla v_i^{(\ell)}|^2 \eta^2 
+ C \int_{\reu} \Big|v_i^{(k)}-v_i^{(\ell)}\Big|^2 |\nabla \eta|^2 = : CI_1 + C I_2, 
\end{equation}
where $C$ depends only on $d$ and $\mu_0$ but not on $k$, $\ell$, or $i$.

By \eqref{viinfty}, the sequence $\big\{v_i^{(k)}\big\}_k$ converges in $L^2_{\loc}(\reu)$,
so $I_2$ tends to $0$ when $k$ and $\ell$ tend to $+\infty$.
For $I_1$, we use H\"older's 
inequality and then the reverse H\"older inequality 
for solutions (see e.g. \cite{kenig1994harmonic} Lemma 1.1.12) to get that
\begin{multline*}
    \abs{I_1}\le \br{\iint\abs{A_k-A_{\ell}}^{2p'}\eta^{p'}}^{\frac{1}{p'}}\br{\iint\abs{\nabla v_i^{(\ell)}}^{2p}\eta^p}^{\frac{1}{p}}\\
    \le C\br{\iint_{\supp\eta}\abs{A_k-A_{\ell}}^{2p'}}^{\frac{1}{p'}}\br{\iint_{K'}\abs{\nabla v_i^{(\ell)}}^{2}}^{\frac{1}{2}},
\end{multline*}
where $p>1$ is sufficiently close to $1$ so that the reverse H\"older inequality for solutions holds, 
$p'$ is the H\"older exponent of $p$, and $K'$ is a slightly larger compact set in $\reu$ than $\supp \eta$. 
Here, the constant $C$ depends on $p$, $\supp\eta$, and $K'$, but is independent of $k$ and $\ell$.  Notice that by 
Caccioppoli's inequality, the 
regularity of solutions, and the maximum principle, 
\[
\iint_{K'}\abs{\nabla v_i^{(\ell)}}^2\le C\iint_{V}\abs{v_i^{(\ell)}}^2
\le C'\norm{v_i^{(\ell)}}^2_{L^\infty(\reu)}\le C'\norm{g_i}_{L^\infty(\rd)}^2\le C',
\]
where  $V\subset\subset\reu$ is an enlargement of the set $K'$, and $C$, $C'$ are constants independent of $\ell$. Therefore, 
\[
\abs{I_1}\le C'\br{\iint_{\supp\eta}\abs{A_k-A_{\ell}}^{2p'}}^{\frac{1}{p'}}.
\]
Since $A_k(x,t)=A(x,t)$ for $t\in(2^{-k},2^k)$, 
\begin{equation}\label{eq Ak_limitae}
    \lim_{k\to\infty}A_k(x,t)= A(x,t) \qquad\text{for } (x,t)\in\reu \text{ a.e.}  
\end{equation}
So $I_1$ tends to $0$ when $k$ and $\ell$ tend to $+\infty$, 
by the dominated convergence theorem, 
hence $I$ tends to $0$ for each given $K$ and $\eta$, and
\eqref{vki_W12conv} follows. Because of \eqref{vki_W12conv}, one can 
let $k$ tend to $+\infty$ 
in the equation
\[
0=\iint A_k\nabla v_i^{(k)}(Y)\nabla \vp(Y)dY 
\]
where $\vp\in C_0^\infty(\reu)$ is any test function, 
and get that $v_i^\infty$ is a weak solution of  $L v_i^\infty=0$ in $\reu$.
It satisfies \eqref{viinfty bdy} and \eqref{viinfty unif}, and we checked earlier 
that the claim \eqref{eq cvg3} follows.

We claim that 
\begin{equation}\label{lmt_open}
     \lim_{k\to\infty}\omega^{X}_{L_k}(U)=\omega_L^{X}(U)  
\end{equation}
for every bounded open subset $U$ of $\Gamma$ and every $X\in\reu$.
This will follow from \eqref{eq cvg1}, \eqref{eq cvg2} and \eqref{eq cvg3}; to see this, let $U$ and $X$ be fixed, and write
\begin{multline*}
    \abs{\omega_{L_k}^X(U)-\omega_L^X(U)}\le \abs{\omega_{L_k}^X(U)-v_i^{(k)}(X)}+\abs{v_i(X)-\omega_L^X(U)}+\abs{v_i^{(k)}(X)-v_i(X)}.
\end{multline*}
Let $\epsilon>0$ be arbitrary. By \eqref{eq cvg1} and \eqref{eq cvg2}, one can choose $i_0$ sufficiently large so that  \[
\abs{\omega_{L_k}^X(U)-v_{i_0}^{(k)}(X)}<\epsilon/3 \quad\text{for all }k\in\ZZ_+ 
\]
and $\abs{v_{i_0}(X)-\omega_L^X(U)}<\epsilon/3$. For this fixed $i_0$, \eqref{eq cvg3} allows us to take $k$ sufficiently large so that 
\[
\abs{v_{i_0}^{(k)}(X)-v_{i_0}(X)}<\epsilon/3.
\]
Altogether, we obtain that for $k$ sufficiently large, 
$\abs{\omega_{L_k}^X(U)-\omega_L^X(U)}<\epsilon$. 
This proves \eqref{lmt_open} (for $U$ open and bounded), but apparently only for some subsequence
that we constructed. However, if \eqref{lmt_open} were to fail (for $X$ and the open set $U$),
we could first replace $\{ L_k \}$ with a subsequence for which $\omega_{L_{k}}^X(U)$ stays 
far from $\omega_L^X(U)$, and then proceed as above with $\{ L_k \}$, 
find a subsequence with \eqref{lmt_open}, a contradiction. So \eqref{lmt_open}  holds for the whole sequence.

It remains to show that \eqref{eq em_limit} (or \eqref{eq em_limit})
also holds for any Borel (and not just open) bounded set.
So let the bounded Borel set $E$ and the pole $X$ be fixed. 
Since $\omega_L^X$ is regular, we can find 
a nonincreasing sequence $\{ U_i \}$ of open sets that contain $E$ and a nondecreasing sequence 
$\{ F_i \}$ of closed sets contained in $E$, such that 
\begin{equation} \label{9a20}
\lim_{i\to\infty}\omega_L^X(U_i)= \lim_{i\to\infty}\omega_L^X(E_i) = \omega_L^X(E).
\end{equation}
For each $i$, notice that $U_i$ and $U_i \sm F_i$ are bounded and open, 
so \eqref{lmt_open} holds for them and yields
\begin{equation} \label{9a21}
\lim_{k\to\infty}\omega^{X}_{L_k}(U_i)=\omega_L^{X}(U_i)
\ \text{ and } \ 
\lim_{k\to\infty}\omega^{X}_{L_k}(U_i \sm F_i)=\omega_L^{X}(U_i \sm F_i).
\end{equation}
Now let $\epsilon > 0$ be given; by \eqref{9a20} we can find $i$ such that 
$\omega_L^X(U_{i}\setminus F_{i})<\frac{\epsilon}{4}$.
Then \eqref{9a21} says that for $k$ large, 
\begin{equation} \label{9a22}
| \omega^{X}_{L_k}(U_i) - \omega_L^{X}(U_i)| 
+ | \omega^{X}_{L_k}(U_i \sm F_i) - \omega_L^{X}(U_i \sm F_i)|
< \frac{\epsilon}{4}
\end{equation}
and since $F_i \subset E \subset U_i$, 
\begin{multline} \label{9a23}
| \omega^{X}_{L_k}(E) - \omega_L^{X}(E)|
\leq | \omega^{X}_{L_k}(E) - \omega_{L_k}^{X}(U_i)|
+ |\omega_{L_k}^{X}(U_i) - \omega_L^{X}(U_i)|
+ |\omega_L^{X}(U_i)-\omega_L^{X}(E)|
\\
= \omega^{X}_{L_k}(U_i \sm E) +  |\omega_{L_k}^{X}(U_i) - \omega_L^{X}(U_i)| 
+ \omega^{X}_{L}(U_i \sm E)
\\
\leq \omega^{X}_{L_k}(U_i \sm F_i) +  |\omega_{L_k}^{X}(U_i) - \omega_L^{X}(U_i)| 
+ \omega^{X}_{L}(U_i \sm F_i) < \epsilon
\end{multline}
by \eqref{9a21} and \eqref{9a22}. So $\lim_{k\to\infty}\omega^{X}_{L_k}(E) = \omega^{X}_{L}(E)$,
and completes the proof of Lemma \ref{lem em_limit}.

\bibliographystyle{alpha}
\bibliography{reference}

\Addresses

\end{document}